\label{key}\documentclass[11pt]{amsart}
\usepackage{amsmath, amssymb, amscd, mathrsfs, slashed, enumerate, url}
\usepackage{color}
\usepackage{extsizes}
\usepackage{xcolor}
\usepackage[all,cmtip]{xy}
\usepackage{multicol}
\usepackage{indentfirst}
\usepackage{latexsym}
\usepackage{bm}
\usepackage{graphicx}
\usepackage{galois}
\usepackage{subfigure,esint}
\usepackage{float}
\usepackage{verbatim}
\usepackage[top=1in, bottom=1in, left=1.25in, right=1.25in]{geometry}
\usepackage{epsfig,dsfont,amsthm,amsfonts,amsbsy}
\usepackage[colorlinks]{hyperref}
\usepackage[toc,page]{appendix}
\usepackage{tikz-cd}
\usepackage{amsthm,thmtools,xcolor}
\newtheorem{theorem}{Theorem}[section]

\newtheorem{corollary}[theorem]{Corollary}

\newtheorem{definition}[theorem]{Definition}
\newtheorem{question}[theorem]{Question}

\newtheorem{lemma}[theorem]{Lemma}

\newtheorem{proposition}[theorem]{Proposition}
\newtheorem{remark}[theorem]{Remark}

\newcommand{\Hit}{\mathrm{Hit}}

\newcommand{\gsf}{g_{\mathrm{sf}}}

\newcommand{\LC}{\mathrm{Lim}}
\newcommand{\msHi}{\msH_{\infty}}

\newcommand{\tlh}{\tilde{h}}
\newcommand{\tH}{\widetilde{H}}
\newcommand{\UpMoc}{\Upsilon^{\mathrm{Moc}}}
\newcommand{\MMHLC}{\MM_{\mathrm{Hit},\mathrm{Lim}}}

\newcommand{\can}{\mathrm{can}}
\newcommand{\Tot}{\mathrm{Tot}}
\newcommand{\MGC}{\mathcal{G}^{\mathbb{C}}}

\newcommand{\dbarE}{\bar{\partial}_E}

\newcommand{\ivt}{\iota_t^{\mathrm{v}}}
\newcommand{\sfm}{\mathrm{sf}}

\newcommand{\nS}{\widetilde{S}_q}
\newcommand{\nSt}{\widetilde{S}_{t^2 q}}

\usepackage{enumitem}

\newcommand{\da}{\dagger}
\newcommand{\Prym}{\mathrm{Prym}}

\newcommand{\MH}{\mathcal{H}}

\newcommand{\MM}{\mathcal{M}}
\newcommand{\MO}{\mathcal{O}}

\newcommand{\Ker}{\mathrm{Ker}}

\newcommand{\End}{\mathrm{End}}

\newcommand{\MC}{\mathcal{C}}
\newcommand{\Tr}{\mathrm{Tr}}

\newcommand{\st}{*}

\newcommand{\pa}{\partial}

\newcommand{\ti}{\times}

\newcommand{\bz}{\bar{z}}
\newcommand{\vp}{\varphi}

\newcommand{\ME}{\mathcal{E}}
\newcommand{\al}{\alpha}

\newcommand{\na}{\nabla}
\newcommand{\ep}{\epsilon}

\newcommand{\calA}{\mathcal A}
\newcommand{\calB}{\mathcal{B}}

\newcommand{\calE}{\mathcal E}
\newcommand{\calF}{\mathcal F}
\newcommand{\calH}{\mathcal H}
\newcommand{\calK}{\mathcal K}
\newcommand{\calL}{\mathcal L}
\newcommand{\calM}{\mathcal{M}}

\newcommand{\calO}{\mathcal{O}}
\newcommand{\calP}{\mathcal{P}}
\newcommand{\calS}{\mathcal{S}}
\newcommand{\calT}{\mathcal{T}}

\newcommand{\calV}{\mathcal V}
\newcommand{\calW}{\mathcal W}

\newcommand{\calZ}{\mathcal Z}

\newcommand{\MT}{\mathcal{T}}

\newcommand{\del}{\partial}
\newcommand{\RR}{\mathbb R}
\newcommand{\CC}{\mathbb C}

\newcommand{\wt}{\widetilde}

\newcommand{\lam}{\lambda}

\newcommand{\MP}{\mathcal{P}}

\newcommand{\diag}{\mathrm{diag}}

\newcommand{\slf}{\mathfrak{sl}}

\newcommand{\su}{\mathfrak{su}}

\newcommand{\MF}{\mathcal{F}}

\newcommand{\Lam}{\Lambda}

\newcommand{\MMH}{\MM_{\mathrm{Hit}}}

\newcommand{\MB}{\mathcal{B}}

\newcommand{\rank}{\mathrm{rank}}
\newcommand{\Id}{\mathrm{Id}}

\newcommand{\tS}{\widetilde{S}}

\newcommand{\Aut}{\mathrm{Aut}}

\newcommand{\id}{\mathrm{id}}
\newcommand{\MMD}{\MM_{\mathrm{Dol}}}

\newcommand{\Lim}{\mathrm{Lim}}
\newcommand{\SLC}{\mathrm{SL}_2(\mathbb{C})}

\newcommand{\SL}{\mathrm{SL}}

\newcommand{\tg}{\tilde{g}}

\newcommand{\ML}{\mathcal{L}}

\newcommand{\MD}{\mathcal{D}}

\newcommand{\tx}{\tilde{x}}

\newcommand{\LS}{\lesssim}

\newcommand{\ad}{\mathrm{ad}}

\newcommand{\app}{\mathrm{app}}
\newcommand{\model}{\mathrm{model}}

\newcommand{\odd}{\mathrm{odd}}
\newcommand{\even}{\mathrm{even}}
\newcommand{\Stab}{\mathrm{Stab}}

\newcommand{\mbR}{\mathbb{R}}

\newcommand{\tp}{\tilde{p}}

\newcommand{\tZ}{\widetilde{Z}}

\newcommand{\tpi}{\tilde{\pi}}
\newcommand{\tlam}{\tlam}

\newcommand{\mbZ}{\mathbb{Z}}

\newcommand{\mbC}{\mathbb{C}}

\newcommand{\msH}{\mathscr{H}}

\newcommand{\msV}{\mathscr{V}}

\newcommand{\BPic}{\overline{\mathrm{Pic}}}

\newcommand{\Div}{\mathrm{Div}}
\newcommand{\SU}{\mathrm{SU}}

\newcommand{\Pic}{\mathsf{Pic}}
\newcommand{\bF}{\mathbf{F}}
\newcommand{\supp}{\mathrm{supp}}

\newcommand{\tom}{\tilde{\omega}}
\newcommand{\tME}{\widetilde{\ME}}
\newcommand{\tvp}{\tilde{\vp}}

\newcommand{\bp}{\mathbf{p}}

\newcommand{\coker}{\mathsf{coker}}
\newcommand{\Asy}{\mathsf{Asym}}
\newcommand{\Sy}{\mathsf{Sym}}
\newcommand{\rest}{\: \rule[-3.5pt]{0.5pt}{11.5pt}\,{}}

\newcommand{\Def}{\mathsf{Def}}

\renewcommand{\max}{\mathrm{max}}

\usepackage{calligra}

\DeclareMathAlphabet{\mathcalligra}{T1}{calligra}{m}{n}

\declaretheoremstyle[
headfont=\color{blue}\normalfont\bfseries,
bodyfont=\color{blue}\normalfont\itshape,
]{colored}

\usepackage[utf8]{inputenc}
\usepackage[english]{babel}
\usepackage{fancyhdr}
\usepackage{accents}

\begin{document}
	\title[The asymptotic geometry on the singular locus]{The asymptotics of the $\SLC$-Hitchin metric on the singular locus: Subintegrable systems}
	\author[He]{Siqi He}
	\email{sqhe@amss.ac.cn}
	\address{Morningside Center of Mathematics,
		Chinese Academy of Sciences,
		Beijing, 100190 China}

	\author{Johannes Horn}
	\email{horn@math.uni-frankfurt.de}
	\address{Johann Wolfgang Goethe-Universität,
          60325 Frankfurt am Main, Germany }

	\author{Nianzi Li}
	\email{lnz@mail.tsinghua.edu.cn}
	\address{Yau Mathematical Sciences Center, Tsinghua University, Beijing 100084, China}
         \begin{abstract}
         We study the asymptotic hyperkähler geometry of the $\SLC$-Hitchin moduli space over the singular fibers of the Hitchin fibration. We extend the previously known exponential convergence results for solutions to the Hitchin equation to the class of locally fiducial Higgs bundles defined by a special local description at the singularities of the spectral curve.
         This condition is satisfied by the Higgs bundles contained in certain subintegrable systems introduced by Hitchin. We prove that the restriction of the hyperkähler metric to the subintegrable system converges exponentially fast to the corresponding semi-flat metric along a ray $(\mathcal{E},t\varphi)$. This answers a question posed by Hitchin in \cite{Hitchin2021subintegrable_special_Kaehler}. More generally, we prove that for each stratum of quadratic differentials there is a closed subset of the corresponding Hitchin fibers, such that the restricted hyperkähler metric converges to a generalized semi-flat metric.
\end{abstract}
\maketitle
\section{Introduction}
Moduli spaces of Higgs bundles over a Riemann surface $\Sigma$ are a central object of study at the intersection of geometric analysis, algebraic geometry, and mathematical physics. Of particular interest is the $L^2$-hyperkähler metric, arising from solutions to the Hitchin equations. Motivated by the work of Gaiotto, Moore and Neitzke \cite{GMNwallcrossing}, understanding the behavior of this metric, especially in the large Higgs fields limit, has become a central problem where all the above viewpoints interact. 
\par The asymptotics of the hyperkähler metric are largely governed by the symplectic geometry of the moduli space. The Hitchin fibration endows a dense subset of the moduli space, the so-called regular locus, with the structure of an algebraically completely integrable system. In particular, by the work of Freed \cite{freed1999special}, the regular locus carries another hyperkähler metric, the so-called semi-flat metric. It was conjectured by Gaiotto, Moore and Neitzke that on the regular locus, the $L^2$-hyperkähler metric is approximated by the semi-flat metric, and that the correction terms decay at the ends of the moduli space. On the regular locus, this conjecture was settled for rank $n$ Higgs bundles on closed Riemann surfaces in \cite{dumasneitzke2019asymptotic,mazzeo2019asymptotic,fredrickson2020exponential,mochizuki2023asymptotic,mochizuki2023asymptotic2,mochizuki2024comparison,mochizuki2025semi} and for certain moduli spaces of Higgs bundles on punctured Riemann surfaces in \cite{fredrickson2022asymptotic,chen2022asymptotic}. \\

In this paper, we consider the asymptotics of the $\SLC$-hyperkähler metric on the singular locus of the Hitchin map, where the torus fibers degenerate. The $\SLC$-Hitchin base is the space of quadratic differentials, and the regular locus projects to quadratic differentials with simple zeros only.
\par
In a first step towards understanding the asymptotics of the hyperkähler metric, one proves the convergence of solutions to the Hitchin equations to so-called limiting configurations. Here limiting configurations are solutions to the decoupled Hitchin equations with singularities at the zero set $Z\subset \Sigma$ of the quadratic differential.
\par
  Let $(\ME, \vp)$ be a stable $\SLC$-Higgs bundle with $\det(\vp) = q$ the associated quadratic differential. Rescaling the Higgs field by a real parameter $t > 0$, one obtains a ray $(\ME, t\vp)$ in the moduli space. By the Kobayashi--Hitchin correspondence, there exists a unique Hermitian metric $H_t$ solving the Hitchin equations, giving rise to a family of flat connections $\MD_t = \nabla_t + \phi_t$, where $\phi_t=t(\vp+\vp^\dagger_{H_t})$.
  For Higgs bundles in the regular Hitchin fibers, foundational works by Mazzeo–Swoboda–Weiss–Witt \cite{mazzeo2012limiting}, Mochizuki \cite{Mochizukiasymptotic,mochizuki2023asymptotic2,mochizuki2024comparison}, and Fredrickson \cite{fredrickson2018generic} established that the corresponding solutions to the Hitchin equations converge to limiting configurations $(\nabla_\infty, \phi_\infty)$. More precisely, on any relatively compact subset $K \subset \Sigma \setminus Z$,
\begin{equation} \label{eq_mochizuki_convergence}
	\lim_{t \to \infty} \| (\nabla_t, \phi_t) - (\nabla_\infty, \phi_\infty) \|_{\MC^l(K)} = 0.
      \end{equation}

\subsection{The convergence rate problem for Higgs bundles with singular spectral curves}
While on the regular fibers the exponential convergence of \eqref{eq_mochizuki_convergence} is proven in \cite{fredrickson2018generic,mochizuki2023asymptotic2}, the rate of convergence for Higgs bundles with singular spectral curves remains largely unexplored. This question was explicitly raised by Mochizuki--Szabó in \cite{mochizuki2023asymptotic}:

\begin{question}{\cite[Page 2, Q2]{mochizuki2023asymptotic}}
	Let $K \subset \Sigma \setminus Z$ be a relatively compact subset. Is the convergence rate in \eqref{eq_mochizuki_convergence} on $K$ dominated by $e^{-C_K t}$ for some $C_K > 0$?
\end{question}

It turns out that the answer reflects the stratification of the singular Hitchin fibers studied in \cite{horn2022semi,horn2022sltypesingularfibers}. Mochizuki--Szabó \cite[Theorem 1.7]{mochizuki2023asymptotic} confirmed exponential convergence in the special case where the Higgs bundle arises as the pushforward of a line bundle on the normalization of the spectral curve. This condition is satisfied on the closed stratum of the stratification of singular Hitchin fibers. To obtain Higgs bundles in other strata, one needs to apply Hecke modifications to the pushforward of a line bundle on the normalized spectral curve. Each stratum can be parametrized by abelian coordinates for line bundles on the normalized spectral curve together with coordinates fixing the Hecke modifications (Hecke parameters). It turns out that there are special Hecke parameters for each stratum such that the limiting configuration is described by a so-called fiducial solution at the higher-order zeros of the quadratic differential. We call the corresponding Higgs bundle \emph{locally fiducial}. For example, the Hitchin section at a singular Hitchin fiber is locally fiducial, but not the pushforward of a line bundle on the normalized spectral curve.

Crucially, we observe that good approximate solutions—analogous to those on the regular locus—can only be constructed for locally fiducial Higgs bundles. For Higgs bundles that are not locally fiducial, the Higgs fields acquire additional degeneracies near the nodes, obstructing standard decoupling techniques. As such, the locus of locally fiducial Higgs bundles provides the natural setting for extending exponential convergence results to singular fibers.

\begin{theorem}\label{thm:intro1}
	Suppose $(\ME, \vp)$ is locally fiducial and $q = \det(\vp)$ has at least one zero of odd order. Then for any relatively compact subset $K \subset \Sigma \setminus Z$ and any $l \in \mathbb{N}$, there exist positive constants $C_{l,K}$ and $C'_{l,K}$, independent of $t$, such that
	\[
	\|(\nabla_t, \phi_t) - (\nabla_\infty, \phi_\infty)\|_{\MC^l(K)} \leq C_{l,K} e^{-C'_{l,K} t}.
	\]
\end{theorem}

\subsection{Asymptotics of the hyperkähler metric on the closed strata of singular Hitchin fibers}
In the second part of the paper we return to the question about the asymptotics of the hyperkähler metric. In the recent works \cite{hitchin2019critical,Hitchin2021subintegrable_special_Kaehler}, Hitchin observed that there are certain subsets $\calM^d$ of the singular locus, such that the restricted Hitchin map $\calH^d: \calM^d \to \calB^d$ defines an algebraically completely integrable system of lower dimension. In turn, Freed's construction can be applied to obtain a semi-flat hyperkähler metric $g_{\sfm,d}$ on $\calM^d$. 
Hitchin then posed the following question:

\begin{question}\cite[Section 5.4]{Hitchin2021subintegrable_special_Kaehler}
	Does the semi-flat metric $g_{\sfm,d}$ provide an approximation of the restricted hyperkähler metric $g_{L^2}\rest_{\MM^d}$?
\end{question}

\par All Higgs bundles in the subintegrable system $\calM^d$ are locally fiducial. Hence, Theorem \ref{thm:intro1} proves the exponential decay of solutions to the Hitchin equations along the subintegrable systems $\calM^d$. Extending the strategy developed in \cite{mochizuki2023asymptotic2} this allows us to give an affirmative answer to Hitchin's question.

\begin{theorem}\label{thm:intro_main1}
	Let $(\ME,t\vp)$ be a ray in $ \MM^d$ with $d < 2g-2$, then there exists  $\varepsilon > 0$ such that
	\[
	(g_{L^2}|_{\MM^d} - g_{\sfm,d})|_{(\ME, t\vp)} = \MO(e^{-\varepsilon t}),
	\]
	as $t \to \infty$.
\end{theorem}
\noindent Note that the case $d=2g-2$ is treated in \cite{Francoetal_unramified} (see Section \ref{ssec:d=2g-2} for more details).
\par
More generally, our method applies to Higgs bundles in the closed stratum of any singular Hitchin fiber for a quadratic differential with at least one zero of odd order. For each partition $\bp$ of $4g-4$, let $\mathcal{B}_{\bp}$ denote the subset of quadratic differentials on $\Sigma$ such that the zero divisor induces the partition $\bp$. Let $\MM_{\bp,V_{\max}}$ denote the collection of closed strata (see Section~\ref{ssec:Hitchin_subint}). The restricted Hitchin map
\[ \MM_{\bp,V_{\max}}\to \mathcal{B}_{\bp} \]
is no longer an algebraically completely integrable system, as the fiber dimension exceeds the base dimension. Nonetheless, we construct a Riemannian metric $g_{\sfm,\bp}$ on $\MM_{\bp,V_{\max}}$ that naturally generalizes the semi-flat metric and prove:

\begin{theorem}
	Suppose the partition $\bp$ is of length $|\bp|> 2g-2$ and contains at least one odd number. Let $(\ME,t\vp)$ be a ray in $\MM_{\bp,V_{\max}}$. Then  there exists $\varepsilon > 0$ such that
	\[ (g_{L^2}|_{\MM_{\bp,V_{\max}}} - g_{\sfm,\bp})|_{(\ME, t\vp)} = \MO(e^{-\varepsilon t}), \]
	as $t \to \infty$. If $\vert \bp \vert \leq 2g-2$ the result still holds for generic rays.
\end{theorem}
\textbf{Acknowledgements.} The authors wish to express their gratitude to many people for their interest and helpful comments. We especially thank Rafe Mazzeo, Laura Fredrickson, Takuro Mochizuki, Martin Möller and Ricahrd Wentworth. S.H. is partially supported by NSFC grant No.12288201 and No.2023YFA1010500. J. H. acknowledges support by Deutsche Forsch\-ungs\-ge\-meinschaft (DFG, German Research Foundation) through the Collaborative Research Centre TRR 326 \textit{Geometry and Arithmetic of Uniformized Structures}, project number 444845124. N.L. is supported by the Shuimu Tsinghua Scholar Program.

\section{Higgs bundles and the Hitchin equations}
This section provides a concise overview of Higgs bundles, the Hitchin moduli space, the spectral curve construction, and the nonabelian Hodge correspondence. For more comprehensive accounts, see \cite{hitchin1987self, hitchin1987stable, Simpson1992}.
\subsection{Higgs bundles}
Let $\Sigma$ be a compact Riemann surface of genus $g \geq 2$, with structure sheaf $\MO = \MO_\Sigma$ and canonical bundle $K = K_\Sigma$. A holomorphic vector bundle $\calE$ on $\Sigma$ is equivalent to a pair $(E,\overline{\partial}_E)$ consisting of a smooth complex vector bundle $E \to \Sigma$ and a Dolbeault operator $\overline{\partial}_E: E \to \Omega^{0,1}(E)$. A Higgs bundle is a pair $(\mathcal{E}, \varphi)$ of a holomorphic vector bundle $\calE$ and a holomorphic section $\varphi \in H^0(\End(\mathcal{E}) \otimes K)$ called the Higgs field. A pair $(\mathcal{E}, \varphi)$ is called an $\SLC$-Higgs bundle if $\rank(\calE) = 2$, $\det(\mathcal{E})$ is holomorphically trivial, and $\Tr(\varphi) = 0$. In this paper, we primarily focus on $\SLC$-Higgs bundles, although rank one Higgs bundles also play a significant role.


An $\SLC$-Higgs bundle $(\mathcal{E}, \varphi)$ is called stable (resp. semistable) if for every line subbundle $\mathcal{L} \subset \mathcal{E}$ such that $\varphi(\mathcal{L}) \subset \mathcal{L} \otimes K$, we have $\deg(\mathcal{L}) < 0$ (resp.\ $\leq 0$). 

If $(\ME,\vp)$ is strictly semistable but not polystable, then the Seshadri filtration  \cite{seshadri1967space} introduces a rank one Higgs bundle $(\mathcal{L}, \omega)$ with $\deg(\mathcal{L}) = 0$, such that the associated graded object $\mathrm{Gr}(\mathcal{E}, \varphi)$ is polystable and isomorphic to $(\mathcal{L}, \omega) \oplus (\mathcal{L}^{-1}, -\omega)$, and we say these two Higgs bundles are S-equivalent. We write $\mathcal{E} = (E, \bar{\partial}_E)$ and define the space
\[ \MC := \{ (\bar{\partial}_E, \varphi) \mid \bar{\partial}_E \varphi = 0, \Tr(\varphi) = 0  \}\,.
  \]
Let $\MC^s$ and $\MC^{ss}$ denote the subsets of stable and semistable Higgs bundles, respectively. The action of the complex gauge group $\MGC := \Aut(E)=\Gamma(\mathrm{GL}(E))$ on $\MC$ is given by $ (\bar{\partial}_E, \varphi)\cdot g:= (g^{-1} \comp \bar{\partial}_E \comp g,\ g^{-1} \varphi g),$ where $g \in \MGC.$

There exists a finite-dimensional quasi-projective moduli space, the Dolbeault moduli space, denoted by $\MMD$. It is constructed using geometric invariant theory, and its closed points correspond to isomorphism classes of polystable $\SLC$-Higgs bundles. We use $\MMD^{\mathsf{st}}$, or simply $\mathcal{M}^{\mathsf{st}}$, to denote the moduli space of stable $\SLC$-Higgs bundles. For more details, we refer the reader to \cite{Nitsure:91, simpson1994moduli}.

\subsection{Spectral curves and the Hitchin fibration}
The Hitchin fibration is the map
\[
\MH: \MM \to H^0(K^2), \quad [(\mathcal{E}, \varphi)] \mapsto \det(\varphi)\,.
\]
In the $\SLC$-case, the Hitchin base is the vector space of holomorphic quadratic differentials on $\Sigma$ which we denote $\MB := H^0(K^2)$. Hitchin \cite{hitchin1987self,hitchin1987stable} showed that $\MH$ is a proper map and defines a fibration by (torsors of) abelian varieties over the open subset $\MB^{\mathrm{reg}} \subset \MB$, consisting of quadratic differentials with simple zeros only. The discriminant locus or the singular locus $\MB^{\mathrm{sing}} := \MB \setminus \MB^{\mathrm{reg}}$ contains differentials with multiple zeros or that vanish identically. For any $q \in \MB$, the fiber $\mathcal{M}_q := \MH^{-1}(q)$ is the moduli space of Higgs bundles with spectral data $q$. A fiber $\mathcal{M}_q$ with  $q\in \MB^{\mathrm{sing}}$ is referred to as a singular Hitchin fiber.

The Hitchin fibers are identified with an abelian torsor by the spectral correspondence. To describe the spectral curve, consider the total space $\Tot(K)$ of the canonical bundle $K$, with projection $\pi: \Tot(K) \to \Sigma$. The pullback bundle $\pi^* K$ has a tautological section $\lambda \in H^0(\Tot(K), \pi^* K)$. For $q \in H^0(K^2)$, the associated spectral curve $S_q$ is defined as $S_q=\{\lam^2-\pi^*q=0\}\subset \Tot(K)$. The projection $\pi: S_q \to \Sigma$ is a branched double cover and the curve $S_q$ is smooth if and only if $q$ has only simple zeros—that is, when $q \in \MB^{\mathrm{reg}}$.

When $S_q$ is irreducible, the BNR correspondence provides a description of the rank two Hitchin fiber. Let $\BPic^d(S)$ be the moduli space of degree $d$ rank 1 torsion-free
sheaves on $S_q$, and $\BPic(S)=\prod_{d\in \mathbb{Z}}\BPic^d(S)$ \cite{cyril1979compactification}.
Then $\BPic^d(S)$ is an irreducible projective scheme containing
$\Pic^d(S)$ as an open subscheme. When $S$
is smooth, all torsion-free sheaves are locally free, i.e.\ $\BPic^d(S)=\Pic^d(S)$. Now, we can formulate the BNR-correspondence:

\begin{theorem}[{\cite{hitchin1987stable,bnr1989spectral}}]
	\label{thm_BNRcorrespondence}
	Let $q\in H^0(K^2)$ be a quadratic differential that is not the square of an abelian differential. There is a bijective correspondence between
	points in $\BPic(S_q)$ and
	isomorphism classes of rank 2 Higgs pairs $(\ME,\vp)$ with
	$\Tr(\vp)=0$ and
	$\det(\vp)=q$. Explicitly:
	if $\calL\in \BPic(S_q)$, then $\ME:=\pi_{\st}(\calL)$ is a rank 2 holomorphic vector
	bundle, and the homomorphism $\pi_{\st}\calL\to \pi_{\st}\calL\otimes K\cong
	\pi_{\st}(\calL\otimes \pi^{\st}K)$ given by multiplication by the
	tautological section $\lam$ defines the Higgs field $\vp$.
      \end{theorem}

For the $\SLC$-version of this statement, see \cite{HauselPauly2012}. However, in this work we take another approach to the singular Hitchin fibers developed in \cite{horn2022semi} by the second author. In a first step one stratifies the fibers by the vanishing of the Higgs field. Sequentially each stratum can be parametrized by line bundles on the normalized spectral curve plus Hecke parameters at the higher order zeros of the quadratic differentials. This will be reviewed in detail in Section \ref{subsec:stratification}. Let us fix notation for the normalized spectral curve.
\subsection{The normalized spectral curve}
For each $0 \neq q \in \MB^{\mathrm{sing}}$ we have a diagram of spectral curves
\begin{equation}
	\label{eq_normalization}
	\begin{tikzcd} \widetilde{S}_q \ar[r,"n"]\ar[rd,"\tilde{\pi}_q"] & S_q  \ar[d,"\pi_q"] \\ & \Sigma
	\end{tikzcd}
\end{equation}
where $S_q$ is the spectral curve and $\wt{S}_q=\mathsf{Norm}(S_q)$ is the normalization of $S_q$. The genus of $S_q$ is $\tg:=g(\wt{S}_q)=2g-1+\frac{r_{\odd}}{2}$, where $r_{\odd}$ denotes the number of odd order zeros of $q$. Let $\tilde{\lambda}= n^* \lambda \in \Gamma(\tilde{\pi}_q^*K)$ be the pullback of the eigensection of the Higgs field satisfying $\tilde{\lambda}^2+ \tilde{\pi}_q^*q=0$. Let $\Lambda=\mathsf{div}{\tilde{\lambda}}$. By definition, the branch locus of $\tilde{\pi}$ are the zeros of $q$ of odd order. Let $D=D(q)=\lfloor \frac{\mathsf{div}(q)}{2} \rfloor$. Then the branch divisor is $\mathsf{div}(q)-2D$. Let $p_D: \mathsf{Tot} K(-D) \to \Sigma$ be the projection and $T: \mathsf{Tot} K(-D) \to p_D^*K(-D)$ be the tautological section. The normalized spectral curve is the zero locus $\{T^2+p_D^*(q/s_D^2)=0\} \subset \mathsf{Tot} K(-D)$, where  $s_D \in H^0(\Sigma, \mathcal{O}(D))$ is vanishing on $D$. In particular, the canonical bundle of the normalized spectral curve is $K_{\wt{S}_q}= \tilde{\pi}^*K_\Sigma(-D)$ by adjunction formula. Furthermore, we have
\begin{align} \tilde{\pi}_*\calO_{\wt{S}_q}=\calO_\Sigma \oplus K_\Sigma^{-1}(D). \label{equ:push_structure_sheaf}
 \end{align}
 There is also a canonical differential $\tilde{\omega} \in H^0(\widetilde{S}_q,K_{\widetilde{S}_q})\cong H^{(1,0)}(\wt{S}_q)$, such that $\tilde{\omega}^2+\tilde{\pi}^*q=0 \in H^0(\widetilde{S}_q,K_{\widetilde{S}_q}^2)$. The sections $\tilde{\omega}$ and $\tilde{\lambda}$ differ by the ramification divisor $R \subset \wt{S}_q$
 \[ \mathsf{div}\tilde{\omega}=\mathsf{div} \tilde{\lambda} + R\,.
   \]   The differential $\tom$ can be understood as a generalization of the Seiberg-Witten differential in this situation. 

\subsection{The stratification of the Hitchin base}\label{ssec:stratification_base}
For each partition
\[ \bp=(m_1,\cdots, m_{r_{\even}},n_1,\cdots,n_{r_{\odd}}) \in \mathbb{N}^{r_{\even}+r_{\odd}}
\]
with $\sum_{i=1}^{r_\odd}n_i+ \sum_{j=1}^{r_\even}m_j=4g-4$ we denote by $\mathcal{B}_{\bp} \subset H^0(\Sigma,K_{\Sigma}^2)$ the subvariety of quadratic differentials $q$ with zero divisor $\mathsf{div}(q)=\sum_{i=1}^{r_\odd}n_ip_i+ \sum_{j=1}^{r_\even}m_jp_j'$ for pairwise distinct $p_i,p_j' \in \Sigma$. The partitions of the form $\bp=(2,\dots,2,1,\dots,1)$ will be important below and we introduce the shorthand $\calB^d=\calB_{\bp}$ for stratum associated to the partition with $d$ twos and $4g-4-2d$ ones.
\par For  $q\in \calB_{\bp}$, the normalized spectral curves have fixed genus $\tg:=g(\wt{S}_q)=2g-1+\frac{r_{\odd}}{2}$. By varying $q$ we obtain the family of normalized spectral curves $\wt{S} \to \calB_{\bp}$ with degree two holomorphic map $\tilde{\pi}: \widetilde{S} \to  \calB_\bp \times \Sigma$.

Given a quadratic differential $q$, we define the divisor
\begin{equation}
	D_0(q)=(n_1-1)p_1+\cdots+(n_{r_{\odd}}-1)p_{r_{\odd}}+(m_1-1)p_1'+\cdots+(m_{r_{\even}}-1)p_{r_{\even}}'
\end{equation}
with $\deg(D_0(q))=4g-4-(r_{\even}+r_{\odd})$. 

\begin{lemma}\label{lemm:tangent_space_base_stratum} Let $\bp$ be a partition of $4g-4$. Then the tangent space to the stratum $\calB_{\bp}$ is given by
  \[ T_q\mathcal{B}_{\bp}=H^0(\Sigma,K^2(-D_0(q)))\,.
  \] In particular, if $r_{\even}+r_{\odd}> 2g-2$, then the stratum $\calB_{\bp}$ is smooth of dimension $r_{\even}+r_{\odd}-(g-1)$. More generally, as long as $r_{\even}+r_{\odd}-(g-1)>0$, this is the dimension of the tangent space at a smooth point of $\mathcal{B}_{\bp}$.
\end{lemma}
\begin{proof}
  In general, the tangent space is contained in the left side, as for a family of quadratic differentials $q_t \in \calB_{\bp}$ locally at a zero of order $k$ we have
  \[ q_t=(z-\gamma(t))^k \mathrm{d} z^{\otimes 2}= (z^k+ \gamma(t)z^{k-1} + \dots ) \mathrm{d} z^{\otimes 2},
  \] where $\gamma(t)$ is the path of the zero on $\Sigma$. Hence, $\frac{\mathrm{d}}{\mathrm{d}t}\rest_{t=0} q_t \in H^0(\Sigma,K^2(-D_0(q)))$. On the other hand, it was proven in \cite[Proposition~4]{Bud2021image_of_strata} that the dimension of the tangent space is $d$, where $d=\dim H^0(X, \calO(\sum_i p_i+\sum_jp_j'))$. Note that $\calO(\sum_i p_i+\sum_jp_j')=K^2(-D_0(q))$ and hence the result follows. The dimension can be computed via the Riemann-Roch theorem.
\end{proof}

Denote by $H^0(\wt{S}_q,K_{\wt{S}})^-$ the $-1$-eigen space of the induced action of the involution $\sigma$ on $H^0(\wt{S}_q,K_{\wt{S}})$. We identify the tangent vectors to the stratum with anti-invariant holomorphic one-forms on the normalized spectral curve via
\begin{equation}\label{eq:tau_map}
	\tau: H^0(\Sigma,K^2(-D_0(q))) \to H^0(\widetilde{S}_q,K_{\widetilde{S}})^-, \quad \dot{q} \mapsto \frac{\tilde{\pi}^*\dot{q}}{2\tilde{\omega}}.
\end{equation}
This map is clearly injective and an isomorphism for $q \in \calB^d$ since the dimensions match.

\subsection{The stratification of the singular fiber}\label{subsec:stratification}
We will recall the stratification of the singular $\SLC$-Hitchin fiber with irreducible spectral curve introduced in \cite{horn2022semi}. Fix $0 \neq q \in \calB$ that has no global square-root.  To $(\calE,\varphi) \in \calH^{-1}(q)$ we can associate a so-called Higgs divisor $\mathsf{div}(\varphi)$ recording the total vanishing of the Higgs field, i.e., $\mathsf{div}(\varphi)|_p=k$ if $\varphi=z^k\varphi'$ with $\varphi'(0)\neq 0$ and $z$ a holomorphic coordinate centered at $p$. By computing the determinant of the Higgs field we obtain $0 \leq \mathsf{div}(\varphi) \leq D(q)=\lfloor \frac{\mathsf{div}(q)}{2} \rfloor$. In particular, the maximal vanishing divisor is
\begin{align} V_{\max}= \lfloor \frac{\mathsf{div}(q)}{2} \rfloor \label{eq_maximal_sigma_divisor}
\end{align}
We decompose the Hitchin fibers according to the Higgs divisors
\[ \calH^{-1}(q)= \bigcup_{0 \leq V \leq D(q)} \calT_V\,, \quad \text{where} \quad \calT_V=\{ (\calE,\varphi)\in \calH^{-1}(q) \mid \mathsf{div}(\varphi)=V\}.
\]
\begin{theorem}[{\cite[Theorem~5.3]{horn2022semi}}]
  Let $0 \leq V \leq D(q)$. Then $\calT_V$ is locally closed and smooth of dimension $3g-3-\deg V$ and we have $\calT_V \subset \overline{\calT}_{V'}$ for $V \geq V'$.
\end{theorem}
The strata $\calT_V$ can be described explicitly using semi-abelian spectral data. For $0 \leq V \leq D(q)$ denote by
\begin{equation} \calP_V= \{ \calL \in \Pic(\wt{S}_q) \mid \mathsf{Nm}(\calL)= K(-V)  \}, \label{eq:twisted_Prym}
\end{equation}
where $\mathsf{Nm}$ is the norm map of $\tpi:\wt{S}_q\to \Sigma$. This is a torsor over the Prym variety $\Prym(\wt{S}_q/\Sigma)= \{ \calL  \in \Pic(\wt{S}_q) \mid \mathsf{Nm}(\calL)=\calO_\Sigma \}$, which is connected if $q$ has at least one zero of odd order and has two connected components otherwise. Its dimension is given by $\dim \Prym(\wt{S}_q/\Sigma)=g-1+\frac12r_{\odd}$ and its tangent space is $H^1(\wt{S},\calO_{\wt{S}})^-$, which denotes the $-1$-eigen space by the induced action of $\sigma$ on $H^1(\wt{S},\calO_{\wt{S}})$.
\begin{theorem}[{\cite[Theorem~6.2]{horn2022semi}}]
  	\label{thm_stratification_fibration}\label{prop_stratification_BNR_data}
        Let $0 \leq V \leq \mathsf{div}(q)$ and let $n_{ss}$ be the number of points $p \in \Sigma$ such that $V_p=\mathsf{div}(q)_p$. There exists a holomorphic map $\mathsf{Ab}: \calT_V \to \calP_V$ with fiber over $\calL \in \calP_V$ given by
           \[ \mathsf{Ab}^{-1}(\calL)=(\mbC^{*})^{k_1}\ti \mbC^{k_2}
           \]
             where $k_1=r_{\even}-n_{ss}$, $k_2=2g-2-\frac12\deg(V)-r_{\even}+n_{ss}-\frac {r_{\odd}}{2}$, and $r_{\even},r_{\odd}$ are the number of even and odd zeros of $q$. In particular, $\dim(\calT_V)=3g-3-\deg(V)$.
\end{theorem}
\begin{proof}
  We explain the proof in some detail as it will be important in the remainder of the work. Here we present an equivalent (in fact dual) reformulation of the approach used in \cite{horn2022semi}. If $(\calE,\varphi)$ is a $\SLC$-Higgs bundle, so is its dual $(\calE^\vee, \varphi^\vee)$. Define a line bundle $\calL^\vee$ on the normalized spectral curve $\tilde{\pi}: \widetilde{S}_q \to \Sigma$ by
  \begin{align} \calL^\vee= \ker( \tilde{\pi}^*\varphi^\vee- \tilde{\lambda} \id_{\tilde{\pi}^* \calE^\vee}) \hookrightarrow \tilde{\pi}^*\calE^\vee\,.
    \label{equ:eigenline_bundle}
  \end{align}
  Dualizing we obtain a morphism $\tilde{\pi}^*\calE \to \calL$ and the pushforward of this morphism composed with the adjoint map yields
  \[ \calE \to \tilde{\pi}_*\tilde{\pi}^*\calE \to \tilde{\pi}_* \calL\,.
  \]
  The composition is a holomorphic Hecke modification $\eta: \calE \to \tilde{\pi}_*\calL$, an inclusion of sheaves that is an isomorphism on $\Sigma\setminus \pi(\mathsf{Sing}(S_q))$. The cokernel $\mathsf{coker}(\eta)$ defines a torsion sheaf supported on $\pi(\mathsf{Sing}(S_q))$. A computation in local coordinates using \cite[Lemma 5.1]{horn2022semi} proves the relation
    \begin{align*}
       \mathsf{length}_p (\coker(\eta)) = \lfloor \frac{\mathsf{div}(q)_p}{2} \rfloor - V_p\,.
    \end{align*}
    Let $D= \lfloor \frac{\mathsf{div}(q)}{2} \rfloor$ as above. Computing determinants we obtain
    \[ \calO_\Sigma=\det(\calE)=\det(\tilde{\pi}_*\calL)(-D + V)= \mathsf{Nm}(\calL) \otimes \det(\tilde{\pi}_* \calO_{\wt{S}_q}) \otimes \calO(-D + V) \,. \]
   By equation~\eqref{equ:push_structure_sheaf} $\det(\tilde{\pi}_* \calO_{\wt{S}_q})=K^{-1}(D)$ we conclude $\mathsf{Nm}(\calL)= K(-V)$ and $\calL \in \Prym_V$ as desired. This defines the holomorphic map $\mathsf{Ab}$.
    \par The fibers of $\mathsf{Ab}$ parametrize the choice of Hecke modification $\eta$ up to equivalence. Locally on $\Sigma$, by changing frame of $\calE$ and $\tilde{\pi}_*\calL$, there is a Smith normal form $\eta=\diag(z^{n},z^{m})$. By construction the Hecke modifications does not change the eigen line bundle $\calL$. Hence, we must have $n=0$ and $m= D_p - V_p$. In particular, the Hecke modification is an isomorphism if $m=0$. For $m>0$, we like to parametrize the Hecke modification by describing it with respect to some preferred frame of $\tilde{\pi}_*\calL$. From here on we distinguish the case of even and odd order zeros of $q$.
    \par
    \textbf{Hecke parameters at a zero of order $2k+1$:} \quad Denote the zero set of $q$ by $Z(q)$. Let $p \in Z(q)$ be a zero of odd order and let $s$ be a frame of $\calL$ in a neighborhood of ${\tilde{\pi}^{-1}p}$. Then $s,zs$ is local frame of the pushforward $\tilde{\pi}_*\calL$ at $p$. The eigensection $s$ projects to zero in $\coker(\eta)$. Therefore, we can assume the Hecke modification to be of the form
    \begin{align} \eta = \begin{pmatrix} 1 & 0 \\ a & z^m \end{pmatrix}. \label{eq:Hecke_mod}
    \end{align} Up to equivalence the Hecke modification is uniquely determined by the truncated germ $u=[a]=u_0+u_1z+ \dots +u_{m-1}z^{m-1} \in \CC[z]/z^m$ (see \cite[Lemma 5.1]{horn2022semi} for details about the dual description). Hence, the Hecke modifications for stratum $\calT_V$ are parametrized by $\CC^{k-V_P}$. This recovers the $u$-coordinate as described in \cite[Proposition~5.8]{horn2022semi}, explicitly $u_s=zu$.
    \par
    \textbf{Hecke parameters at a zero of order $2k$:} \quad
    Let $p \in Z(q)$ be a zero of even order and $\tilde{\pi}^{-1}p=\{y_1,y_2\}$. Choose frames $s_i$ of $\calL$ in a neighborhood of $y_i$ for $i=1,2$. We can choose either of the sections $s_i$ to be fixed and represent every Hecke modification by a matrix of the form \eqref{eq:Hecke_mod}. The condition that up to equivalence the eigen line bundles descend to local line subbundles is $a(0)\neq 0$. Hence, the Hecke modification is uniquely determined by the truncated germ $c=[a]=c_0+c_1z+ \dots c_{m-1}z^{m-1} \in \CC[z]/z^m$ with $c_0 \neq 0$. Therefore, the Hecke modifications for the stratum $\calT_V$ at the zero $p$ are parametrized by $\CC^* \times \CC^{k-V_p-1}$.
  \end{proof}

\textbf{u-coordinates at even order zeros:} \quad  Below it will be more convenient to work with a different parametrization at the even zeros. Fix the frame $s_+=s_1+s_2, s_-=s_1-s_2$ of $\tilde{\pi}_*\calL$ in a neighborhood of $p$. The condition that the eigen sections $s_1,s_2$ descend to eigen sections of $(\calE,\varphi)$ implies that at least one of $s_+$ or $s_-$ descends to a section of $\calE$. This choice leads to two coordinate charts. First, let us assume that $s_+$ descends to $\calE$. Then we can represent the Hecke modification by a matrix of the form \eqref{eq:Hecke_mod} and obtain a coordinate $u_+=[a]=u_{+,0}+u_{+,1}z+ \dots u_{+,m-1}z^{m-1} \in \CC[z]/z^m$. The condition that $s_1,s_2$ descend to section of $\calE$ is equivalent to $u_{+,0} \neq \pm 1$ up to equivalence of Hecke modifications. Second, if $s_-$ descends to a section of $\calE$, we obtain Hecke parameters $u_- \in \CC[z]/z^m$. Therefore, we parametrize the moduli of Hecke parameters $\cong \CC^* \times \CC^{k-V_p-1}$ by two open sets $U_+,U_-$ with coordinates $u_\pm \in \CC^m\setminus \{ \pm 1\} \times \CC^{m-1}$. The transition function on the overlap $u_\pm \neq 0$ is $u_+ \mapsto u_-=[ u_+^{-1}] \in \CC[z]/z^m$.
\par
The advantage of the description in $u$-coordinates is that the Higgs field $\varphi$, which is the pullback of $\tilde{\pi}_*\tilde{\lambda}$ by $\eta$ takes a convenient form. Let $\mathsf{ord}_p(q)=n$ and $m=\lfloor \frac{n}{2}\rfloor - V_p$. With respect to the induced frame of $\calE$ and the corresponding $u$-coordinate the Higgs field is given by
\begin{align}   &\varphi=\begin{pmatrix} u z^{k+1} & z^{m+k+1} \\ z^{k-m}(1-u^2z) & -u z^{k+1} \end{pmatrix} d z \quad \text{for } n=2k+1\,, \label{equ:phi_u-coord_odd}\\
\qquad &\varphi=\begin{pmatrix} u_\pm z^{k} & z^{m+k} \\
  z^{k-m}(1-u_\pm^2) & -u_\pm z^{k} \end{pmatrix} d z \quad \text{for } n=2k\,. \label{equ:phi_u-coord_even}
\end{align}
The $u$-coordinates depend on the choice of local sections of $\calL$ at the preimages of the zeros $\tilde{\pi}^{-1}(Z(q))$. However it is easy to see that the coordinates are invariant under multiplying $s$ at the preimage of an odd zero or $s_+,s_-$  at the preimage of an even zero with a constant (cf. \cite[Proposition~5.8]{horn2022semi}).

\begin{corollary}[Minimal dimensional stratum]\label{coro:max_stratum}
  Let $V=V_{\text{max}}=\lfloor \frac{\mathsf{div}(q)}{2} \rfloor$. Then there is an isomorphism $\calT_{V_{\text{max}}} \cong \calP_{V_{\text{max}}}$ its inverse is given by $\calL \to \tilde{\pi}_*(\calL,\tilde{\lambda})$.
\end{corollary}

\subsection{Rank of Hitchin map}
In this section we will show that the vanishing divisor of the Higgs field controls the rank of the Hitchin map.

\begin{definition}\label{def:Mr}
  Let $0 \leq r \leq 3g-3$. Define
  \[ \calM_r=\{(\calE,\varphi) \in \calM^{\mathsf{st}} \mid \rank(d \calH_{(\calE,\varphi)})=r\}.
  \]
\end{definition}

The tangent space of $\mathcal{M}$ at a stable point $(\calE,\varphi)$ is given by the hypercohomology of the complex of sheaves
\begin{align} \mathsf{Def}(\calE,\varphi): \quad \End^0(\calE) \xrightarrow{[\varphi, \cdot ]} \End^0(\calE) \otimes K, \label{equ:def_complex}
\end{align}
where $\End^0(\calE)$ denotes the $\Tr=0$ endomorphisms. The homology sheaves of this complex
\[ \calH^0=\ker[\varphi, \cdot], \quad \calH^1=\coker[\varphi,\cdot]
\] give rise to a spectral sequence with second page given by
\[ 	\begin{tikzcd}
		 E_2^{10}=H^1(\Sigma,\calH^0 ) & E_2^{11}=H^1(\Sigma,\calH^1)  \\
		 E_2^{00}=H^0(\Sigma,\calH^0) & E_2^{01}=H^0(\Sigma,\calH^1).
	\end{tikzcd}
\] From the five-term exact sequence we obtain
\begin{align} 0 \rightarrow H^1(\Sigma,\ker[\varphi,\cdot ]) \rightarrow T_{(\calE,\varphi)}\mathcal{M}=H^1(\Def(\calE,\varphi)) \rightarrow H^0(\Sigma, \coker[\varphi,\cdot]) \rightarrow 0. \label{equ:TM}
\end{align}

\begin{proposition}\label{prop:diff_of_Hit}
  Let $(\calE,\varphi) \in \mathcal{M}^{\mathsf{st}}$ with $\varphi \neq 0$ and vanishing divisor $V=V(\varphi)$. Then there is an exact sequence
\begin{align} 0 \rightarrow  H^1(\Sigma,K^{-1}(V)) \rightarrow T_{(\calE,\varphi)}\calM \rightarrow  H^0(\Sigma,K^{2}(-V)) \oplus H^0(V,\calO_V^2) \rightarrow 0, \label{eq:ses:tangent_space}
\end{align} where the projection to the first factor of the second map is the differential of the Hitchin map.
\end{proposition}
\begin{proof}
  Let $s_V$ be a canonical section of $\calO(V)$. Clearly, $\ker[\varphi,\cdot]=\ker[\varphi/s_V,\cdot]$ and the image of $\varphi/s_V: K^{-1}(V) \to \End^0(\calE)$ is contained in the kernel. Because, $\varphi/s_V \neq 0$ and $\Tr(\varphi)=0$ its point-wise commutator is indeed one-dimensional. Hence, $\ker[\varphi,\cdot]\cong K^{-1}(V)$. The trace pairing
\[ -\Tr(\cdot,\cdot):\quad  \End^0(\calE) \otimes \End^0(\calE) \otimes K(-V) \to K(-V).
\]
identifies the cokernel $\coker[\varphi/s_V,\cdot]$ with $\ker[\varphi/s_V,\cdot]^\vee \otimes K(-V)=K^2(-2V)$. We obtain the following commuting diagram
\[ \begin{tikzcd} 0 \ar[r] & \End^0(\calE) \ar[r,"\id"] \ar[d,"{[ \varphi/s_V,\cdot ]} "] & \End^0(\calE)  \ar[r] \ar[d,"{[ \varphi,\cdot ]}"] & 0 \ar[d] &\\
0 \ar[r] & \End^0(\calE)\otimes K(-V) \ar[r,"s_V \id"] \ar[d,"c_10",two heads] & \End^0(\calE) \otimes K \ar[r] \ar[d,"c_2",two heads] &  \calO_V^3 \ar[r] \ar[d,two heads]& 0 \\
0 \ar[r] & \coker[\varphi/s_V,\cdot ] \ar[r,"s_V"] & \coker[\varphi,\cdot ]\ar[r] & \calO_V^3 \ar[r] & 0 & 
\end{tikzcd}.
\] Therefore, the cokernels are related by
\[ \coker[\varphi,\cdot] = \coker[\varphi/s_V,\cdot] \otimes\calO(V) \oplus \calO_V^2\cong K^2(-V) \oplus \calO_V^2.
\] Together with (\ref{equ:TM}) we obtain the exact sequence stated in the proposition.
We are left with showing that the differential to the Hitchin map is identified with the projection to $H^0(\Sigma,K^{2}(-V))$. Using the Dolbeault resolution, an element of $H^1(\Def(\calE,\varphi))$ can be represented by $\tau \in A^{1}(\End^0(\calE))$, such that $(\overline{\partial}_E + \ad\varphi)\tau=0$. Then the projection to the torsion-free part cokernel $c_2(\tau^{(1,0)}) \in A^{(1,0)}(\coker([ \varphi,\cdot ]))$ is holomorphic. The torsion-free part of the cokernel $\coker([ \varphi,\cdot ])$ is identified with $K^2(-2V)$ by $\alpha \mapsto -\Tr(\alpha\varphi)$. Hence, the induced map on hypercohomology is given by
\[ H^1(\Def(\calE,\varphi)) \to H^0(\Sigma,K^2(-V)), \quad \tau \mapsto -\Tr(\tau^{(1,0)}\varphi)\,.
\]   Note that for $X\in \mathfrak{sl}(2,\mathbb{C})$ the determinant can be expressed in terms of the trace of the square by $\det(X)=-\tfrac12 \Tr(X^2)$. Hence, the above map is exactly the differential of the Hitchin map.
\end{proof}

\begin{proposition}\label{prop:relation_rank_divisor}
  We have the following identifications. For $r \geq g$
  \[ \calM_r=\left\{ (\calE,\varphi) \in \calM^{\mathsf{st}} \mid \varphi \neq 0, \deg(\mathsf{div}(\varphi))=3g-3-r \right\}
    \] and
    \[ \calM_{g-1} = \left\{ (\calE,\varphi) \in \calM^{\mathsf{st}} \mid \varphi \neq 0, \deg(\mathsf{div}(\varphi))=2g-2, \calO(\mathsf{div}(\varphi)) \ncong K_\Sigma \right\}.
  \]
\end{proposition}
\begin{proof}
By Proposition \ref{prop:diff_of_Hit} for $\varphi \neq 0$ this reduces to the computation of the dimension of $H^0(\Sigma,K^2(-V))$. For $\deg(V) < 2g-2$ this is Riemann-Roch. For $\deg(V)=2g-2$, the dimension equals $g-1$, if and only if $\calO(V) \ncong K$. If $\calO(V) \ncong K$ and $\det(\varphi)\neq0$, the spectral curve is reducible and $(\calE,\varphi)$ is strictly polystable. If the Higgs field $\varphi \neq 0$ is nilpotent, then we can globally write
\[ (\calE,\varphi)=\left(\calL \oplus_b \calL^{-1}, \begin{pmatrix} 0 & \alpha \\ 0 & 0 \end{pmatrix}\right)
\] with $\alpha \in H^0(\Sigma,\calL^2K)$. The condition $\calO(V) \cong K$ implies $\deg(\calL)=0$. Hence, $(\calE,\varphi)$ is again strictly semi-stable. Finally, if $\varphi=0$, the differential of the Hitchin map $\frac12\Tr(\varphi\phi)$ is zero.
\end{proof}
\noindent The first term in the exact sequence \eqref{eq:ses:tangent_space} is the tangent space of the stratum.
\begin{lemma}\label{lemm:cohomology_spectral}
  Let $(\calE,\varphi) \in \calT_V$. Then the tangent space to the stratum of the singular Hitchin fiber indexed by $V$ is given by $T_{(\calE,\varphi)}\calT_V=H^1(\Sigma,K^{-1}(V))$.
\end{lemma}
\begin{proof}
  Note that by \cite[Lemma 1]{Cook_paper} every torsion-free sheaf $\calF$ on $S_q$ is given by the pushforward of a locally free sheaf $\calF'$ on a partial normalization. The partial normalizations of $S_{q,V}$ are the curves defined as the zero locus
  \[ S_{q,V}=Z(\lambda^2+ \det(\varphi)/s_V^2) \subset \mathsf{Tot}(K_\Sigma(-V))\,.
  \] The stratum $\calT_V$ is the subset of torsion-free sheaves that are the pushforward of a locally free sheaf on $S_{q,V}$. Let $\pi_{q,V}:S_{q,V} \to \Sigma$ be the induced projection. This identifies the stratum with the Prym variety $\calT_V \cong \Prym(\pi_{q,V}:S_{q,V} \to \Sigma)$. In particular, we obtain the identification of tangent spaces  $T_{(\calE,\varphi)}\calT_V=T_{\calF}\Prym(S_{q,V})=H^1(S_{q,V},\calO_{S_{q,V}})^-$, where $-$ denotes the anti-invariant part with respect to the involution changing the sheets. By definition of $S_{q,V}$ we have $(\pi_{q,V})_*\calO_{S_{q,V}}=\calO_\Sigma \oplus K^{-1}(V)$. Therefore, the anti-invariant part of the first cohomology can be identified as $H^1(S_{q,V},\calO_{S_{q,V}})^-=H^1(\Sigma,\pi_{V*}\calO_{q,V})^-= H^1(\Sigma,K^{-1}(V))$.
\end{proof}

\subsection{Hitchin's subintegrable system and collections of minimal strata}\label{ssec:Hitchin_subint}
In this subsection, we will review the subintegrable systems introduced by Hitchin in \cite{hitchin2019critical}. A subintegrable system is a subset of a complex symplectic manifold $M$ together with a half-dimensional set of Poisson commuting functions, such that the fibers are Lagrangian tori. Each subset $\calM_r$ contains a subintegrable system as a dense subset. In terms of the stratification of the singular Hitchin fibers the Higgs bundles in the subintegrable system are in the closed stratum $\calT_{V_{max}}$.

Let $\calB^d \subset \calB$ denote the locus of the Hitchin base of quadratic differentials with exactly $d$ double zeros and the other zeros simple that are not a square of an abelian differential. For $q\in \calB^d$ the genus of the normalized spectral curve $\wt{S}_q$ is $4g-3-d$. Note that for a quadratic differential $q \in \calB^d$ the divisor that gives the double zeros is $D(q)=\lfloor \frac{\mathsf{div}(q)}{2} \rfloor=V_{\max}$.

Denote by
\[ \calM^d= \{ (\calE,\varphi) \in \calH^{-1}(\calB^d) \mid \mathsf{div}(\varphi)=V_{\max} \}
\] the collections of closed strata $\calT_{V_\max}$ over $\calB^d$. Then Hitchin showed:
\begin{theorem}[\cite{hitchin2019critical}]
  The restricted Hitchin map $\calM^d \to \calB^d$ defines an algebraically completely integrable system of dimension $6g-6-2d$.
\end{theorem}
As we saw in Corollary \ref{coro:max_stratum} the fibers of $\calM^d$ are torsors for the Prym variety of the normalized spectral curve $\Prym(\tilde{\pi}: \wt{S}_q \to \Sigma)$. By Lemma \ref{lemm:tangent_space_base_stratum} the tangent space to $\calB^d$ is $T_q\calB^d= H^0(\Sigma,K^2(-D(q)))$. Furthermore, by Lemma \ref{lemm:cohomology_spectral} the tangent space to the stratum $\calT_{V_{\max}}$ is $H^1(\Sigma,K^{-1}(D(q)))$. Hence, exact sequence \eqref{eq:ses:tangent_space} can be interpreted as tangential sequence of the restricted Hitchin map
\[ 0 \to T_\calL\calT_{D(q)} \to T_{(\mathcal{E},\varphi)}\calM^d \xrightarrow{d \calH}  T_q \calM^d \to 0.
\]
The duality of the tangent space of the fibers and the tangent space to the base is given by Serre duality $H^1(\Sigma,K^{-1}(D(q)))^\vee=H^0(\Sigma,K^2(-D(q)))$.

\begin{lemma}
  Let $g-1 \leq r \leq 3g-3$ and $d=3g-3-r$. Then $\calM^d \subset \calM_r$ is an open and dense subset.
\end{lemma}
\begin{proof}
  From Proposition \ref{prop:relation_rank_divisor} the subset relation is clear. Further, $\calM^d \subset \calM_r$ is open as locally no extra zeros of $q$ collide and $\calM^d=\calM_{r} \cap \calH^{-1}(\calB^d)$. We conclude denseness by showing that $\calB^d$ is dense in $\calH(\calM^r)$. Note that the dimension of $\calM_{r} \cap \calH^{-1}(q)$ is $r=3g-3-d$ independent of $q \in \calH(\calM_r)$.
  \par
  The vanishing divisor of the Higgs field is bounded by $V \leq D(q)=\lfloor \frac{\mathsf{div}(q)}{2} \rfloor$. Hence, fixing $\deg(V)$ gives a lower bound on $\deg(D)$. We saw in Lemma \ref{lemm:tangent_space_base_stratum} that for a given partition $\bp$ the tangent space to $\calB_\bp$ is contained in $H^0(\Sigma,K^2(D_0(q)))$. It is easy to see that under the lower bound on $\deg(D)$ the generic dimension of $H^0(\Sigma,K^2(D_0(q)))\leq 3g-3-d$ with equality if and only if $\bp=(1,\dots,1,2,\dots,2)$.
\end{proof}

More generally, we will be interested in the collection of closed strata. Let
\[ \calM_{\bp,V_{\max}} =\{ (\calE,\varphi) \in \calH^{-1}(\calB_{\bp}) \mid V(\varphi)=\lfloor \frac{\mathsf{div}(\det\varphi)}{2} \rfloor \}.
\] By Corollary \ref{coro:max_stratum} $(\calE,\varphi)=\tilde{\pi}_*(\calL,\tilde{\lambda})$ for all $(\calE,\varphi) \in \calM_{\bp,V_{\max}}$. Hence, $\calM_{\bp,V_{\max}}$ is a torsor for the universal Prym variety $\Prym(\tilde{\pi}: \widetilde{S} \to \Sigma \times \calB_{\bp})$. As a consequence of Proposition \ref{prop:diff_of_Hit} we obtain:
\begin{lemma}\label{lemm:exact_sequence_M_Vmax}
  Let $\bp$ contain at least one odd integer or $\deg(V_\max)\equiv 0 \pmod 2$. Let $(\calE,\varphi) \in \calM_{\bp,V_{\max}}$ and $\calL$ the corresponding line bundle on $\wt{S}_{q}$. There is an exact sequence of tangent spaces
  \[ \begin{tikzcd}
      T_{\calL}\Prym(\wt{S}_q)=H^1(\Sigma,K^{-1}(V_{\max})) \ar[r,hook]  &T_{(\calE,\varphi)}\calM_{\bp,V_{\max}} \ar[r,two heads, "d \calH"] &T_q\calB_{\bp}=H^0(\Sigma,K^2(-D_0)).
      \end{tikzcd}
  \]
\end{lemma}
\begin{proof}
  From Proposition \ref{prop:diff_of_Hit} and the computation of the tangent space to $\calB_{\bp}$ in Lemma \ref{lemm:tangent_space_base_stratum} we obtain the commutative diagram
  \[
  \begin{tikzcd}
     H^1(\Sigma,K^{-1}(V_{\text{max}})) \ar[r,hook]\ar[d,"="] &  T_{(\calE,\varphi)}\calM_{\bp,V_{\max}} \ar[r,two heads,"d \calH"]\ar[d,hook] & H^0(\Sigma,K^2(-D_0(q))) \ar[d,hook,"\cdot s_{D_0(q)-V_\max}"] \\
    H^1(\Sigma,K^{-1}(V_{\text{max}})) \ar[r,hook] &  T_{(\calE,\varphi)}\calM \ar[r,two heads] & H^0(\Sigma,K^2(-V_{\text{max}})) \oplus H^0(\Sigma,\calO_{V_\text{max}})\,.
  \end{tikzcd}
  \]
  To obtain the surjectivity of the restricted Hitchin map we employ Lemma \ref{lem:exist_sigma_invariant_pointwise} proving the existence of $\sigma$-invariant line bundles in the corresponding torsor over the Prym variety. This requires the extra condition on $q$.
\end{proof}

\section{The Limiting configuration and deformations}
In this section, we introduce the construction of limiting configurations and the associated exponential decay results \cite{mazzeo2012limiting, Mochizukiasymptotic, mochizuki2023asymptotic}. Furthermore, we study the deformation of the moduli space of limiting configurations on singular fibers, extending the construction in \cite{mazzeo2012limiting} from smooth to singular fibers.

\subsection{The limiting configuration equations}

Let $E$ be a trivial, smooth, rank 2 vector bundle over a Riemann surface $\Sigma$, and let $H_0$ be a background Hermitian metric on $E$. Let $\mathfrak{sl}(E)$ (resp.\ $\mathfrak{su}(E)$) denote the bundle of traceless (resp. traceless skew-Hermitian) endomorphisms of $E$.

Let $Z$ be a finite set of distinct points on $\Sigma$. Consider a smooth unitary connection $A$ on $E|_{\Sigma \setminus Z}$ and a smooth field $\phi = \vp + \vp^{\da} \in \Omega^1(i\su(E))|_{\Sigma \setminus Z}$, where $\vp = \phi^{(1,0)}$ and $\vp^{\da}$ is the Hermitian adjoint of $\vp$ with respect to $H_0$. The \emph{decoupled Hitchin equations} on $\Sigma \setminus Z$ are given by:
\begin{equation} \label{eq_decoupled_Hitchin_equation}
	\begin{split}
		F_A = 0\ ,\quad [\vp, \vp^{\da}] = 0\ ,\quad \bar{\pa}_A \vp = 0\ .
	\end{split}
\end{equation}
The set $Z$ is called the \emph{singular set} of the limiting configuration, and we denote its complement by $\Sigma^{\circ} = \Sigma \setminus Z$. Since general solutions to \eqref{eq_decoupled_Hitchin_equation} may exhibit singular behavior near $Z$, we impose the following additional condition:

\begin{definition}
	A solution $(A, \phi)$ to the decoupled Hitchin equations on $\Sigma^{\circ}$ is called a \emph{limiting configuration} if $\phi \neq 0$, and $|\phi|$ extends to a continuous function on $\Sigma$ with $|\phi|^{-1}(0) = Z$.
\end{definition}

This extra condition on the continuity of $|\phi|$ ensures that $\det(\vp)$ extends to a holomorphic quadratic differential $q = \det(\vp)$ on $\Sigma$ such that $Z = q^{-1}(0)$. Taubes \cite{Taubes20133manifoldcompactness} showed that any solution to the decoupled Hitchin equations arising as the limit of a sequence of solutions to the full Hitchin equations is a limiting configuration.

Two limiting configurations $(A_1, \phi_1 = \vp_1 + \vp_1^{\da})$ and $(A_2, \phi_2 = \vp_2 + \vp_2^{\da})$ are said to be \emph{gauge equivalent} if $\det(\vp_1) = \det(\vp_2)$ on $\Sigma$, and there exists a smooth unitary gauge transformation $g \in \Gamma(\Sigma^{\circ}, \SU(E))$ such that $(A_1, \phi_1)\cdot g = (A_2, \phi_2)$ on $\Sigma^{\circ}$.

We define the moduli space of limiting configurations as
\[
\MMH^{\LC} = \left\{ \text{limiting configurations of } \eqref{eq_decoupled_Hitchin_equation} \right\} / \sim,
\]
equipped with a Hitchin map
\[
\msHi: \MMH^{\LC} \to \MB = H^0(K^2),
\]
and we aim to understand the fibers $\msHi^{-1}(q)$ for each $q \in H^0(K^2)$.

Note that the smooth vector bundle $E$ with Hermitian metric $H_0$ is defined over the entire surface $\Sigma$. Given a solution $(A, \phi)$ to \eqref{eq_decoupled_Hitchin_equation}, we obtain a Higgs bundle $(\ME, \vp)$ that is defined only on $\Sigma^{\circ}$. The Hermitian metric $H_0$ can be used to determine a holomorphic extension of $\ME$ to all of $\Sigma$, following the method in \cite[Theorem~2]{simpson1990harmonic}. Our situation is a special case of a tame harmonic bundle, where the analysis simplifies significantly because $|\phi|$ is bounded.

We define a filtered structure as follows: for each real number $\al$, let $r$ denote the distance to $Z$ on $\Sigma$. We define the sheaves
\[
\ME_{\al} := \left\{ s \in \Gamma(\ME) \,\middle|\, |s|_{H_0} \leq C r^{\al - \ep} \text{ for all } \ep > 0 \right\}.
\]
Since $|\vp|$ is bounded, $\vp$ also induces a map $\vp_{\al}: \ME_{\al} \to \ME_{\al} \otimes K.$ The Cornalba–Griffiths condition, see \cite[Proposition~3.1]{simpson1990harmonic}, is automatically satisfied in our setting, so each $\ME_{\al}$ is coherent.

Thus, we obtain a filtered Higgs bundle $(\ME_{\al}, \vp_{\al})$. At the level $\al = 0$, we canonically associate a Higgs bundle $(\ME_0, \vp_0)$ defined over the entire Riemann surface $\Sigma$. For notational convenience, we sometimes drop the subscript and refer to $(\ME, \vp)$ as the canonically associated Higgs bundle of the limiting configuration.

\subsection{The construction of the limiting configuration metric}\label{subsec:lim_config}

In this subsection, we introduce the construction of the limiting configuration and the convergence theory developed in \cite{mazzeo2012limiting, Mochizukiasymptotic}; see also \cite{he2023algebraiccompactification} for further details.

Now, we recall the Mochizuki's construction from \cite{Mochizukiasymptotic}. Let $(\mathcal{E}, \varphi)$ be a Higgs bundle on $\Sigma$ with $\det \varphi = q$. Consider the pullback
$(\tME, \tvp) := (\tilde{\pi}^{\ast} \mathcal{E}, \tilde{\pi}^{\ast} \varphi)$ to $\widetilde{S}_q$. 
Let $\wt{\calL}_0=\mathsf{Ab}(\ME,\varphi)\in \calP_V$ and define the $\sigma$-invariant Higgs bundle on $\wt{S}_q$
\[
\tME_0 = \wt{\calL}_0 \oplus \sigma^{\ast} \wt{\calL}_0, \quad \tvp_0 = \mathrm{diag}(\tilde{\omega}, -\tilde{\omega}).
\]
There is a natural inclusion $\iota: (\tME, \tvp) \to (\tME_0, \tvp_0)$ defined as the dual of the inclusion of the eigen line bundles \eqref{equ:eigenline_bundle}. It is an isomorphism over $\tS_q \setminus \tZ$.  Recall that for any real coefficient divisor $V = \sum_{x \in Z} v_x x$, there is a canonical weight function
\[
\chi_V(x) :=
\begin{cases}
	v_x, & x \in \supp V\ ;\\
	0,   & x \notin \supp V\ .
\end{cases}
\]

A line bundle $\ML$ together with a weight function $\chi_V$ defines a parabolic line bundle $(\ML, \chi_V)$. We denote by $\MF_*(\ML, \chi_V)$ the associated filtered line bundle; see \cite[Section 3]{he2023algebraiccompactification} for an explanation.

 \subsection{Mochizuki's convergence theorem}

We recall Mochizuki’s construction of the limiting configuration metric \cite[Sections 4.2.1 and 4.3.2]{Mochizukiasymptotic} and the convergence theorem.

Let $V$ be the Higgs-divisor of $(\ME, \vp)$.

\begin{definition}
	Let $\msV$ be a real coefficient divisor on $\tS_q$ with $\msV \leq \Lambda$. The divisor $\msV$ and its weight function $\chi_{\msV}$ are said to be \emph{compatible} if
	\[
	\msV + \sigma^* \msV = V - \Lambda.
	\]
	The weight function $\chi_{\msV}$ is called \emph{canonical} if $\msV = \frac{1}{2}(V - \Lambda)$.
\end{definition}

We now explain how to obtain a limiting configuration from a canonical weight function, as introduced in \cite[Sections 4.2.1 and 4.3.2]{Mochizukiasymptotic}. Given a compatible weight function $\chi_{\msV}$, we define the filtered bundle
\[
\MF_{*}(\wt{\calL}_0, \chi_{\msV}),
\]
and let $\tlh$ be the harmonic metric defined by the associated parabolic structure. This metric is well-defined up to a positive multiplicative constant. To fix this constant, we assume that
\[
\sigma^{\ast} \tlh \otimes \tlh = 1,
\]
which uniquely determines $\tlh$. We then define the metric $\tH = \mathrm{diag}(\tlh, \sigma^{\ast} \tlh)$ on $\tME_0$, satisfying $\det(\tH) = 1$.

For the resulting harmonic bundle $(\tME_0, \vp_0, \tH)$, let $\widetilde{\nabla}$ be the unitary connection determined by $\tH$. Since $\tH$ is diagonal, it follows that over $\tS_q \setminus \tZ$, we have $F_{\widetilde{\nabla}} = 0$ and $[\vp_0, {\vp_0}_{\tH}^{\da}] = 0$. Furthermore, as the inclusion $\iota: (\tME, \tvp) \to (\tME_0, \tvp_0)$ is an isomorphism over $\tS_q \setminus \tZ$, the metric $\tH$ also defines a metric on $(\tME, \tvp)$ over $\tS_q \setminus \tZ$.

For any $\tx \in \tS_q \setminus \tZ$ with $x := p(\tx)$, we have natural isomorphisms
\[
(\tME_0, \tvp_0)|_{\sigma(\tx)} \cong
(\tME_0, \tvp_0)|_{\tx} \cong
(\tME, \tvp)|_{\tx} \cong
(\ME, \vp)|_x\ .
\]
Therefore, $\tH$ induces a metric $H^{\LC}$ on $\Sigma \setminus Z$, which may be regarded as the pushforward of $\tlh$. In \cite[Theorem 5.2]{horn2022sltypesingularfibers}, this pushforward metric is written explicitly in local coordinates.

Thus, starting from any Higgs bundle $(\ME, \vp)$ and a compatible weight function, we obtain a limiting configuration metric $H^{\Lim}$, which defines
\[
(\na^{\LC}, \phi^{\LC} = \vp + \vp^{\da}_{H^{\Lim}}) \in \MMH^{\LC}\ .
\]

When $q$ has only odd zeros, we have $\sigma^* \msV = \msV$, and a weight function $\chi_{\msV}$ is compatible if and only if $\msV = \frac{1}{2}(V - \Lambda)$. In this case, $\chi_{\msV}$ is uniquely given by the canonical weight function.

When $q$ has even zeros and $\chi_{\msV}$ is a compatible weight function, then for any pair of even zeros $\tp^\pm$ satisfying $\sigma(\tp^+) = \tp^-$, the condition $\msV + \sigma^* \msV = V - \Lambda$ allows for a one-parameter family of freedom over each such pair. Hence, for a fixed pair $(q, V)$, the set of compatible weight functions is parameterized by $\mbR^{r_{\even}}$.

In summary, we conclude the following:

\begin{proposition}
	\label{prop_extra_dimension}
	For any Higgs bundle $(\ME, \vp) \in \MM_{q, V}$, there exists an $r_{\even}$-dimensional family of limiting configurations parameterized by compatible weight functions.
\end{proposition}

Define
\[
\MF_{*, \can} := \MF_{*} \left( \wt{\calL}_0, \chi_{\frac{1}{2}(V - \Lambda)} \right),
\]
which is the canonical filtered bundle. We denote by $(\na_{\can}^{\LC}, \phi_{\can}^{\LC})$ the corresponding limiting configuration, referred to as the \emph{canonical limiting configuration}.

We define the \emph{analytic Mochizuki map} $\UpMoc$ as
\begin{equation} \label{eq_analytic_moc_irreducible}
	\UpMoc: \MM_q \longrightarrow \MMH^{\LC}, \quad
	[(\ME, \vp)] \mapsto [(\na_{\can}^{\LC}, \phi_{\can}^{\LC})].
\end{equation}

The following theorem of Mochizuki and Mochizuki-Szab\'o explains why the weight $\chi_{\frac{1}{2}(V - \Lambda)}$ is canonical:

\begin{theorem}[{\cite{Mochizukiasymptotic}, \cite[Theorem 1.7]{mochizuki2023asymptotic}}]
	\label{thm_moc_convergence_irreducible_case}
	Let $(\ME, t\vp)$ be a family of stable Higgs bundles with scaling parameter $t \to \infty$. Then this family has a unique limiting configuration $\UpMoc(\ME, \vp)$ such that for any compact set $K \subset \Sigma \setminus Z$,
	\[
	\lim_{t \to \infty} \left\| (\na_t, \phi_t) - \UpMoc(\ME, \vp) \right\|_{\MC^l(K)} = 0.
	\]
	Moreover, if $(\ME, \vp)\in \MT_{V_{\max}}$, then there exist positive constants $C_{l,K}$ and $C'_{l,K}$, independent of $t$, such that
	\[
	\left\| (\na_t, \phi_t) - \UpMoc(\ME, \vp) \right\|_{\MC^l(K)} \leq C_{l,K} e^{-C'_{l,K} t},
	\]
	where $\MC^l$ denotes the $l$-th derivative continuous norm.
\end{theorem}

\subsection{The deformation of the limiting configurations}

The deformation theory of the moduli space of limiting configurations over smooth fibers was studied in \cite[Section 3]{mazzeo2016ends}. In this subsection, we extend those results to study deformations over singular fibers.

Using the filtered structure, for any solution $(A, \phi)$ to the decoupled Hitchin equations, we denote by $(\ME, \vp)$ the canonically associated Higgs bundle. We denote the zero divisor of the quadratic differential by $\Div(q) = \sum_{x \in Z} m_x x$ and the Higgs divisor associated to $(\calE,\varphi)$ by $V = \sum_{x \in Z} v_x x$. Then, by Theorem \ref{thm:LimConfig}, near any $x \in Z$, the Higgs bundle has the following canonical local form:
\begin{equation}
	\label{eq_canonicalform_Higgsbundle}
	\vp = \begin{pmatrix}
		0 & z^{v_x} \\
		z^{m_x - v_x} & 0
	\end{pmatrix} dz.
\end{equation}

We now state a generalization of \cite[Lemma 3.1]{mazzeo2016ends}:

\begin{lemma}
	\label{lemma_normalform_Higgsbundle}
	Let $q \in \MB_\Sigma$, and let $(\ME_1, \vp_1)$ and $(\ME_2, \vp_2)$ be two canonically associated Higgs bundles lying in the same stratum defined by $D$. Then there exists a unitary transformation $g$ on $\Sigma^{\circ}:=\Sigma\backslash Z(q)$ such that $\vp_1^g = \vp_2$.
\end{lemma}

\begin{proof}
	Since $\Sigma^{\circ}$ is homotopy equivalent to a union of circles, any complex vector bundle over $\Sigma^{\circ}$ is topologically trivial, and any bundle with connected fiber admits a global section over $\Sigma^{\circ}$. Over $\Sigma^{\circ}$, each $\vp_i$ can be viewed as a map $\vp_i: \Sigma^{\circ} \to \mathfrak{sl}_2(\CC)$.

	By the local normal form \eqref{eq_canonicalform_Higgsbundle}, for any pair $\vp_1, \vp_2$, there exists a local unitary gauge transformation $g$ such that $g^{-1} \vp_1 g = \vp_2$.

	Consider the space
	\[
	\MC_{\vp_1,\vp_2} = \left\{ (p, g_p) \in \Sigma^{\circ} \times SU(2) \mid g_p^{-1} \vp_1 g_p = \vp_2 \right\} \to \Sigma^{\circ},
	\]
	which is a smooth fiber bundle with fiber $S^1$. Since $S^1$ is connected, $\MC_{\vp_1,\vp_2}$ admits a global section over $\Sigma^{\circ}$.
\end{proof}

Therefore, we may always assume a canonical form for the Higgs field in a limiting configuration and focus only on the metric and connection, analogous to the approach in \cite{mazzeo2016ends}.

Given $q$ and a stratum $V$, we denote by $\vp = \vp_{q, V}$ the canonical form of the Higgs field over $\Sigma^{\circ}$. The infinitesimal complex stabilizer of $\vp$ is the holomorphic line bundle
\[
L^{\mbC}_{q, V} := \left\{ \gamma \in \mathfrak{sl}(E) \,\middle|\, [\gamma, \vp] = 0 \right\}.
\]
By Lemma \ref{lemma_normalform_Higgsbundle}, the line bundle $L^{\mbC}_{q,V}$ depends only on the quadratic differential $q$ and the Higgs divisor $V$. Moreover, define
\[
L_{q,V} := L^{\mbC}_{q,V} \cap \su(E), \quad \text{and} \quad iL_{q,V}
\]
to be the real and imaginary parts, corresponding to the skew-Hermitian and Hermitian elements, respectively. By the Jacobi identity, $L_{q,V}$ is closed under the Lie bracket $[\cdot, \cdot]$. We emphasize that these line bundles are defined over $\Sigma^{\circ}$.

Now we state the following simple but useful lemma from \cite[Lemma 4.4]{mazzeo2016ends}:

\begin{lemma}{\cite[Lemma 4.4]{mazzeo2016ends}}
	Over $\Sigma^{\circ}$, if $A$ is a unitary connection such that $\bar{\pa}_A \vp = 0$, then
	\[
	F_A^{\perp} := F_A - \tfrac{1}{2} \Tr(F_A)\Id \in \Omega^2(L_{q,V}).
	\]
\end{lemma}

\begin{proof}
	The proof is identical to that of \cite[Lemma 4.4]{mazzeo2016ends}, since over $\Sigma^{\circ}$ the Higgs bundle can always be written in local coordinates as $\vp = \mathrm{diag}(\lambda, -\lambda)$.
\end{proof}

We now consider the deformation of limiting configurations. Let $(A_0, \phi_0 = \vp_0 + \vp_0^{\da})$ and $(A_1, \phi_1 = \vp_1 + \vp_1^{\da})$ be two solutions to \eqref{eq_decoupled_Hitchin_equation}. Then, by Lemma \ref{lemma_normalform_Higgsbundle}, up to a gauge transformation $g \in \Gamma(\Sigma^{\circ}, \SU(E))$, we may assume $\vp := \vp_0 = \vp_1$.

Write $A_1 = A_0 + \alpha$ with $\alpha \in \Omega^1(\su(E))$. Since $\bar{\pa}_{A_0} \vp = \bar{\pa}_{A_1} \vp = 0$, we obtain $[\alpha, \vp] = 0$. By \cite[Lemma 4.9]{mazzeo2016ends}, this implies $\alpha \in \Omega^1(L_{q,V})$, i.e., $\alpha$ takes values in a line bundle, which further implies $[\alpha, \alpha] = 0$.

Moreover, since both $A_0$ and $A_1$ satisfy the decoupled Hitchin equations, we have
\begin{equation}
	[\alpha \wedge \vp] = 0, \quad d_{A_0} \alpha = 0.
\end{equation}
As $F_{A_0} = 0$, the operator $d_{A_0}$ defines a flat connection on $L_{q,V}$, and therefore $\alpha$ determines a cohomology class in
\[
H^1(\Sigma^{\circ}; L_{q,V}).
\]

We define the ungauged vertical deformation space at $(A_0, \phi_0)$ as
\begin{equation}
	Z^1(\Sigma^{\circ}; L_{q,V}) := \left\{ \alpha \in \Omega^1(\Sigma^{\circ}, L_{q,V}) \,\middle|\, d_{A_0} \alpha = 0 \right\}.
\end{equation}

Let $\Stab_{\phi_0}$ denote the group of unitary gauge transformations on $\Sigma^{\circ}$ that stabilize $\phi_0$. For $g \in \Stab_{\phi_0}$, if $g = \exp(\gamma)$ for some $\gamma \in \Omega^0(\Sigma^{\circ}, L_{q,V})$, then the action of $g$ on $A_1 = A_0 + \alpha$ is given by
\[
  A_1\cdot g = g^{-1} d_{A_0} g + g^{-1} \alpha g = \alpha + d_{A_0} \gamma.
\]
We define
\begin{equation}
	B^1(\Sigma^{\circ}; L_{q,V}) := \left\{ d_{A_0} \gamma \,\middle|\, \gamma \in \Omega^0(\Sigma^{\circ}, L_{q,V}) \right\}.
\end{equation}
Hence, the infinitesimal vertical deformation space is
\begin{equation}
	H^1(\Sigma^{\circ}; L_{q,V}) := Z^1(\Sigma^{\circ}; L_{q,V}) / B^1(\Sigma^{\circ}; L_{q,V}).
\end{equation}

As $g=\exp(\gamma)$, the remaining elements in $\Stab_{\phi_0}$ form an integral lattice $H^1_{\mbZ}(\Sigma^{\circ},L_{q,V})$ under the exponential map. Therefore, we conclude:

\begin{proposition}
	The moduli space of limiting configurations for fixed $q$ and Higgs divisor $V$ can be identified with
	\[
	\msH_V^{-1}(q) = H^1(\Sigma^{\circ}; L_{q,V}) / H^1_{\mbZ}(\Sigma^{\circ}; L_{q,V}).
	\]
\end{proposition}

\begin{proposition}
	Let $(\ME, \varphi) \in \calM$ be a Higgs bundle with Higgs divisor $V$. Then
	\[
	L^{\mbC}_{q,V} = K^{-1}(V).
	\]
	Moreover, if $k$ denotes the number of zeros of $q$ and $q$ is irreducible, then
	\begin{equation}
		\dim_{\mbR} H^1(\Sigma^{\circ}; L_{q,V}) = k + 2g - 2.
	\end{equation}
\end{proposition}

\begin{proof}
	By definition,
	\[
	L^{\mbC}_{q,V} = \ker\left([\varphi, \cdot]: \slf(E) \to \slf(E) \otimes K \right).
	\]
	For generic $x \in \Sigma \setminus Z$, the Higgs field is regular semisimple, and the adjoint action $[\vp,\cdot]$ has one-dimensional kernel. Hence, $L^{\mbC}_{q,V}$ is a line bundle.
	Let $s_V$ be the canonical section of $\MO(V)$. Then the image of the map $\frac{\varphi}{s_V}: K^{-1}(V) \to \slf(E)$
	lies in the kernel. Conversely, if a section of $\slf(E)$ commutes with $\varphi$, it defines a $\tilde{\pi}_* \calO_{\tS_q}$-module homomorphism of $E$. We have $\tilde{\pi}_* \calO_{\tS_q} = \calO_\Sigma \oplus K^{-1}(V),$
	where the splitting corresponds to the symmetric and anti-symmetric parts under the involution $\sigma$. Thus,
	\[
	\tilde{\pi}_* \calO_{\tS_q} \cap \slf(E) = K^{-1}(V) = L^{\mbC}_{q,V}.
	\]
	For the dimension counting, we apply \cite[Lemma 4.10]{mazzeo2016ends}. Since $L_{q,V}$ is a real line bundle, its Euler characteristic is given by
	\begin{equation}\label{eq:Euler_char_LqV}
		\chi(\Sigma^{\circ}; L_{q,V}) = \chi(\Sigma^{\circ}) = 2 - 2g - k.
	\end{equation}
	As $q$ is irreducible, there are no parallel sections of $L_{q,V}$, so $H^0(\Sigma^{\circ}; L_{q,V}) = 0$.	Let $B_{\epsilon}$ be a sufficiently small tubular neighborhood of $Z$, and set $M := \Sigma \setminus B_{\epsilon}$. Then
	\[
	H^2(\Sigma^{\circ}; L_{q,V}) = H^2(M; L_{q,V}) \cong H^0(M, \partial M; L_{q,V}) = 0,
	\]
	since $L_{q,V}$ admits no parallel sections. Therefore, by \eqref{eq:Euler_char_LqV},
	\[
	\dim_{\mbR} H^1(\Sigma^{\circ}; L_{q,V}) = k + 2g - 2.
	\]
\end{proof}

We now explain the reason for the extra dimension in the space of limiting configurations compared to the moduli space of Higgs bundles. Let $r_{\even}$ and $r_{\odd}$ denote the number of even and odd zeros of the quadratic differential $q$, respectively. Then we have
\[
\dim_{\mbR} H^1(\Sigma^{\circ}; L_{q,V}) = r_{\even} + r_{\odd} + 2g - 2.
\]
On the other hand, let $V$ be the Higgs divisor. Then the dimension of the corresponding Prym variety is
\[
\dim_{\mbR} \calP_V= 2g - 2 + r_{\odd}.
\]
Hence, we find that
\[
\dim_{\mbR} H^1(\Sigma^{\circ}; L_{q,V}) - \dim_{\mbR} \calP_V = r_{\even}.
\]
These additional $r_{\even}$ real dimensions are precisely the degrees of freedom arising from the choice of weight at each even zero, as described in Proposition \ref{prop_extra_dimension}.

\subsection{The canonical form of a Higgs bundle}
Using the notation $\Lambda$ and $V$ from earlier, we describe the canonical local form of a Higgs bundle in various coordinate frames.

\begin{theorem}[{\cite[Theorem 5.2]{horn2022sltypesingularfibers}, \cite[Section 4.3.2]{Mochizukiasymptotic}}]
	\label{thm:LimConfig}
	Let $(\ME, \varphi)$ be an $\mathrm{SL}_2(\mathbb{C})$ Higgs bundle with irreducible spectral curve, and let $q = \det(\vp)$. Write $\mathsf{div}(q) = \sum_{p \in Z} m_p p$, and let $V = \sum_{p \in Z} v_p p$ be the Higgs divisor of $\vp$, $u$ be the u-coordinate. For each $p \in \Lambda$, there exists a neighborhood $U_p$ with local coordinate $z$ and a holomorphic frame of $\ME|_{U_p}$ such that the Higgs field takes the form
	\begin{equation} \label{equ:localform_Higgs}
		\varphi = \begin{pmatrix}
			0 & z^{v_p} \\
			z^{m_p - v_p} & 0
		\end{pmatrix}.
	\end{equation}

	Moreover, for the limiting configuration metric, let $n_p := m_p - 2v_p$. Then:
	\begin{itemize}
		\item[(i)] If $m_p$ is odd, then the Hermitian metric for the limiting configuration in this coordinate is
		\begin{equation}
			\label{eq_localform_limitingconfiguration_odd_zero}
			H^{\LC} = \begin{pmatrix}
				g_1 |z|^{\frac{n_p}{2}} & g_2 z^{\frac{1 - n_p}{2}} |z|^{\frac{n_p}{2}} \\
				\bar{g}_2 \bar{z}^{\frac{1 - n_p}{2}} |z|^{\frac{n_p}{2}} & g_1 |z|^{-\frac{n_p}{2}}
			\end{pmatrix},
		\end{equation}
		where $g_1$ is a smooth, positive real function and $g_2$ is a complex function satisfying $g_1^2 - |g_2|^2 |z| = 1$. Moreover, $g_2=0$ if and only if $u=0$.

		\item[(ii)] If $m_p$ is even, then the limiting configuration metric is
		\begin{equation}
			\label{eq_localform_limitingconfiguration_even_zero}
			H^{\LC} = \begin{pmatrix}
				g_1 |z|^{\frac{n_p}{2}} & g_2 z^{-\frac{n_p}{2}} |z|^{\frac{n_p}{2}} \\
				\bar{g}_2 \bar{z}^{-\frac{n_p}{2}} |z|^{\frac{n_p}{2}} & g_1 |z|^{-\frac{n_p}{2}}
			\end{pmatrix},
		\end{equation}
		where $g_1, g_2$ are smooth, positive real functions satisfying $g_1^2 - g_2^2 = 1$. Moreover, $g_2=0$ if and only if $u_\pm=0$.
	\end{itemize}
\end{theorem}

\begin{proof}
	This theorem follows from the description of Higgs bundles in singular Hitchin fibers using the $u$-coordinate, as described in Section \ref{subsec:stratification}. We choose the parabolic structure on the line bundle $\calL \in \calP_V$ with canonical weight $\chi_{\frac{1}{2}(\Lambda - V)}$. Choosing frames adapted to the parabolic structure near $Z(\omega)$ uniquely determines a $u$-coordinate. The functions $g_1, g_2$ are then determined through this $u$-coordinate. For odd-order zeros of $q$, the condition $u = 0$ corresponds to $g_1 = 1$, $g_2 = 0$. Similarly, for even-order zeros, $u_{\pm} = 0$ gives $g_1 = 1$, $g_2 = 0$.
\end{proof}

The advantage of this description, compared to the original one in \cite{Mochizukiasymptotic}, is that the canonical form of the Higgs bundle is expressed directly over the original Riemann surface, rather than being pulled back to the normalization of the spectral curve. However, the trade-off is that the Hermitian metric is no longer diagonal.

\begin{definition}
	\label{def_admissiblel_Higgs_bundle}
	A Higgs bundle $(\ME, \vp)$ is called \emph{locally fiducial} if in the limiting configuration metric given in the canonical coordinates \eqref{eq_localform_limitingconfiguration_even_zero} and \eqref{eq_localform_limitingconfiguration_odd_zero}, we have $g_1 = 1$.
\end{definition}

According to Section \ref{subsec:stratification}, in each stratum, the holomorphic map \( \mathsf{Ab} : \MT_V \to \MP_V \) has fibers of the form
\[
\mathsf{Ab}^{-1}(\calL) \cong (\mathbb{C}^\ast)^{k_1} \times \mathbb{C}^{k_2}.
\]
This space parametrizes the Hecke transformations. A Higgs bundle is locally fiducial if and only if all the Hecke parameters vanish in the $u$-coordinates. In particular, each stratum contains locally fiducial Higgs bundles.

When \( V = V_{\max} \), the map \( \mathsf{Ab} \) is a bijection, and every Higgs bundle \( (\ME, \vp) \in \MT_{V_{\max}} \) is locally fiducial. By Corollary \ref{coro:max_stratum}, these bundles satisfy $(\ME, \vp) = \tpi_*(\calL, \tilde{\lambda})$ for some line bundle \( \calL \) on the normalization \( \tS_q \). In particular, the Higgs bundles studied by Mochizuki–Szabó \cite{mochizuki2023asymptotic} in Theorem \ref{thm_moc_convergence_irreducible_case} are special cases of locally fiducial Higgs bundles.

We also note that for \( V \neq V_{\max} \), although \( (\ME, \vp) \) may not arise as the pushforward of a line bundle on the normalization, it is related to such a pushforward by a Hecke transformation. Moreover, the Hecke transformation preserves the limiting configuration. More precisely, let \( (\ME_1, \vp_1) \) and \( (\ME_2, \vp_2) \) be two Higgs bundles, and let \( (\nabla_i^{\Lim}, \phi_i^{\Lim}) \) denote their canonical limiting configurations. If
\[
\mathsf{Ab}(\ME_1, \vp_1) = \mathsf{Ab}(\ME_2, \vp_2),
\]
then the limiting configurations agree:
\[
(\nabla_1^{\Lim}, \phi_1^{\Lim}) = (\nabla_2^{\Lim}, \phi_2^{\Lim}).
\]
We refer to \cite[Section 6]{he2023algebraiccompactification} for further discussions.

\begin{lemma}
	Let $(\calE,\vp) \in \calM$, such that $\det(\vp)\neq 0$ is a quadratic differential with no global square-root. Let $H^{\LC}$ denote the limiting configuration of Theorem \ref{thm:LimConfig}. Then $(\calE,\vp)$ is locally fiducial if and only if for each small enough neighborhood $U \subset \Sigma$ of a zero $p \in Z(q)$, there exists a symmetric bilinear pairing $C_U$ on $\calE$ compatible with $H^{\LC}$, such that $\vp$ is symmetric with respect $C_U$.
\end{lemma}
\begin{proof}
	Let $C_U$ denote a symmetric bilinear pairing on $\calE_U$. With respect to a local trivialization we can represent $C_U$ by a symmetric matrix $C$ with coefficients in $\calO_U$ and $H^{\LC}$ by a Hermitian, positive definite matrix $H$ with coefficients in $C^\infty_U$. Then the compatibility condition of local orthogonal and unitary structure is $\overline{H}^{-1}C=CH$. We can choose a frame such that $H=\diag(\vert z \vert^k , \vert z \vert^{-k})$. If $k>0$, this implies $C$ has zero diagonal. Now, using the local form of $\vp$ from equation \eqref{equ:phi_u-coord_odd} and \eqref{equ:phi_u-coord_even} we conclude that $\vp$ is symmetric with respect to $C$, i.e.\ $\vp^{\top}C=C\vp$, if and only if the $u$-coordinate is zero. If $k=m_p-2v_p=0$ the existence of a compatible $C_U$ is clear and the local fiduciality condition at $p$ is vacuous (no Hecke parameters at $p$).
\end{proof}

\section{Exponential convergence to the limiting configuration}
Let $(\ME,\vp)$ be a Higgs bundle, $(A_t,\phi_t)$ be the corresponding flat connection for $(\ME,t\vp)$, and $(A_{\infty},\phi_{\infty})$ be the limiting configuration. In this section, we will extend the exponential convergence result of Mochizuki-Szabó \cite{mochizuki2023asymptotic} to the case when $(\ME,\vp)$ is a locally fiducial Higgs bundle. In particular, we establish the following result.

\begin{theorem}\label{thm:exp_conv}
	Suppose $(\ME,\vp)$ is a locally fiducial Higgs bundle and $q = \det(\vp)$ has at least one zero of odd order, then over any compact set $K\subset \Sigma\setminus Z$, there exist constants $C_{l,K},C'_{l,K}$ depending on $l$ and $K$ but independent of $t$ such that
	\begin{equation}
		\|(A_t,\phi_t)-(A_{\infty},\phi_{\infty})\|_{\MC^{l}(K)}\leq C_{l,K}e^{-C'_{l,K}t}.
	\end{equation}
\end{theorem}
Here the $\MC^l$-norm is taken with respect to a fixed Hermitian metric $H_0$ on $\ME$ and a fixed Kähler metric $g_{\Sigma}$ on $\Sigma$. Throughout this section, we assume that $q$ has at least one odd zero, and let $C,\delta,C_{l,K},C'_{l,K}$ be constants independent of $t$ but may vary from line to line.

\subsection{Model fiducial solutions}\label{subsec:ModelSolution}
By Theorem \ref{thm:LimConfig}, we have a canonical local form for every locally fiducial Higgs bundle and the corresponding limiting configuration. In this subsection, we will consider the model fiducial solutions, analogous to \cite{mazzeo2012limiting}.

Over $\mbC$ with coordinate $z=x+iy=re^{i\theta}$, we consider the Higgs bundle
\begin{equation}
	\ME=\MO\oplus \MO,\;\vp=\begin{pmatrix}
		0 & z^d\\
		z^{m-d} & 0
	\end{pmatrix},
\end{equation}
which can be extended to be a wild Higgs bundle over $\mathbb{P}^1$. Moreover, let $H$ be the Hermitian metric which solves the Hitchin equation. Following \cite[Corollary 2.11]{collier2014asymptotics}, we may assume $H=\diag(h,h^{-1})$, where $h$ is a Hermitian metric on the trivial bundle $\MO$, which is a positive function.

The Hitchin equation for the harmonic metric $H_t$ of the Higgs bundle $(\ME,t\vp)$ can be written as
\begin{equation}
	\Lam \bar{\pa}(H_t^{-1}\pa H_t)+t^2[\vp,H_t^{-1}\vp^{\dagger}H_t]=0.
\end{equation}
We may write $H_t=\diag(e^{u_t},e^{-u_t})$, then the equation becomes
\begin{equation}
	\Delta u_t-t^2(r^{2d}e^{2u_t}-r^{2m-2d}e^{-2u_t})=0,
\end{equation}
where $\Delta=\frac14 (\pa_x^2+\pa_y^2)$. We will assume that $u_t$ is a function of $r$.

Let \( u_{\infty} = \frac{1}{2}(m - 2d) \log r \) and \( v_t = u_t - u_{\infty} \). Then, \( v_t(r) \) satisfies the following equation:
\begin{equation}\label{eqn:vtODE}
	v_t''+r^{-1}v_t'=8t^2r^m\sinh(2v_t).
\end{equation}
Define \( \rho = \frac{4t r^{\frac{m}{2} + 1}}{1 + \frac{m}{2}} \). We can then write \( v_t(r) = \psi(\rho) \), where \( \psi(\rho) \) satisfies the following Painlevé type III equation:
\[
\psi'' + \frac{\psi'}{\rho} = \frac{1}{2} \sinh(2\psi),
\]
which has been precisely discussed in \cite[Section 3.2]{mazzeo2012limiting} and \cite{fredrickson2018generic}.

This ODE has a unique solution with the following asymptotics:
\[
\psi(\rho) \sim \frac{1}{\pi} K_0(\rho)\sim \rho^{-\frac12}e^{-\rho} \quad \text{as} \quad \rho \to \infty, \quad \psi(\rho) \sim -\frac{m - 2d}{m + 2} \log \rho \quad \text{as} \quad \rho \to 0,
\]
where $K_0(\rho)$ is the Macdonald function of order $0$ and \( \psi(\rho) \) is monotonically decreasing and strictly positive.

From this we can conclude
\begin{lemma}
	\label{lem:Painlev_model}
	Let $K\subset \mbC\setminus \{0\}$ be a compact set. Then $\|u_t-u_{\infty}\|_{\MC^l(K)}\LS e^{-C_{K,l}t}$, where $C_{K,l}$ is a constant depending only on $K$ and $l$.
\end{lemma}
\begin{corollary}
	For each $t\in (0,\infty)$, let $H_t$ be the harmonic metric of the model Higgs bundle $(\ME,t\vp)$, then $H_t$ is rotationally symmetric. Moreover, over any compact set $K\subset \mbC\setminus \{0\}$, $H_t$ converges smoothly exponentially to $H_{\infty}=\mathrm{diag}(r^{\frac12(m-2d)}, r^{-\frac12(m-2d)})$ as $t\to \infty$.
\end{corollary}

We may express the model solutions in the unitary gauge. Let $H_t=g_t^2$ with $g_t=\diag(h_t^{\frac12},h_t^{-\frac12})$, we write $h_t=r^{\frac12(m-2d)}e^{v_t}$. Then in the unitary gauge, we have $$\phi_t=\phi_{t}^{1,0}dz+\phi_{t}^{0,1}d\bz,\;A_t^{1,0}=g_t^{-1}\pa g_t,$$ and by a straightforward computation, we obtain
\begin{equation}\label{eqn:HiggsUnitary_eq}
	A_t=\Big(\frac{1}{8}(m-2d)+\frac{1}{4}r\pa_rv_t\Big)\begin{pmatrix}
		1 & 0\\
          0 & -1
	\end{pmatrix}\Big(\frac{dz}{z}-\frac{d\bz}{\bz}\Big),\;\phi_t^{1,0}=\begin{pmatrix}
	0 & r^{\frac{m}{2}-d}z^{d}e^{v_t}\\
	r^{-(\frac{m}{2}-d)}z^{m-d}e^{-v_t} & 0
\end{pmatrix}.
\end{equation}
Moreover, when $t\to \infty$, we obtain $(A_t,\phi_t)\to (A_{\infty},\phi_{\infty})$ with
\begin{equation}\label{eqn:InfHiggsUnitary_eq}
	A_{\infty}=\frac{1}{8}(m-2d)\begin{pmatrix}
		1 & 0\\
		0 & -1
	\end{pmatrix}\Big(\frac{dz}{z}-\frac{d\bz}{\bz}\Big),\;\phi_{\infty}^{1,0}=\begin{pmatrix}
		0 & r^{\frac{m}{2}-d}z^{d}\\
		r^{-(\frac{m}{2}-d)}z^{m-d}& 0
	\end{pmatrix},
\end{equation}
and $(A_{\infty},\phi_{\infty})$ is a limiting configuration.

\subsection{Approximate solution}\label{subsec:app_sol}
For $(\ME,\vp,H)$, we write $H=H_0h$, where $H_0$ is a fixed Hermitian metric on $\ME$ and $h$ is a smooth section of $\End(\ME)$ which is positive definite and Hermitian with respect to $H_0$. In other words, $H(u,v)=H_0(h(u),v)$ for any sections $u,v$ of $\ME$. We consider the following map
\begin{equation}
	\bF((\ME,\vp,H)):=h^{\frac12}(F_{A_H}+[\vp,\vp^{\da}_H])h^{-\frac12},
\end{equation}
which is skew-Hermitian with respect to $H_0$.

Now, we will explicitly construct the approximate solution using the local model solutions in the previous subsection. Let $\Sigma_p(\ep)$ be a disk centered at $p\in Z$ with radius $\ep$. By Theorem \ref{thm:LimConfig}, there exists $\ep_0$ such that for all $p\in Z$, there is a local holomorphic coordinate $z=re^{i\theta}$ and a holomorphic frame of $\ME$ over $\Sigma_p(\ep_0)$ where
\begin{equation}
	\ME=\MO\oplus \MO,\;\vp=\begin{pmatrix}
		0 & z^d\\
		z^{m-d} & 0
	\end{pmatrix} dz,\;H_{\infty}=\begin{pmatrix}
	r^{\frac12(m-2d)} & 0\\
	0 & r^{-\frac12(m-2d)}
\end{pmatrix}.
\end{equation}
By the conformal invariance of the Hitchin equation $\bF((\ME,\vp,H))=0$, we may assume that $\ep_0=2$.

Let $\chi$ be a smooth cut-off function such that $\chi=0$ on $\Sigma\setminus \Sigma_p(1)$, $\chi$ depends only on $r$ on $\Sigma_p(1)$, and $\chi=1$ on $\Sigma_p(1/2)$. The approximate Hermitian metric $H_t^{\app}$ is defined by
\[
  H_t^{\app}=\left\{
    \begin{array}{ll}
      \mathrm{diag}( r^{\frac12(m-2d)}e^{\chi v_t}, r^{-\frac12(m-2d)}e^{-\chi v_t})& \text{ on }\Sigma_p(1),\\
      H_{\infty}& \text{ on }\Sigma\setminus \bigcup_{p\in Z}\Sigma_p(1).
    \end{array}
\right.
\]
Note that $H_t^{\app}$ is exactly the model solution on $\Sigma_p(1/2)$, hence $H_t^{\app}$ is smooth.

\begin{proposition}
	Suppose $(\ME,\vp)$ is locally fiducial, then there exists a family of approximate metrics $H_t^{\app}$ such that \begin{itemize}
		\item [(i)]$\supp(\bF(\ME,t\vp,H_t^{\app}))\subset \Sigma_p(1)\setminus \Sigma_p(1/2)$,
		\item [(ii)] Let $K$ be a compact set in $\Sigma\setminus Z$, for $t$ sufficiently large, we have
		\begin{equation}
			\|\bF(\ME,t\vp,H_t^{\app})\|_{\MC^l(K)}\LS e^{-C_{K,l} t}.
		\end{equation}
	\end{itemize}
\end{proposition}
\begin{proof}
	(i) follows directly from the construction, since $H_{t}^{\app}$ satisfies the Hitchin equation over $\Sigma\setminus \Sigma_p(1)$ and $\Sigma_p(1/2)$. For (ii), since $h$ and $F_{H_t^{\app}}+t^2[\vp,\vp^{\da_{H_t^{\app}}}]$ are both diagonal in the above holomorphic frame over $\Sigma_p(1)$, we have
\begin{align*}
	\bF(\ME,t\vp,H_t^{\app})&=F_{H_t^{\app}}+t^2[\vp,\vp^{\da_{H_t^{\app}}}]=\diag(f_t(r),-f_t(r)),\text{ where}\\
	f_t(r)&:=\frac14\Big(\pa_r^2+\frac{1}{r}\pa_r\Big)(\chi v_t)-t^2r^m(e^{2\chi v_t}-e^{-2\chi v_t}).
\end{align*}
	Over a compact set $K\subset \Sigma\setminus Z$, we may assume $r\geq r_0$ and we have $$|t^2r^m(e^{2\chi v_t}-e^{-2\chi v_t})|_{\MC^l(K\cap \Sigma_p(1))}\LS e^{-C_{K,l}t},\;\Big|\frac14\Big(\pa_r^2+\frac{1}{r}\pa_r\Big)(\chi v_t)\Big|_{\MC^l(K\cap \Sigma_p(1))}\LS e^{-C_{K,l}t},$$
	which imply (ii).
\end{proof}

\subsection{Exponential convergence}

In this subsection, using the methods of \cite{mazzeo2016ends, fredrickson2018generic, fredrickson2022asymptotic}, we show that along a ray of Higgs bundles $(\bar{\partial}_E,t\varphi)$, $t\in[1,\infty)$, the approximate metric $H_t^{\app}$ converges exponentially to the harmonic metric $H_t$ as $t\to\infty$. Let $(A_t,\Phi_t)$ be the pair such that $(\bar{\partial}_E,t\varphi,H_t^\app)$ and $(\bar{\partial}_{A_t},t\Phi_t,H_0)$ are complex gauge equivalent, i.e.,
\[(\bar{\partial}_{A_t},t\Phi_t,H_0)=\left( \bar { \partial } _ { E } , t\varphi , H_t^\app \right)\cdot g_t = \left( g_t ^ { - 1 } \comp \bar { \partial } _ { E } \comp g_t , g_t ^ { - 1 }t \varphi g_t , H_t^\app \cdot g_t \right), \]
 where $(H_t^\app \cdot g_t) ( v , w ) = H_t^\app ( g_t v , g_t w )$. Let $\gamma_t$ be defined by $H_t=H_t^{\app}\cdot(g_te^{-\gamma_t}g_t^{-1})$, then $\bF_t^{\app}(\gamma_t)=0$, where
 \begin{equation}\label{eqn:Hitchin_Op}
   \bF_t^{\app}(\gamma):=-i\star\Big(F_{A_t^{\exp(\gamma)}}+t^2[e^{-\gamma}\Phi_t e^\gamma,e^\gamma{\Phi_t}^{\dagger}_{H_0}e^{-\gamma}]\Big),
 \end{equation}
$A_t^{\exp(\gamma)}$ is the Chern connection $D(e^{-\gamma}\bar{\partial}_{A_t}e^{\gamma},H_0)$. The linearization of $\bF_t^\app$ at $\gamma=0$ is
\begin{align}
  L_t\gamma&:=\Delta_{A_t}-i\star t^2 M_{\Phi_t}\gamma,\text{ where }\\
  \Delta_{A_t}&:=d_{A_t}^\ast d_{A_t}\gamma,~M_{\Phi_t}\gamma=[{\Phi_t}^\ast,[\Phi_t,\gamma]]-[\Phi_t,[{\Phi_t}^\ast,\gamma]].
\end{align}
Here the adjoints `$\ast$' are taken with respect to $H_0$.

Write $\gamma=\left(\begin{smallmatrix}
  u_0&u_1\\\bar{u}_1&-u_0
\end{smallmatrix}\right)$, then by \eqref{eqn:HiggsUnitary_eq}, near a zero of $\det\varphi$ we have \begin{multline*}
  L_t\gamma=\begin{pmatrix}
  \Delta_0 u_0  & \Delta_0 u_1 - 4F_t^m i \partial_\theta u_1 + 4(F_t^m)^2 u_1 \\
  \Delta_0 \bar{u}_1 + 4F_t^m i \partial_\theta \bar{u}_1 + 4(F_t^m)^2 \bar{u}_1 &  -\Delta_0 u_0 \end{pmatrix}\\
  +8r^m t^2 \begin{pmatrix} 2 \cosh(2 v_t) u_0  &  \cosh(2v_t) u_1 - e^{-(m-2d)i\theta} \bar{u}_1 \\
\cosh(2v_t )  \bar{u}_1 - e^{(m-2d)i\theta} u_1 &  -2 \cosh(2v_t) u_0 \end{pmatrix},
\end{multline*}
where $v_t$ solves \eqref{eqn:vtODE}, $F_t^m(r) = \frac12 (\frac12(m-2d) + r\partial_r v_t)$, and \[
\Delta_{k} := -\partial_r^2	u_1-r^{-1}\partial_r u_1-r^{-2}\left(\partial_{\theta} + ik\right)^2.
\]
Since $F_t^m$ vanishes and $r^m\cosh(2v_t)$ is regular at $r=0$, the indicial roots of $L_t$ are those of $\Delta_0$, i.e., $\mathbb{Z}$.

As $t\to\infty$, we have $(A_t,\Phi_t)\to (A_\infty,\Phi_{\infty})$, $\Phi_{\infty}=\phi_{\infty}^{1,0}$, given by \eqref{eqn:InfHiggsUnitary_eq}. Then
\begin{align*}
    \Delta_{A_\infty}\gamma&=\begin{pmatrix}
\Delta_0 u_0  &  \Delta_{\frac12(m-2d)} u_1\\ \overline{\Delta_{\frac12(m-2d)} {u}_1} & -\Delta_0 u_0 \end{pmatrix},\\
-i \star M_{\Phi_\infty} \gamma &= 8r^m\begin{pmatrix}
2u_0&u_1-e^{-(m-2d)i\theta}\bar u_1\\
\bar u_1-e^{(m-2d)i\theta}\bar u_1&-2u_0
\end{pmatrix}.
\end{align*}
It follows that the indicial roots of $L_\infty^0 := \Delta_{A_\infty} - i \star M_{\Phi_\infty}$ are $\mathbb{Z}$ for the diagonal terms, and $\mathbb{Z}+\frac m2$ for the off-diagonal terms.

By \cite[Proposition~5.1]{fredrickson2018generic}, $L_t: H^2(i\mathfrak{su}(E))\to L^2(i\mathfrak{su}(E))$ is invertible, we will show that the norm of its inverse $L_t^{-1}$ is bounded by $Ct^2$.

\begin{proposition}
    $L_t^{-1}:L^2(i\mathfrak{su}(E))\to L^2(i\mathfrak{su}(E))$ is bounded.
\end{proposition}
\begin{proof}
    Consider the operator $L_t^0\gamma:=\Delta_{A_t}\gamma-i\star M_{\Phi_t}\gamma$. Since $L_t\geq L_t^0>0$, we only need to show that the smallest eigenvalue $\lambda_t^0$ of $L_t^0$ is uniformly bounded below by some positive constant. Suppose on the contrary that $\lambda_{t_j}^0\to 0$ for a sequence $t_j\to\infty$.

    Near $p\in Z$, let $z$ be a holomorphic coordinate centered at $p$ as in Theorem \ref{thm:LimConfig}. Define smooth positive weight functions $\mu_t$ on $\Sigma$ such that
    \[\mu_t=\begin{cases}(t^{-2/(1+m/2)}+|z|^2)^{1/2}\quad&\text{ on }\Sigma_p(3/4),~p\in Z,\\ 1 \quad &\text{ on } \Sigma\backslash \cup_{p\in Z} \Sigma_p(1),\end{cases}\]
    and that $\mu_t$ is uniformly bounded. Let $\psi_t$ be an eigenfunciton of $L_t^0$ with eigenvalue $\lambda_t^0$, normalized so that $\sup_{\Sigma} \mu_t^\delta |\psi_t|=1$, where $\delta>0$ is a constant to be chosen later. Write $\mu_j=\mu_{t_j}$, $\psi_j=\psi_{t_j}$, and choose $q_j$ such that $\mu_j^\delta(q_j)|\psi_j(q_j)|=1$.

    We claim that $\mu_j^\delta |\psi_j|$ converges to $0$ on any compact subset of $\Sigma\backslash Z$. If not, then for some compact $K\subset \Sigma\backslash Z$, we have (after passing to a subsequence) $\sup_K \mu_j^\delta |\psi_j|\geq \eta>0$.  By elliptic regularity and a diagonalization argument, there exists a subsequence of $\psi_j$ (still denoted by $\psi_j$) converging in $C^\infty$ on any compact subset of $\Sigma\backslash Z$ to $\psi_\infty$. Then we have $L_{\infty}^0\psi_{\infty}=0, |\psi_{\infty}|\leq |z|^{-\delta}$ on $\Sigma_p(3/4)$, and $\psi_{\infty}\neq 0$. $L_\infty^0$ is a conic differential operator, so $\psi_\infty$ has a complete asymptotic expansion near each $p\in Z$ in $r^{\nu_i}$ and $\log r$, where the $\nu_i$ are the indicial roots of $L_{\infty}^0$. Choose $\delta$ to be smaller than the absolute value of the first nonzero indicial root, then each $\nu_i\geq 0$ as $|\psi_\infty|\leq |z|^{-\delta}$. Integrate by parts,
    \[\lambda_j^0\lVert \psi_j\rVert^2=\langle L_j^0\psi_j,\psi_j\rangle=\int |d_{A_j}\psi_j|^2+|[\Phi_j,\psi_j]|^2.\]
    Since $|\psi_j|^2\leq\mu_j^{-2\delta}$, the left hand side tends to $0$, and we have $d_{A_{\infty}}\psi_\infty=[\Phi_{\infty},\psi_{\infty}]=0$. It follows that $\psi_\infty$ is a multiple of $\Phi_\infty$, and $|\psi_{\infty}|$ is constant. Note that $m$ is odd for at least one of $p$, then near $p$, $|\psi_{\infty}|=O(r^{1/2})\to 0$ as $r\to 0$, and $\psi_{\infty}=0$, a contradiction.

    By the above claim, after passing to a subsequence, we may assume that $q_j\to q_\infty\in Z$. Let $z$ be a holomorphic coordinate around $q_\infty$ as above. Let $z_j=z(q_j)$, $w=t^{1/(1+m/2)}z$, and $w_j=t_j^{1/(1+m/2)}z_j$. There are two possibilities, either $|w_j|$ is bounded or $|w_j|\to\infty$.

    Suppose that $|w_j|\leq C$ for all $j$. Let $\tilde{\psi}_j(w)=t_j^{-1/(1+m/2)\delta}\psi_j(t_j^{-1/(1+m/2)}w)$. Then $|\tilde{\psi}_j(w)|\leq (1+|w|^2)^{-\delta/2}$ with equality at $w_j$. Recall that by construction, on $\{ |z|\leq 1/2\}$, $H_t^{\app}$ is the model solution and we may write $(A_t,\Phi_t)=(A_t^{\model},\Phi_t^{\model})$. Then we have $\Delta_{A_t}=t^{2/(1+m/2)}\Delta_{\varrho}$ where $\Delta_{\varrho}$ is the Laplacian associated to $A_1^{\model}$. Similarly $t^2 M_{\Phi_t}=t^{2/(1+m/2)}M_{\varrho}$, where $M_{\varrho}$ is independent of $t$. Since $\psi_j$ are eigenfunctions of $L_{t_j}^0$, we have
    \[(\Delta_{\varrho}-i\star t_j^{-2/(1+m/2)}M_{\varrho})\tilde{\psi}_j(w)=t_j^{-2/(1+m/2)}\lambda_j^0\tilde{\psi}_j(w),\]
    on $\{|w|\leq t_j^{2/(1+m/2)}/2\}$. Let $j\to\infty$, a subsequence of $\tilde{\psi}_j$ converges to $\tilde{\psi}_\infty\neq 0$, defined on $\CC$. Then $\tilde{\psi}_\infty\sim |w|^{-\delta}$ as $w\to\infty$, and $\Delta_{\varrho}\tilde{\psi}_{\infty}=0$. Consider the splitting $\tilde{\psi}_\infty=\psi_{\infty}'+\psi_{\infty}''$ into diagonal and off-diagonal parts. Then $\Delta_0 \psi_{\infty}'=0$, $\psi_{\infty}'=O(|w|^{-1})$ as $w\to\infty$, justifying the integration by parts $\langle \Delta_{\varrho}\psi_\infty',\psi_{\infty}'\rangle=|d_{A_\varrho}\psi_\infty'|^2=0$. $|\psi'_{\infty}|$ is constant and vanishes at infinity, so it is identically zero. Expanding $\psi_\infty''$ into Fourier series, then the $k^{\mathrm{th}}$ Fourier component satisfies
$ (-r^{-2} (r\del_r)^2 - (i \del_\theta - 2 F_1^m)^2) \psi_{\infty,k}'' = 0$.  As $\varrho \to \infty$, this operator converges
to $- r^{-2} (r\del_r)^2 - (k - (m-2d)/2)^2$, and by standard ODE theory, any bounded solution must decay exponentially. Then $\psi_{\infty}''=0$ by the same argument as above. We obtain $\tilde{\psi}_\infty=0$, a contradiction.

The remaining case is that $\sigma_j:=|w_j|\to\infty$. Let $\hat{w}=w/\sigma_j$. Define $\hat{\psi}_j(\hat{w})=\sigma_j^{-\delta}\tilde{\psi}_j(\sigma_j\hat{w})$, then $|\hat{\psi}_j(\hat{w})|\leq (\sigma_j^{-2}+|\hat{w}|^2)^{-\delta/2}$, with equality at some point $\hat{w}_j$ with $|\hat{w}_j|=1$. A subsequence of $\hat{\psi}_j$ converges to $\hat{\psi}_\infty$ on $\CC$, with $|\hat{\psi}_\infty|\leq |w|^{-\delta}$ and $\Delta_{A_{\infty}}\hat{\psi}_{\infty}=0$. The operator $\Delta_{A_{\infty}}$ is conic at both $0$ and $\infty$, and is homogeneous of degree $-2$. Since none of the terms $r^{\nu_i},r^0,\log r$ are bounded by $r^{-\delta}$ at both $0$ and $\infty$, we must have $\hat{\psi}_\infty=0$, a contradiction.
\end{proof}

By the same proof as in \cite[Proposition~5.2b]{fredrickson2018generic}, we have $\lVert L_t^{-1}\rVert_{\mathcal{L}(L^2,H^2)}\leq Ct^2$. The following is analogous to \cite[Lemma~6.8]{mazzeo2016ends}.

\begin{lemma}
The approximate solution satisfies
\begin{equation} \label{eqn:appbound}
 \lVert A_t\rVert_{C^1} \leq Ct
\end{equation}
on the disk $\Sigma_p(1)$, so that for any $H^{k+1}$ section $\gamma$, $k=0,1$,
\begin{equation} \label{eqn:bound2}
 \lVert d_{A_t} \gamma \rVert_{H^k} \leq Ct \lVert\gamma\rVert_{H^{k+1}},
\end{equation}
and moreover,
\begin{equation} \label{eqn:bound3}
\lVert L_t \gamma\rVert_{L^2} \leq Ct^2 \lVert\gamma\rVert_{H^2}.
\end{equation}
\end{lemma}
\begin{proof}
Recall that $A_t=F_{\chi,t}^m \mathrm{diag}(i,-i) d\theta$, where $F_{\chi,t}^m(r)=\frac14(m-2d)+\frac12 r\partial_r(\chi v_t)(r)$. $F_{\chi,t}^m$ has the same asymptotics as $F_t^m$, so is uniformly bounded in $t$. By the construction in Section \ref{subsec:ModelSolution}, we have $F_t^m=\eta(\rho)$, $\rho=\frac{4tr^{\frac{m}{2}+1}}{1+\frac{m}{2}}$, where $\eta(\rho)=\frac14(m-2d)+\frac{1}{2}(1+\frac{m}{2})\rho\psi'(\rho)$. Then
\[ |\partial_r F_t^m(r)|=|4tr^{m/2}\eta'(\rho)|=4tr^{m/2}\Big|\frac{1}{4}\Big(1+\frac{m}{2}\Big)\rho\sinh(2\psi(\rho))\Big|\leq Ct.\]
This implies the desired estimate for $|\partial_r F_{\chi,t}^m|$.
\end{proof}
Using this lemma, we obtain the following estimate for the nonlinear terms as in \cite[Lemma~ 6.9]{mazzeo2016ends}.
\begin{lemma}
There exists a constant $C>0$ such that
\begin{equation}
 \lVert Q_t(\gamma_1) - Q_t(\gamma_2)\rVert_{L^2} \leq C \rho t^2  \lVert \gamma_1 -\gamma_2 \rVert_{H^2}
\end{equation}
for all $\rho \in (0, 1]$ and $\gamma_1, \gamma_2$ satisfying $ \lVert\gamma_i \rVert_{H^2} \leq \rho$.
\end{lemma}
As in \cite[Sec.~6.2]{fredrickson2018generic}, we have the following exponential convergence result.
\begin{theorem}\label{thm:ExpConvergence}
Fix a locally fiducial Higgs bundle $(\bar{\partial}_E, \varphi) \in \mathcal{M}$. For $t$ sufficiently large, there is a unique $H_0$-Hermitian $\gamma_t$ satisfying
$\|\gamma_t\|_{H^2\left(i \mathfrak{su}(E)\right)} \leq C e^{-\delta t}$,
such that  $\mathbf{F}^\app_t(\gamma_t)=0$, i.e., $(d_{A_t^{\exp (\gamma_t) }}, e^{-\gamma_t}\Phi_te^{\gamma_t})$ solves
Hitchin's equations.  (Equivalently, $H_t=H_t^\app\cdot(g_t e^{-\gamma_t}g_t^{-1})$ is harmonic.)
\end{theorem}
Higher order derivatives of $\gamma_t$ can be estimated as in \cite[Sec.~4.2]{mochizuki2023asymptotic}. In fact, by the above theorem and the Sobolev embedding, we have $\lVert\gamma_t\rVert_{\MC^0}\leq Ce^{-\delta t}$. By the construction of $H_t^{\app}$, we have $\lVert g_t\rVert_{\MC^l(K)}\leq C_{l,K}$, and $|s(H_\infty,H_t)-\mathrm{id}|_{H_\infty}\leq C_{0,K} e^{-C'_{0,K} t}$ on $K$ for any compact $K\subset \Sigma\backslash Z$, where $s(H_\infty,H_t)$ is given by $H_t=H_{\infty}s(H_\infty,H_t)$. By \cite[Theorem~4.3]{mochizuki2023asymptotic},
\[\lVert s(H_\infty,H_t)-\mathrm{id}\rVert_{\MC^l(K)}\leq C_{l,K}e^{-C'_{l,K}t},\]
and Theorem \ref{thm:exp_conv} follows.


\section{Gauss-Manin deformations}\label{sec:GM_deform}
For a fixed partition $\bp$, the collection of lowest strata $\calM_{\bp,V_\max}$ is a torsor for the abelian scheme $\Prym(\tilde{\pi}: \wt{S}_q \to \Sigma) \to \calB_{\bp}$. The tangent bundle $T\Prym(\tilde{\pi})$ features a Gauss-Manin connection. To define a (generalized) semi-flat metric on $\calM_{\bp,V_\max}$ we need to understand the tangent vectors along $\calM_{\bp,V_\max}$ that correspond to flat sections of the Gauss-Manin connection. In this section, we will characterize Gauss-Manin families of line bundles in $\Prym(\tilde{\pi}: \wt{S}_q \to \Sigma)$ and the corresponding families of Higgs bundles in $\calM_{\bp,V_{\max}}$. Subsequently,
this allows us to split the exact sequence \eqref{lemm:exact_sequence_M_Vmax} for the tangent bundle $T\calM_{\bp,V_{\max}}$.

\subsection{Singular flat metrics}\label{ssec:sing_flat_metric}
Fix a smooth point $q \in \calB_{\bp}$. Fix a stratum $\mathcal{T}_V\subset \calH^{-1}(q)$ indexed by the divisor $V$.  
Denote by $p\in \Sigma$ the zeros of $q$ and by ${\tilde{p}} \in \wt{S}_q$ their preimages. Furthermore, let  $\wt{S}_q^\circ=\wt{S}_q \setminus \tilde{\pi}^{-1}(Z(q))$ denote the complement of the preimages of the zeros. Let $\calL \in \calP_V$ be a line bundle on the normalized spectral curve satisfying the twisted Prym condition \eqref{eq:twisted_Prym}. Recall that in Section \ref{subsec:lim_config} we defined the canonical weight function $\chi_\msV$ with $\msV=\frac12(\Lambda-\tilde{\pi}^*V)$. The parabolic degree of $(\calL,\chi_\msV)$ is zero. Hence, there exists a flat Hermitian metric $h_\calL$ on $\calL \rest _{\wt{S}_q^\circ}$, such that for each ${\tilde{p}}$, there exists a coordinate $w$ centered at ${\tilde{p}}$ and a local frame $s_{\tilde{p}}$ of $\calL$ at ${\tilde{p}}$ such that $h_\calL(s_{\tilde{p}},s_{\tilde{p}})=\vert w \vert^{\tilde{\pi}^*V_{\tilde{p}}-\Lambda_{\tilde{p}}}$. The Hermitian metric $h_\calL$ is unique up to a positive constant.
\par
Note that whenever $V_{\tilde{p}}= \Lambda_{\tilde{p}}$ the Hermitian metric $h_\calL$ extends smoothly over $\tilde{p}$. In particular, when $V=V_{\max}$, the Hermitian metric extends to $\wt{S}_q\setminus \supp R$, where $R$ is the ramification divisor of $\tilde{\pi}$.
\par
Let $\kappa_\calL: \pi_1(\wt{S}^\circ_q) \to \mathsf{U}(1)$ denote the monodromy representation of $(\calL,\nabla_{h_\calL})$, where $\nabla_{h_\calL}$ denotes the Chern connection. $\kappa_\calL$ descends to a group homomorphism $\kappa_\calL:H_1(\wt{S}^\circ_q, \mathbb{Z}) \to \mathsf{U}(1)$, which we denote by the same symbol.

\begin{lemma}\label{lemm:monodromy_away_from_zeroes}
  Let $\calL_1,\calL_2$ be two line bundles as above. Let $U_{\tilde{p}}$ be a simply connected neighborhood of ${\tilde{p}}$ such that $U_{\tilde{p}} \cap U_{{\tilde{p}}'}=\varnothing$. Let
  \[ \iota: H_1(\wt{S}_q\setminus \bigcup_{\tilde{p}} U_{\tilde{p}}, \mathbb{Z}) \to H_1(\wt{S}_q^\circ, \mathbb{Z})
  \] be the morphism induced by inclusion. If $\kappa_{\calL_1} \circ \iota=\kappa_{\calL_2} \circ \iota$, then there exists a holomorphic isomorphism $\calL_1 \cong \calL_2$.
\end{lemma}
\begin{proof}
  Choose a loop $\gamma_{\tilde{p}}$ around ${\tilde{p}}$. Then by the compatibility with the parabolic structure the monodromy $\kappa_\calL(\gamma_{\tilde{p}})$ is fixed to be $(-1)^{2\alpha_{\tilde{p}}}$. Together with the assumption this implies $\kappa_{\calL_1}=\kappa_{\calL_2}$. Hence, we obtain an isomorphism of flat bundles $(\calL_1,\nabla_{h_{\calL_1}})\cong(\calL_1,\nabla_{h_{\calL_1}})$, and in particular, a holomorphic isomorphism $\calL_1\cong \calL_2$.
\end{proof}

\begin{proposition}\label{prop:sigma-invar_monodromy}
  Let $\calL \in \wt{\mathcal{T}}_V$ be $\sigma$-invariant, i.e., $\sigma^*\calL=\calL$. Then $\kappa_\calL$ is valued in $\{\pm 1\}$.
\end{proposition}
\begin{proof}
  By the Prym condition and $\sigma$-invariance we have a symmetric pairing
  \[ C_\calL: \calL \otimes \calL \cong \calL \otimes \sigma^*\calL \cong \calO(\Lambda-\wt{V}).
  \] We can define a Hermitian metric by $h_\calL(s,s)=\vert C_\calL(s,s) \vert$. Clearly, this metric is compatible with the symmetric pairing. This reduces the structure group of $\calL$ to $\mathrm{O}(1)=\{\pm 1\}$.
\end{proof}

\begin{lemma}\label{lem:exist_sigma_invariant_pointwise}
  If $q$ has a zero of odd order or $\deg(V)\equiv 0 \pmod 2$, then there exists a $\sigma$-invariant line bundle $\calL \in \wt{\mathcal{T}}_{V}$.
\end{lemma}
\begin{proof}
  If $\deg(V)\equiv 0 \pmod 2$, we define a $\sigma$-invariant line bundle by $\calL=\tilde{\pi}^*(K(-V)^{\frac12})$. If $\deg(V)\equiv 1 \pmod 2$, we choose a zero $p$ of odd order. Hence, ${\tilde{p}}=\tilde{\pi}^{-1}(p)$ is a ramification point. We define $\calL=\tilde{\pi}^*(K(-V+p))^{\frac12}(-{\tilde{p}})$.
\end{proof}

\subsection{Gauss-Manin line bundles}
Throughout this subsection we fix a quadratic differential $q$ and a compatible Higgs divisor $V$, so that the conditions of Lemma \ref{lem:exist_sigma_invariant_pointwise} are satisfied.
Let $W \subset \calB_{\bp}$ be a polydisc around $q$ and $\wt{S}_W$ the family of normalized spectral curves over $W$. Let $\mathcal{Z}\subset \Sigma \times W$ be the divisor of zeros of $w \in W$. Then for each Higgs divisor $V=\sum k_i p_i$ compatible with $q$, we obtain a divisor $\mathcal{V}\subset\mathcal{Z}$ by  $\calV=\sum k_i Z_i$, where $Z_i\subset \calZ$ is the integral divisor of $\Sigma \times W$ with $Z_i \cap (\Sigma \times \{q\})=p_i$. Denote by $\calP_\calV \subset \mathsf{Jac}(\wt{S} \to \calB_{\bp})$ the torsor for the relative Prym variety defined in \eqref{eq:twisted_Prym}.
From the construction in Lemma \ref{lem:exist_sigma_invariant_pointwise} it follows that:
\begin{lemma}
  For each polydisc $W$ around $q$ there exists a family of $\sigma$-invariant holomorphic line bundles $\underline{M}$ in $\wt{\mathcal{T}}_{\mathcal{V}}$.
\end{lemma}
\begin{proof}
  The construction from the proof of Lemma \ref{lem:exist_sigma_invariant_pointwise} can be done in families upon labeling one zero of $w \in W$ in a holomorphically varying way. That is clearly possible on every polydisc.
\end{proof}
In the sequel, fix a holomorphic line bundle $\underline{M}$ on $\wt{S}_W$ with $\underline{M}_w \in \wt{\mathcal{T}}_{\mathcal{V}_w}$ for each $w\in W$ such that $\sigma^*\underline{M}_w\cong \underline{M}_w$.
\begin{definition}
  A family of line bundles $\underline{\calL} \in \wt{\mathcal{T}}_{\calV}(\wt{S}_W)$ is called Gauss-Manin if the monodromy associated to the flat line bundle $\underline{\calL}_w \otimes \underline{M}_w^{-1} \in \Prym(\wt{S}_w)$ is independent of $w \in W$.
\end{definition}
From Proposition \ref{prop:sigma-invar_monodromy} we obtain
\begin{corollary}
  \begin{enumerate}
    \itemsep0em
  \item The definition is independent of the choice of the $\sigma$-invariant line bundle $\underline{M}$.
  \item Every $\sigma$-invariant line bundle on $\wt{S}_W$ is Gauss-Manin.
  \end{enumerate}
\end{corollary}

\begin{lemma}\label{lemm:exist_GM_line bundle}
  Let $\calL \in \calP_V$. Then there exists a Gauss-Manin holomorphic line bundle $\underline{\calL}$ on $\wt{S}_W$, such that $\underline{\calL}_{q}=\calL$ and $\underline{\calL} \in \calP_{\calV}$. The Gauss-Manin line bundle $\underline{\calL}$ is uniquely defined up to holomorphic isomorphism.
\end{lemma}
\begin{proof}
  $\calL \otimes \underline{M}_{q}^{-1} \in \Prym(\wt{S}_q)$ is a holomorphic line bundle of degree zero on $\wt{S}_{q}$. Hence, it defines a monodromy representation $\widehat{\kappa}_\calL: \pi_1(\wt{S}_{q}) \to \mathsf{U}(1)$. As $\wt{S}_W$ is homotopy-equivalent to $\wt{S}_{q}$, it induces a representation $\widehat{\kappa}_{\underline{\calL}}: \pi_1(\wt{S}_{W}) \to \mathsf{U}(1)$. Let $(\widehat{\underline{\calL}},\widehat{\nabla})$ denote the corresponding unitary flat bundle on $\wt{S}_{W}$. The Prym condition is given by integrating to zero over the $\sigma$-symmetric cycles on $\wt{S}_q$ and is therefore preserved by the Gauss-Manin connection. Therefore, $\underline{\calL}=\widehat{\underline{\calL}}\otimes \underline{M} \in \calP_\calV$ is a holomorphic line bundle on $\wt{S}_W$ that is Gauss-Manin. For uniqueness see \cite[Lemma 3.9]{mochizuki2023asymptotic2}.
\end{proof}

Lemma \ref{lemm:monodromy_away_from_zeroes} indicates that we can check the condition of being Gauss-Manin on the complement of the preimages of the zeros of $q$ as along as we can glue in the singular Hermitian metric at $\tilde{\pi}^{-1}Z(q)$. To make this precise let us fix a suitable open cover of $\wt{S}_W$.
\par
For each $p \in Z(q)$ let $U_p' \subset U_p$ be two nested simply connected open neighborhood of $p$, such that $U'_p$ is relatively compact in $U_p$, $U_{p_1} \cap U_{p_2}=\varnothing$ for $p_1 \neq p_2 \in Z(q)$  and $Z(w) \subset \bigcup_p U'_p$ for each $w \in W$. Let
\begin{align*}
  &\Sigma_0=\Sigma \setminus \bigcup_{p \in Z(q)}\overline{U}'_p, \quad &(\wt{S}_W)_0= \tilde{\pi}_W^{-1} (\Sigma_0\times W), \\
  &(\wt{S}_W)_{\tilde{p}}= \tilde{\pi}_W^{-1}(U_p \times W), \quad  &(\wt{S}_W)_{0,{\tilde{p}}}=(\wt{S}_W)_{0}\cap (\wt{S}_W)_{{\tilde{p}}}.
\end{align*}
We obtain a biholomorphism $r: (\wt{S}_W)_0 \to W \times (\wt{S}_{q})_0$ by using the covering $\tilde{\pi}: (\wt{S}_w)_0 \to (\Sigma)_0$ for each $w$. Then we have

\begin{proposition}\label{prop:charc_of_GM_line_bundle}
  Let $\underline{\calL} \in \calP_{\mathcal{V}}$ be a holomorphic line bundle on $\wt{S}_W$ and $\calL=\underline{\calL}_q$. Let $h_{\calL}$ be the singular flat metric on $\calL$. Then $\underline{\calL}$ is Gauss-Manin if and only if there exists a holomorphic isomorphism $\phi_0: \underline{\calL}|_{(\wt{S}_W)_0} \cong r^*\calL$ and a non-degenerate symmetric pairing $C_{\tilde{p}}: \underline{\calL}|_{(\wt{S}_W)_{\tilde{p}}} \otimes \underline{\calL}|_{(\wt{S}_W)_{\tilde{p}}} \to \calO(\Lambda-\mathcal{V})$, such that
  \begin{enumerate}\itemsep0em
    \item  $C_{\tilde{p}}$ is compatible with $\phi_0$, i.e. $\phi_0^*C_{\tilde{p}}|_{(\wt{S}_{q})_{0,\tilde{p}}}=C_{\tilde{p}}|_{(\wt{S}_{W})_{0,\tilde{p}}}$.
    \item $C_{\tilde{p}}|_{(\wt{S}_{q})_{\tilde{p}}}$ is compatible with the singular flat metric $h_{\calL}|_{(\wt{S}_{q})_{\tilde{p}}}$.
    \end{enumerate}
\end{proposition}

\begin{proof}
  We only prove the condition for being Gauss-Manin here. The other direction follows as in \cite[Section~3.2.6]{mochizuki2023asymptotic2}.
  We define a Hermitian metric on $\underline{\calL}|_{(\wt{S}_W)_0}$ by $h_0=\phi^*_0h_{\underline{\calL}_{q}}$. Let $h_{\tilde{p}}$ denote the Hermitian metric on $\underline{\calL}|_{(\wt{S}_W)_{\tilde{p}}}$ compatible with the symmetric pairing (see proof of Proposition \ref{prop:sigma-invar_monodromy}). Then by the two conditions $h_0$ and $h_{\tilde{p}}$ for each ${\tilde{p}}$ glue to a singular flat metric $h_{\underline{\calL}}$ on $\underline{\calL}$ compatible with the parabolic structure defined in Section \ref{prop:sigma-invar_monodromy}. However, by definition, the monodromy of $h_{\underline{\calL}}$ on $(\wt{S}_w)_0$ is independent of $w \in W$. Hence, it is a Gauss-Manin family by Lemma \ref{lemm:monodromy_away_from_zeroes}.
\end{proof}

\subsection{Gauss-Manin families of Higgs bundles} Now, we describe what it means for a family of Higgs bundles in the Hitchin fibers over $\calB_{\bp}$ to correspond to a Gauss-Manin family of line bundles. Here we will make use of the parametrization of singular Hitchin fibers in terms of semi-abelian spectral data as explained in Section \ref{subsec:stratification}.
\par
Note that a choice of a symmetric pairing at the preimage of the zeros of $q$ trivializes the semi-abelian spectral data. To see this recall that the $u$-coordinates are invariant upon multiplying the chosen section at the preimage of a zero of the quadratic differential by a constant. Given the symmetric pairing, there is a local section $s$ of $\calL$ at $\tilde{p}$ such that $C_{\tilde{p}}(s,s)=z^{\Lambda_{\tilde{p}}-V_{\tilde{p}}}$, unique up to $\pm 1$. Therefore, we obtain a splitting of the fiber bundle $\mathsf{Ab}: \calT_V \to \calP_V$, that is an isomorphism
\[ \calT_V \cong \calP_V \times (\mathbb{C}^*)^{k_1}\times \CC^{k_2}\,,
  \]
where $k_1,k_2$ are defined in Theorem \ref{thm_stratification_fibration}.
\begin{definition} Let $W \subset \calB_{\bp}$ be a polydisc.
  Let $\underline{(\calE,\varphi)} \to W \times \Sigma$ be a family of Higgs bundles with $\calH\underline{(\calE,\varphi)} \subset W$ and holomorphically varying Higgs divisor $\mathcal{V}$. Let $(\underline{\calL},\eta)$ be the associated family of semi-abelian spectral data. Then $\underline{(\calE,\varphi)}$ is a Gauss-Manin family if $\underline{\calL}$ is a Gauss-Manin family of line bundles and the Hecke modification $\eta$ is constant with respect to the above splitting.
\end{definition}

From Lemma \ref{lemm:exist_GM_line bundle} we immediately conclude
\begin{lemma}\label{lemm:existence_of_GM_Higgs_bundles}
  Let $q \in \calB_{\bp}$ be a smooth point given by a quadratic differential with at least one zero of odd order, $(\calE,\varphi) \in \calH^{-1}(q)$ and let $W \subset \calB_{\bp}$ be a polydisc centered at $q$. Then there exists a Gauss-Manin family $\underline{(\calE,\varphi)}$ on $\Sigma \times \calW$ with $\underline{(\calE,\varphi)}_{q}=(\calE,\varphi)$. It is unique up to isomorphism.
\end{lemma}

  \begin{proposition}[Mochizuki's horizontal criterion]\label{lemm:characterization_Gauss_Manin_Higgs_bundles}
    Let $r: \Sigma_0 \times W \to \Sigma_0 \times \{q\}$ denote the projection.
    A family $\underline{(\calE,\varphi)}$ of locally fiducial Higgs bundles is  Gauss-Manin if and only if there exists a holomorphic isomorphism $\psi_0: \underline{\calE}|_{\Sigma_0 \times W} \cong r^*\underline{\calE}_{\Sigma_0 \times\{ q\}}$ and a non-degenerate symmetric pairing $C_{\tilde{p}}$ inducing an orthogonal structure on $(\underline{\calE},\varphi)|_{\Sigma_{\tilde{p}}}$, such that
    \begin{enumerate}
    \item  The isomorphism $\psi_0$ is induced by an isomorphism of the abelian part of the spectral data.
    \item  $C_{\tilde{p}}$ is compatible with $\psi_0$, i.e. $\psi_0^*C_{\tilde{p}}|_{\Sigma_{0,{\tilde{p}}} \times \{q\}}=C_{\tilde{p}}|_{{\Sigma_{0,{\tilde{p}}} \times W}}$.
    \item $C_{\tilde{p}}|_{\Sigma_{\tilde{p}} \times \{q\}}$ is compatible with the limiting configuration on $(\underline{\calE},\varphi)_{q}$.
    \end{enumerate}
  \end{proposition}
\begin{proof}
  This is a translation of Proposition \ref{prop:charc_of_GM_line_bundle} via the spectral correspondence. Recall from Section \ref{subsec:stratification} that we obtain $(\calE,\varphi,H_\infty)$ from $(\calL,h_\calL)$ by a Hecke modification of Higgs bundles
  \[ \eta: (\calE,\varphi) \to \tilde{\pi}_*(\calL,\lambda)
  \] at $\calZ$. From Proposition \ref{prop:charc_of_GM_line_bundle} we obtain an isomorphism $\phi_0: \underline{\calL}|_{(\wt{S}_W)_0} \cong r^*\underline{\calL}_{b_0}$. It induces an isomorphism $\tilde{\pi}_*\underline{\calL}|_{\Sigma_0 \times W} \cong r^*\tilde{\pi}_*\underline{\calL}_{\Sigma_0\times q}$. Since the Hecke modification is an isomorphism on $\Sigma^\circ$, it descends to an isomorphism $\psi_0:\underline{\calE}|_{\Sigma_0 \times W} \cong r^*\underline{\calE}_{\Sigma_0 \times\{ q\}}$.
 \par  Further, note that in Proposition \ref{prop:charc_of_GM_line_bundle} the symmetric pairing can be multiplied by a $c_{\tilde{p}} \in U(1)$ without changing the conclusion. Let $p \in Z(q)$ be an even order zero and $\{{\tilde{p}}_+,{\tilde{p}}_-\}= \tilde{\pi}^{-1}p$. We can choose frames $s_+, s_-$ and symmetric pairings $C_{{\tilde{p}}_{\pm}}$ compatible with the Hermitian metric such that $C_{{\tilde{p}}_{+}}(s_+,s_+)=-C_{{\tilde{p}}_{-}}(s_-,s_-)$. Under this assumption, there is an induced symmetric pairing on $\tilde{\pi}_*\underline{\calL}$ defined as follows. Let $U \subset \Sigma$ be open. We can associate to a function $f \in \calO_{\tilde{\pi}^{-1}U}$ the function $\mathsf{nm}f=\Tr(F_f)$ on $U$, where $F_f \in \End(\tilde{\pi}_*\calO_{\tilde{\pi}^{-1}U})$ is the multiplication by $f$. Given local sections $s_1,s_2$ of $\tilde{\pi}_*\underline{\calL}\rest_U=\underline{\calL}\rest_{\tilde{\pi}^{-1}U}$ in a neighborhood $U$ of a zero $p$ we define $\widehat{C}_p(s_1,s_2)=\mathsf{nm}( C_{\tilde{\pi}^{-1}p}(s_1,s_2))$.  This defines a symmetric pairing on the pushforward $\tilde{\pi}_*(\calL)\rest_{U_p}$ for each $p$, such that the induced solution to the decoupled Hitchin equations is compatible. Locally at $p$, the pushforward Higgs field, the symmetric pairing, and the pushforward of the Hermitian metric are given as follows
  \begin{align*} &\tilde{\pi}_*\tilde{\lambda}=\begin{pmatrix} 0 & z^{\lfloor \frac{\mathsf{ord}_p(q)+1}{2}  \rfloor}  \\  z^{\lfloor \frac{\mathsf{ord}_p(q)}{2}  \rfloor} & 0 \end{pmatrix},\quad \widehat{C}_p=\begin{pmatrix} 0 & z^{-\lfloor \frac{\mathsf{ord}_p(q)}{2}  \rfloor+V_p} \\ z^{-\lfloor \frac{\mathsf{ord}_p(q)}{2}  \rfloor+V_p} &0 \end{pmatrix}, \\
    &\tilde{\pi}_*h_\calL=\diag( \vert z \vert^{-\lfloor \frac{\mathsf{ord}_p(q)+1}{2}  \rfloor+V_p}, \vert z \vert^{-\lfloor \frac{\mathsf{ord}_p(q)}{2}  \rfloor+V_p} ).
  \end{align*} As $(\mathcal{E},\varphi)$ is locally fiducial, the Hecke modification is given by $\eta=\diag(1, z^{\lfloor \frac{\mathsf{ord}_p(q)}{2}  \rfloor-V_p})$. Hence, $C_p$ descends to a non-degenerate symmetric pairing on $(\mathcal{E},\varphi)$ compatible with the Higgs field and the Hermitian metric.
  \par For the converse, it is clear that the isomorphism $\psi_0$ induces an isomorphism of line bundles $\phi_0: \underline{\calL}|_{(\wt{S}_W)_0} \cong r^*\underline{\calL}_{(\wt{S}_q)_0}$. Furthermore, it is easy to see that local symmetric pairings on $\calE$ induce local symmetric pairings on $\underline{\calL}$ compatible with the singular flat Hermitian metric $h_{\underline{\calL}_q}$.
\end{proof}

\subsection{Gauss-Manin infinitesimal deformations}\label{subsec:GM_infinitesimal_deform}
Recall that the tangent space at a smooth point $(\mathcal{E},\varphi)$ of $\mathcal{M}$ is given by $T_{(\mathcal{E},\varphi)}\mathcal{M}=H^1(\mathsf{Def}(\mathcal{E},\varphi))$. Let $\mathcal{A}^\bullet(\mathsf{Def}(\mathcal{E},\varphi))$ be the Dolbeault resolution of $\mathsf{Def}(\mathcal{E},\varphi)$. The space of global sections of $\mathcal{A}^j(\mathsf{Def}(\mathcal{E},\varphi))$ is denoted by $A^j(\mathsf{Def}(\mathcal{E},\varphi))$, which is the space of smooth $\End^0(\mathcal{E})$-valued $j$-forms.  An element in $ H^1(\Sigma,\mathsf{Def}(\mathcal{E},\varphi))$ can be represented by a cocycle $\tau\in A^1(\mathsf{Def}(\mathcal{E},\varphi))$, i.e., $(\dbarE+\ad\varphi)\tau=0$. Two cocycles $\tau_1,\tau_2$ represent the same cohomology class in $H^1(\Sigma,\mathsf{Def}(\mathcal{E},\varphi))$ if and only if $\tau_1-\tau_2=(\dbarE+\ad\varphi)\alpha$ for some $\alpha\in A^0(\mathsf{Def}(\mathcal{E},\varphi))$.

Let $\mathcal{M}_r$ be defined as in Definition \ref{def:Mr}. Recall that by Proposition \ref{prop:diff_of_Hit} we have a short exact sequence
\begin{equation}\label{eq:SES_tangentMr}
  0\to H^1(\Sigma,K^{-1}(V))\to T_{(\mathcal{E},\varphi)}\mathcal{M}_r\xrightarrow{d\calH}H^0(\Sigma,K^2(-V))\to 0.
\end{equation}
By Lemma \ref{lemm:cohomology_spectral} we have
\[ \pi_{\ast}\mathcal{O}_{S_V}\cong \mathcal{O}_\Sigma\oplus K^{-1}(V),\quad \pi_{\ast}K_{S_V}\cong K^2(-V)\oplus K.
\]
Then \eqref{eq:SES_tangentMr} becomes
\begin{equation}\label{eq:SES_tangentMr1}
  0\to H^1(S_V,\mathcal{O}_{S_V})^-\xrightarrow{\iota^{\mathrm{v}}} T_{(\mathcal{E},\varphi)}\mathcal{M}_r\xrightarrow{d\calH}H^0(S_V,K_{S_V})^-\to 0.
\end{equation}

\begin{lemma}[{\cite[Lemma~2.11]{mochizuki2023asymptotic2}}]
  Suppose $\tau\in A^1(\mathsf{Def}(\mathcal{E},\varphi))$ is a 1-cocycle representing a class in $T_{(\mathcal{E},\varphi)}\mathcal{M}_r\subset H^1(\Sigma,\mathsf{Def}(\mathcal{E},\varphi))$. Then $\tau$ induces $\rho(\tau)\in H^0(S_V,K_{S_V})^-$.
\end{lemma}

 Let $(\mathcal{E},\varphi)$ be in the lowest stratum, i.e., $V=V_{\mathsf{max}}$. In this case $S_V=\widetilde{S}_q$ and $(\mathcal{E},\varphi)\in \MM^d$, where $d=3g-3-r$. A tangent vector $[(\beta,\psi)]$ at $(\mathcal{E},\varphi)$ is called Gauss-Manin if there exists a Gauss-Manin family of Higgs bundles through $(\mathcal{E},\varphi)$ with tangent vector $[(\beta,\psi)]$. We denote by $T_{(\mathcal{E},\varphi)}\mathcal{M}^{GM}$ the space of Gauss-Manin tangent vectors.


Let $U$ be a subset of $\Sigma$ and denote $(\mathcal{E}_U,\varphi_U):=(\mathcal{E},\varphi)|_U$. Suppose $C_U$ is a non-degenerate symmetric pairing of $(\mathcal{E}_U,\varphi_U)$. Let $\Sy^0_{C_U}(\mathcal{E}_U)$ (resp. $\Asy^0_{C_U}(\mathcal{E}_U)$) denote the sheaf of trace-free holomorphic endomorphisms of $\mathcal{E}_U$ which are symmetric (resp. anti-symmetric) with respect to $C_U$. We obtain the following subcomplex $\mathsf{Def}(\mathcal{E}_U,\varphi_U;C_U)$ of $\mathsf{Def}(\mathcal{E}_U,\varphi_U)$:
\[\Asy^0_{C_U}(\mathcal{E}_U)\xrightarrow{[\varphi,\cdot]}\Sy^0_{C_U}(\mathcal{E}_U)\otimes K. \]
Let $A^\bullet (\mathsf{Def}(\mathcal{E}_U,\varphi_U;C_U))$ be the space of global sections of the Dolbeault resolution of $\mathsf{Def}(\mathcal{E}_U,\varphi_U;C_U)$. Then we have
\[A^\bullet (\mathsf{Def}(\mathcal{E}_U,\varphi_U;C_U))=\Omega^{1,0}(\Sy^0_{C_U}(\mathcal{E}_U))\oplus \Omega^{0,1}(\Asy^0_{C_U}(\mathcal{E}_U)).\]

\begin{lemma}[{\cite[Lemma~2.14]{mochizuki2023asymptotic2}}]\label{lemm:local_exactness}
Let $U$ be a non-compact connected open subset of $\Sigma$. Let $\tau\in A^1(\mathsf{Def} (\mathcal{E}_U,\varphi_U;C_U))$ be such that $(\dbarE+\ad\varphi)\tau=0$ and $\rho(\tau)=0$. Then there exists a unique section $\sigma$ of $A^0(\mathsf{Def}(\mathcal{E}_U,\varphi_U;C_U))$ such that $(\dbarE+\ad\varphi)\sigma=\tau$ and $\mathrm{Supp}\,\sigma\subset \mathrm{Supp}\,\tau$.
\end{lemma}

By \cite[Eq.~(18)]{mochizuki2023asymptotic2}, we may regard $\coker(\ad\varphi)|_{\Sigma_0}$ as a subsheaf of $(\End^0(\mathcal{E})\otimes K)|_{\Sigma_0}$, and
$\coker(\ad\varphi)|_{\Sigma_0\cap U_p}$
as a subsheaf of
$\bigl(
\Sy^0_{C_p}(\mathcal{E}_p)\otimes K
\bigr)|_{\Sigma_0\cap U_p}$. Define a subcomplex $\mathfrak{A}^{\bullet}$ of the total Čech complex
$\Tot\, \check{C}^{\bullet}(\mathsf{Def}(\mathcal{E},\varphi))$ by
\[
 \mathfrak{A}^0:=\bigoplus_{p\in Z(q)} H^0(U_p,\Asy^0_{C_p}(\mathcal{E}_p)),
\]
\[
 \mathfrak{A}^1:= H^0(\Sigma_0,\coker(\ad\varphi))\oplus \bigoplus_{p\in Z(q)}
 H^0(U_p,\Sy^0_{C_p}(\mathcal{E}_p)\otimes K)\oplus\bigoplus_{p\in Z(q)}
H^0(\Sigma_0\cap U_p,\Asy^0_{C_p}(\mathcal{E}_p)),
\]
\[
 \mathfrak{A}^2:=\bigoplus_{p\in Z(q)} H^0(\Sigma_0\cap U_p,\Sy^0_{C_p}(\mathcal{E}_p)\otimes K).
\]
Then we have the morphism $\Upsilon_0:H^1\bigl(\mathfrak{A}^{\bullet}\bigr)
\to H^1(\Sigma,\mathsf{Def}(\mathcal{E},\varphi))$. By Proposition \ref{prop:diff_of_Hit}, we have the morphism $d\calH:
 H^1(\Sigma,\mathsf{Def}(\mathcal{E},\varphi))\to H^0(\widetilde{S}_q,K_{\widetilde{S}_q})^-$. Similar to \cite[Theorem~3.28]{mochizuki2023asymptotic2}, using Proposition \ref{lemm:characterization_Gauss_Manin_Higgs_bundles} we have the following.

\begin{theorem}
\label{thm:iso_horizontal}
The composite $d\calH\circ\Upsilon_0$ is an isomorphism and $\mathrm{Im}(\Upsilon_0)=T_{(\mathcal{E},\varphi)}\mathcal{M}^{GM}$.
\end{theorem}

We define the isomorphism
$\iota^{\mathrm{h}}:
H^0(\widetilde{S}_q,K_{\widetilde{S}_q})^- \to T_{(\mathcal{E},\varphi)}\mathcal{M}^{GM}$
by
\[
 \iota^{\mathrm{h}}(\nu)
 =\Upsilon_0\circ
 \bigl(
d\calH\circ\Upsilon_0
 \bigr)^{-1}(\nu).
\]
We also regard $\iota^{\mathrm{h}}$ as a morphism
$H^0(\widetilde{S}_q,K_{\widetilde{S}_q})^-
\to H^1(\Sigma,\mathsf{Def}(\mathcal{E},\varphi))$.

  Let $\nu\in H^0(\nS,K_{\nS})^-$. The multiplication of $\nu$ on $\calL$ induces \[F_\nu \in H^0(\Sigma,\End^0(\mathcal{E})\otimes K(Z(q))).\]
By \cite[Lemma~2.10]{mochizuki2023asymptotic2},
there exists a section $g_p\in H^0(U_p,\Asy^0_{C_p}(\mathcal{E}_p)(Z(q)))$ on $U_p$ such that
\[
\varphi_p:=F_{\nu}|_{U_p}-\ad(\varphi)(g_p)
\in H^0(U_p,\Sy^0_{C_p}(\mathcal{E}_p)\otimes K_{U_p}).
\]
Let $\psi_p:=g_{p}|_{U_p\cap \Sigma_0}$ and $\varphi_0:=F_{\nu}|_{\Sigma_0}$. Then we obtain a $1$-cocycle
$a=(\varphi_0,\varphi_p,\psi_p)$ in $\mathfrak{A}^1$. The cohomology class
$[a]\in H^1(\mathfrak{A}^{\bullet})$ is $(d\calH\circ\Upsilon_0)^{-1}(\nu)$. Its image in $H^1(\Sigma,\mathsf{Def}(\mathcal{E},\varphi))$
is $\iota^{\mathrm{h}}(\nu)$. Similar to \cite[Corollary~3.32]{mochizuki2023asymptotic2}, we have the following.
\begin{corollary}\label{prop:splitting}
Let $\tau\in A^1(\mathsf{Def}(\mathcal{E},\varphi))$ be a $1$-cocycle such that $\rho(\tau)=\nu$. We also assume the following conditions:
\begin{enumerate}
 \item [(i)] $\tau|_{\Sigma\setminus\bigcup U_p'}
       =F_{\nu|\Sigma\setminus\bigcup U_p'}$,
 \item [(ii)] $\tau|_{U_p}\in A^1(\mathsf{Def}(\mathcal{E}_p,\varphi_p;C_p))$
      for each $p\in Z(q)$.
\end{enumerate}
Then, $[\tau]=\iota^{\mathrm{h}}(\nu)$ holds in $H^1(\Sigma,\mathsf{Def}(\mathcal{E},\varphi))$. In particular, $[\tau]\in T_{(\mathcal{E},\varphi)}\mathcal{M}^{GM}$. Conversely, any element of $T_{(\mathcal{E},\varphi)}\mathcal{M}^{GM}$ is described in this way.
\end{corollary}

Let $H$ be a harmonic metric of $(\mathcal{E},\varphi)$, and $\nabla_H=\dbarE+\partial_{E,H}$ be the Chern connection of $(\mathcal{E},H)$. Fix a conformal metric $g_{\Sigma}$ of $\Sigma$. Let $(\dbarE+\ad\varphi)^{\ast}_{H,g_\Sigma}$ be the formal adjoint of $\dbarE+\ad\varphi$ with respect to $H$ and $g_\Sigma$. Let
\begin{equation}\label{eq:Laplacian}
    \Delta_{H,g_\Sigma}=(\dbarE+\ad\varphi)^{\ast}_{H,g_\Sigma}(\dbarE+\ad\varphi)+(\dbarE+\ad\varphi)(\dbarE+\ad\varphi)^{\ast}_{H,g_\Sigma}.
\end{equation}
$\tau\in A^1(\mathsf{Def}(\mathcal{E},\varphi))$ is called a harmonic 1-form, if $\Delta_{H,g_\Sigma}\tau=0$, or equivalently \[(\dbarE+\ad\varphi)\tau=(\partial_{E,H}+\varphi^\dagger_H)\tau=0.\]
We have $H^1(\Sigma,\mathsf{Def}(\mathcal{E},\varphi))\cong \mathrm{Harm}^1(\End^0(\mathcal{E}),\ad\varphi,H)$, the space of harmonic 1-forms.

Let $\omega_{\mathcal{M},(\mathcal{E},\varphi)}$ be the alternating bilinear form on $A^1(\mathsf{Def}(\mathcal{E},\varphi))$ defined by
\[\omega_{\mathcal{M},(\mathcal{E},\varphi)}(\tau_1,\tau_2)=-\int_\Sigma \Tr(\tau_1\wedge \tau_2).\]
It restricts to an alternating bilinear form on $\mathrm{Harm}^1(\End^0(\mathcal{E}),\ad\,\varphi,H)\cong T_{(\mathcal{E},\varphi)}\mathcal{M}$. Let $\omega_{\mathcal{M}^d}$ be the  restriction of $\omega_{\mathcal{M}}$ to $\MM^d$.

\begin{lemma}[{\cite[Corollary~3.33]{mochizuki2023asymptotic2}}]
$\mathrm{Im}(\iota^{\mathrm{h}})$ is Lagrangian with respect to $\omega_{\mathcal{M}^d}$.
\end{lemma}

\section{The semi-flat metric}\label{sec:semi-flat}
Recall that in Section \ref{ssec:Hitchin_subint} we defined the subintegrable system $\calH_d: \calM^d \to \calB^d$ with fibers being abelian torsors over $\mathsf{Prym}(\widetilde{S}_q)$. By Freed's construction \cite{freed1999special}, the base locus $\calB^d$ carries a special Kähler metric and the subintegrable system $\calM^d$ carries a semi-flat hyperkähler metric. In this section we describe the semi-flat hyperkähler metric on $\calM^d$.

\subsection{The special Kähler metric on $\calB^d$}
A K\"ahler manifold $(M^{2m},\omega,I)$ is called special K\"ahler if there exists a flat, symplectic, torsion-free connection $\na$ such that $d_{\na}I=0$, where $I$ is regarded as a $TM$-valued 1-form. The special Kähler metric associated to an algebraically completely integrable system decodes the periods of the abelian scheme.

The polarization on the $\mathsf{Prym}(\widetilde{S}_q)$ is given by the restriction of the intersection pairing to the lattice $H_1(\widetilde{S}_q,\mathbb{Z})^-$. Locally on $\calB^d$ we can choose cycles $\gamma_i,\delta_i$ with $i=1, \dots, 3g-3-d$ that form a relative symplectic basis of the odd homology $\Gamma_q=H_1(\widetilde{S}_q,\mathbb{Z})^-$. We identify the tangent space to $\calB^d$ with the cotangent space to the Prym variety by
\[ \tau: H^0(\Sigma,K^2(-D)) \to H^0(\widetilde{S}_q,K_{\widetilde{S}_q})^-, \quad \dot{q} \mapsto \frac{\tilde{\pi}^*\dot{q}}{2\tilde{\omega}}.
\]
Here $\tilde{\pi}^*\dot{q}$ is a quadratic differential with simple zeros at the preimages of the nodes and order $2$ zeros at $R$. Hence the fraction is holomorphic. Furthermore, $\tilde{\omega}$ is anti-symmetric and the pullback $\tilde{\pi}^*\dot{q}$ is symmetric with respect to $\sigma$. Hence $\tau$ maps to the anti-symmetric differentials. By dimension counting it is an isomorphism.
\begin{lemma}
  The conjugate special Kähler coordinates are given by $z_i=\int_{\gamma_i} \tilde{\omega}$ and $w_i=\int_{\delta_i} \tilde{\omega}$ and the Riemannian metric is given by
  \begin{align} g_{\mathrm{sK}}(\dot{q},\dot{q})=\frac{\kappa}{4}\int_{\Sigma} \frac{\vert \dot{q} \vert ^2}{\vert q\vert} = \frac{\kappa}{2} \int_{\widetilde{S}_q} \vert \tau(\dot{q})\vert ^2. \label{eq:g_sK}
  \end{align} 
\end{lemma}
\begin{proof}
  Define the complex coordinates $z_i$ and $w_i$ as above. Let $q_t: (-\epsilon,\epsilon) \to \calB^d$ with $q_t=q_0+t\dot{q}+\calO(t^2)$. Then 
  \[  \mathrm{d}\, \mathrm{Re}(z_i)(\dot{q})=\frac{\mathrm{d}}{\mathrm{d}t}\Big|_{t=0}\int_{\gamma_i}\mathrm{Re} (\lambda(q_t))= \int_{\gamma_i}\mathrm{Re} (\tau(\dot{q}))
  \] as $\tilde{\omega} + \epsilon \tau(\dot{q})$ satisfies the equation $T^2+ \tilde{\pi}^*q_t=0$ to first order. Hence, $\mathrm{d}\, \mathrm{Re}(z_i)(\dot{q})=0$ if and only if for all tangent vectors $\dot{q} \in T\calB^d$ the cap product of $\mathrm{Re} (\tau(\dot{q}))$ with $\gamma_i$ is zero. Hence, $x_i=\mathrm{Re}(z_i)$ defines a flat coordinate with respect to the connection on $T\calB^d$ induced by the Gauss-Manin connection on the family of abelian varieties $\Prym(\widetilde{S})$. Similarly, $y_i=\mathrm{Re}(w_i)$ defines a flat real coordinate and together $x_i,y_i$ define a flat coordinate system. The special Kähler coordinates are by definition the associated holomorphic coordinates to $z_i,w_i$.
  The Kähler potential was computed in \cite{Hitchin2021subintegrable_special_Kaehler} to be given by
  \[ \calK= \int_{\Sigma} \vert q \overline{q} \vert \mathrm{d} A= \frac{1}{2} \int_{\wt{S}_q} \vert \tilde{\omega} \vert ^2  \mathrm{d} A   .
  \] By differentiating we obtain the description of the Kähler metric given above.
\end{proof}


\subsection{The semi-flat hyperkähler metric}
 As explained above $\mathcal{M}^d$ is a torsor over the relative Prym variety of the normalized spectral cover given by $\mathsf{Prym}(\wt{S})=T^\vee \calB/\Gamma$ for a family of lattices $\Gamma$. The Gauss-Manin connection provides a local splitting of tangent vectors
\[ T\mathsf{Prym}(\wt{S})= T T^\vee \calB= T\calB \oplus T^\vee \calB.
\] With respect to this splitting the semi-flat hyperkähler metric is defined by
\begin{align} g_{\mathrm{sf}}=g_{\mathrm{sK}} \oplus g_{\mathrm{sK}}^\vee \label{eq:g_sf}.
\end{align}
The pairing $T\calB^d \times T^\vee\calB^d \to \CC$ is realized in the present case by the Serre pairing between $H^0(\widetilde{S}_q,K_{\widetilde{S}_q})^-$ and $H^1(\widetilde{S}_q,\calO_{\widetilde{S}_q})^-$.
\par
Hitchin's subintegrable system is a torsor for $\mathsf{Prym}(\wt{S})$ and the Gauss-Manin splitting of the tangent space is given by
\[ T_{(\mathcal{E},\varphi)}\calM^d = (T_{(\calE,\varphi)}\calM^d)^{GM} \oplus (T_{(\calE,\varphi)}\calM^d)^{\mathrm{v}}\,.
\]
Furthermore, we computed in Lemma \ref{lemm:cohomology_spectral}
\[ (T_{(\calE,\varphi)}\calM^d)^{GM} =H^0(\wt{S}_q,K_{\wt{S}_q})^- \quad \text{and} \quad (T_{(\calE,\varphi)}\calM^d)^{\mathrm{v}}=H^1(\wt{S}_q,\calO_{\wt{S}_q})^-\,.
\] Using this the semi-flat hyperkähler metric is described by the following proposition:

\begin{proposition}[{\cite[Proposition 2.14]{fredrickson2020exponential}}]
  The semi-flat hyperkähler metric is determined by the following properties:
  \begin{enumerate}
  \item Horizontal and vertical tangent vectors are orthogonal.
  \item  On horizontal tangent vectors $\dot{\tau}\in H^0(\wt{S}_q,K_{\wt{S}_q})^-$ it is $\Vert \dot{\tau} \Vert^2= \frac{1}{2}\kappa \int_{\wt{S}_q} \vert \dot{\tau} \vert^2$.
  \item  On vertical tangent vectors $\beta \in H^1(\wt{S}_q,\calO_{\wt{S}_q})^-$ it is $\Vert \beta \Vert^2= 2\kappa^{-1} \int_{\wt{S}_q} \vert \beta^h \vert^2$\,, where $\beta^h$ is a harmonic representative for the associated class $\beta^{0,1} \in H^{0,1}(\wt{S}_q)^-$.
  \end{enumerate}
\end{proposition}
\begin{proof}
  The properties (i) and (ii) are clear from equations \eqref{eq:g_sK} and \eqref{eq:g_sf}. For the last property we have to identify $g_{\mathrm{sK}}^\vee$. The identification of $T_{q}\calB= H^{1,0}(\wt{S}_q)^-$ with its dual using $g_{\mathrm{sK}}$ is given by $\dot{\tau} \mapsto g_{\mathrm{sK}}(\dot{\tau},\cdot)= -\frac{1}{2} \kappa \int \ast \dot{\tau} \wedge \cdot$. Note that Serre duality pairs $H^{0,1}(\wt{S}_q)^i \otimes H^{1,0}(\wt{S}_q)^j$ to zero for $(i,j)=(+,-)$ or $(-,+)$. Hence, Serre duality induces a non-degenerate pairing on the anti-invariant parts. Denoting by $\calH^{(k,l)}$ the harmonic $(k,l)$-forms we obtain the isomorphism
  \[ F: H^{1,0}(\wt{S}_q)^-=\mathcal{H}^{1,0}(\wt{S}_q)^- \to \mathcal{H}^{0,1}(\wt{S}_q)^-, \quad \dot{\tau} \mapsto -\tfrac{1}{2} \kappa \ast \dot{\tau}\,,
  \] where $\ast$ denotes the (conjugate linear) Hodge star operator. Hence, for $\beta \in H^1(\wt{S}_q,\calO_{\wt{S}_q})^-$ the induced metric $g_{\mathrm{sK}}^\vee$ is $g_{\mathrm{sK}}$ evaluated on  the inverse $F^{-1}$ applied to a harmonic representative. This is the formula stated in the proposition.
\end{proof}

\section{Comparison of metrics}
\subsection{Metrics}
In this subsection, we review the definitions of two hyperkähler metrics on $\mathcal{M}$ which we will compare. The Hitchin metric $g_{L^2}$ is a complete hyperkähler metric on $\mathcal{M}$, while the semi-flat metric $g_{\mathrm{sf}}$ is an incomplete hyperkähler metric defined only on certain subsets of $\mathcal{M}$.

We will define the Hitchin metric $g_{L^2}$ on $\mathcal{M}$ in terms of harmonic 1-forms, as in \cite[Section~4.2.1]{mochizuki2023asymptotic2}. Let $H$ be a harmonic metric of $(\mathcal{E},\varphi)$, and $\nabla_H=\dbarE+\partial_{E,H}$ be the Chern connection of $(\mathcal{E},H)$. Fix a conformal metric $g_{\Sigma}$ of $\Sigma$. Recall that $H^1(\Sigma,\mathsf{Def}(\mathcal{E},\varphi))\cong \mathrm{Harm}^1(\End^0(\mathcal{E}),\ad\,\varphi,H)$. Then the Hitchin metric is given by
\[g_{L^2}|_{(\mathcal{E},\varphi)}(\tau,\tau)=2\sqrt{-1}\int_{\Sigma}\Tr\big(\tau^{1,0}\wedge (\tau^{1,0})^\dagger_H-\tau^{0,1}\wedge(\tau^{0,1})^\dagger_H\big),\]
for $\tau\in \mathrm{Harm}^1(\End^0(\mathcal{E}),\ad\,\varphi,H)$. By \cite[Section~4.2.3]{mochizuki2023asymptotic2}, this coincides with Hitchin's original definition.

 As in Section \ref{subsec:GM_infinitesimal_deform}, there is a splitting $\iota^{\mathrm{h}}$ of
\[0 \to H^1(\nS,\mathcal{O}_{\nS})^- \to T_{(\mathcal{E},\varphi)}\calM^d \xrightarrow{d \calH}  H^0(\nS,K_{\nS})^- \to 0,\]
which induces the following isomorphism:
\[T_{(\mathcal{E},\varphi)}\mathcal{M}^d\cong H^1(\nS,\mathcal{O}_{\nS})^-\oplus H^0(\nS,K_{\nS})^-.  \]
Using this isomorphism, we define the semi-flat metric on $\MM^d$ by
\[\gsf|_{(\mathcal{E},\varphi)}([\tau],[\tau])=2\sqrt{-1}\int_{\nS}\big(\tau^{1,0}\wedge \overline{\tau^{1,0}}-\tau^{0,1}\wedge\overline{\tau^{0,1}}\big),\]
where $\tau$ is a harmonic 1-form on $\nS$ with $[\tau^{0,1}]\in H^1(\nS,\mathcal{O}_{\nS})^-$ and $[\tau^{1,0}]\in H^0(\nS,K_{\nS})^-$. According to Sections \ref{sec:GM_deform} and \ref{sec:semi-flat}, on the subintegrable system $\MM^d$, this definition coincides with Freed's classical construction of the semi-flat metric.

\subsection{Comparing $g_{L^2}$ and $g_{\mathrm{sf}}$}
\subsubsection{Vertical deformations} Let $\tau\in H^1(\nS,\mathcal{O}_{\nS})^-$, represented by a harmonic $(0,1)$-form on $\nS$, also denoted by $\tau$. Via the natural identification of $\nSt$ and $\nS$, we regard $\tau$ as an element in $H^1(\nSt,\mathcal{O}_{\nSt})^-$. We will study the harmonic representative ${\tt V}(\tau,t)\in A^1(\mathsf{Def}(\mathcal{E},t\varphi))$ of $\ivt(\tau)\in H^1(\Sigma,\mathsf{Def}(\mathcal{E},t\varphi))$.

By the assumption $V=V_{\mathrm{max}}$, we have $\tilde{\pi}_\ast(\calL)=\mathcal{E}$. The multiplication of $\tau$ on $\calL$ induces a $C^\infty$ section $F_\tau$ of $\Omega^{0,1}_{\Sigma^\circ}\otimes\End^0(\mathcal{E})$, which is a harmonic 1-form of $(\End(\mathcal{E}),\ad\,\varphi,H_\infty)|_{\Sigma^\circ}$. Let $N(q)$ be a neighborhood of $Z(q)$. By the asymptotic decoupling \cite[Section~2]{Mochizukiasymptotic}, we have the following estimates.
\begin{lemma}[{\cite[Lemma~4.22]{mochizuki2023asymptotic2}}]\label{lem:Ftau_Est}
    There exist positive constants $B_i$ ($i=0,1,2$) such that
    \[ |F_\tau|_{H_t,g_\Sigma}\leq B_0\lVert\tau\rVert,\quad |(\dbarE+\ad\,t\varphi)^\ast_{H_t,g_{\Sigma}}F_\tau|_{H_t}\leq B_1\exp(-B_2t)\lVert\tau\rVert,\]
    on $\Sigma\backslash N(q)$ for any $t\geq 1$.
\end{lemma}

  Let $\Sigma_p:=\Sigma_p(2)$. $\nS|_{\Sigma_p}$ has $|\tilde{\pi}^{-1}(p)|$ connected components. There exists an anti-holomorphic function $\beta_p$ on $\nS|_{\Sigma_p}$ such that $\bar{\partial}\beta_p=\tau|_{\nS|_{\Sigma_p}}$ and $\beta_p(\tilde{p})=0$ for all $\tilde{p}\in \tilde{\pi}^{-1}(p)$. The multiplication of $\bar{\beta}_p$ on $\calL$ induces $F_{\bar{\beta}_p}\in C^{\infty}(\End^0(\mathcal{E}))$, satisfying
\[(\partial_{E,H_t}+\ad\,t\varphi_{H_t}^\dagger)(F_{\bar{\beta}_p})^\dagger_{H_t}=0.\]
We obtain the following harmonic 1-form of $(\End^0(\mathcal{E}),t\ad\,\varphi,H_t)$ on $\Sigma_p$:
\[V_p(\tau,t):=(\dbarE+\ad\, t\varphi)(F_{\bar{\beta}_p})^\dagger_{H_t}.\]
\begin{lemma}[{\cite[Lemma~4.24]{mochizuki2023asymptotic2}}]\label{lem:Vert_Local_Est1}
    For any $1/4<r_2<r_1<2$, there exist $B_i>0$ ($i=0,1,2,3$) such that
    \[|(F_{\bar{\beta}_p})^\dagger_{H_t}|_{H_t}\leq B_0\lVert\tau\rVert,\quad |{\tt V}_p(\tau,t)-F_\tau|_{H_t,g_{\Sigma}}\leq B_1\exp(-B_2 t)\lVert\tau\rVert,\quad |{\tt V}_p(\tau,t)|_{H_t,g_\Sigma}\leq B_3\lVert\tau\rVert, \]
    on $\Sigma_p(r_1)\backslash \Sigma_p(r_2)$ for any $t\geq 1$.
\end{lemma}
\begin{lemma}\label{lem:Vert_Local_Est2}
For any $0<r<2$, there exist $B,B'>0$ such that $ \lVert {\tt V}_p(\tau,t) \rVert_{L^2(\Sigma_p(r)),H_t}\leq B\lVert\tau\rVert$ and $|{\tt V}_p(\tau,t)|_{H_t,g_\Sigma}\leq B't^{2/(m_p+1)}\lVert\tau\rVert$ on $\Sigma_p(r)$ for any $t\geq 1$.
\end{lemma}
\begin{proof}
This follows from the proof of \cite[Lemma~4.16]{mochizuki2023asymptotic2}, the only difference is that we use the rescaling $z\mapsto t^{2/(m_p+1)}z$ instead of $z\mapsto tz$.
\end{proof}

For each $p\in Z(q)$, let $\chi_p:\Sigma\to [0,1]$ be a $C^\infty$ function such that $\chi_p=1$ on $\Sigma_p(1/2)$ and $\chi_p=0$ on $\Sigma\backslash\Sigma_p(1)$. Define the smooth $(0,1)$-form $\tau'$ on $\nS$ as
\[\tau'=\tau-\sum_{p\in Z(q)}\bar{\partial}(\tilde{\pi}^\ast(\chi_p)\cdot\beta_p),\]
which induces a 1-cocycle $F_{\tau'}\in A^1(\mathsf{Def}(\mathcal{E},t\varphi))$ such that $[F_{\tau'}]=\iota_t^{\mathrm{v}}([\tau])$. Let
\[{\tt V}'(\tau,t)=F_{\tau'}+\sum_{p\in Z(q)}(\dbarE+\ad\,t\varphi)(\chi_p\cdot (F_{\bar{\beta}_p})^\dagger_{H_t}),\]
which is a smooth section of $\Omega^1_\Sigma\otimes\End^0 (\mathcal{E})$. By the construction and Lemmas \ref{lem:Ftau_Est}-\ref{lem:Vert_Local_Est2} we have the following.
\begin{lemma}\label{lem:V_prime_Est}
\begin{enumerate}
  \item $(\dbarE+\ad\, t\varphi){\tt V}'(\tau,t)=0$ and $[{\tt V}'(\tau,t)]=\iota_t^{\mathrm{v}}([\tau])$.
  \item There exist $B_1,B_2>0$, such that $\ |{\tt V}'(\tau,t)-F_\tau|_{H_t,g_\Sigma}\leq B_1\exp(-B_2 t)\lVert \tau\rVert\ $ on $\Sigma\backslash \bigcup \Sigma_p(1/2)$ and $|{\tt V}'(\tau,t)-V_p(\tau,t)|_{H_t,g_\Sigma}\leq B_1\exp(-B_2 t)\lVert \tau\rVert$ on $\Sigma_p$.
  \item There exists $B>0$ such that $|{\tt V}'(\tau,t)|_{H_t,g_\Sigma}\leq Bt^{2/(m_p+1)}\lVert\tau\rVert$ and $\lVert {\tt V}'(\tau,t)\rVert_{L^2,H_t}\leq B\lVert \tau\rVert$.
  \item There exist $B_3,B_4>0$ such that $|(\dbarE+\ad\,t\varphi)^\ast_{H_t,g_\Sigma}{\tt V}'(\tau,t)|_{H_t}\leq B_3\exp(-B_4t)\lVert\tau\rVert$.
\end{enumerate}
\end{lemma}
Let $\Delta_t$ be $\Delta_{H_t,g_{\Sigma}}$ defined in \eqref{eq:Laplacian}. There exists a unique $\gamma(\tau,t)\in A^0(\End^0(\mathcal{E}))$ such that
\[\Delta_t\gamma=(\dbarE+\ad\, t\varphi)^\ast_{H_t,g_\Sigma}(\dbarE+\ad\,t\varphi)\gamma(\tau,t)=(\dbarE+\ad\,t\varphi)^\ast_{H_t,g_\Sigma}{\tt V}'(\tau,t).\]
\begin{proposition}\label{Prop:Bound_Inv_Laplacian}
    There exists $C>0$ such that $\lVert\Delta_t s\rVert_{L^2,H_t}\geq C\lVert s\rVert_{L^2,H_t}$ for any $t\geq 1$ and $s\in A^0(\End^0(\mathcal{E}))$.
\end{proposition}
\begin{proof}
Since $(\mathcal{E},t\varphi)$ is stable, we have $\ker\Delta_t=\{0\}$. Let $a_t>0$ be the first eigenvalue of $\Delta_t$. It suffices to show that $\inf_{t\geq 1} a_t>0$. Suppose on the contrary that there is a sequence $t(i)\to\infty$ such that $a_{t(i)}\to 0$. There is a sequence $s_i\in A^0(\End^0(\mathcal{E}))$ such that $\lVert s_i\rVert_{L^2,H_{t(i)}}=1$ and $\Delta_{t(i)}s_i=a_{t(i)}s_i$. We have
\[\lVert t[\varphi,s_i]\rVert_{L^2,H_{t(i)}}+\lVert t[\varphi^\dagger_{H_{t(i)}},s_i]\rVert_{L^2,H_{t(i)}}+\lVert \dbarE (s_i)\rVert_{L^2,H_{t(i)}}+\lVert\partial_{E,H_{t(i)}}(s_i)\rVert_{L^2,H_{t(i)}}\to 0. \]
After passing to a subsequence, we may assume that $s_i$ converges weakly to $s_\infty$ in $L_1^2$ locally on $\Sigma^\circ$. We have
\begin{equation}\label{eq:s_infty_condition}
    \dbarE s_\infty=[\varphi,s_\infty]=\partial_{E,H_\infty} s_\infty=[\varphi^\dagger_{H_\infty},s_\infty]=\Tr\, s_\infty=0.
\end{equation}

The proof of \cite[Lemma~2.44]{mochizuki2023asymptotic2} applies to our case, which implies that $s_\infty$ is a nonzero bounded section of $E$ on $\Sigma^\circ$. Let $p$ be an odd zero of $q$. Recall that in a neighborhood $U$ of $p$, we have
\[\varphi=\begin{pmatrix}0&z^d\\z^{m-d}&0 \end{pmatrix},\quad H_\infty=\begin{pmatrix} r^{\frac{1}{2}(m-2d)}&0\\0& r^{-\frac{1}{2}(m-2d)}\end{pmatrix}. \]
By \eqref{eq:s_infty_condition} we have $s_\infty=f\varphi$ or $f \Id_E$ for $f$ holomorphic on $U\backslash\{p\}$, and $|s_{\infty}|_{H_\infty}$ is constant on $\Sigma^\circ$. Then we must have $f=0$ and $s_\infty=0$ on $U\backslash\{p\}$. $|s_{\infty}|_{H_\infty}$ is constant, so $s_\infty=0$ on $\Sigma^\circ$. However, we have $\lVert s_\infty\rVert_{L^2,H_\infty}=1$, which is a contradiction.
\end{proof}
\begin{lemma}\label{lem:Vert_Gamma_Est}
  There exist $B_1,B_2>0$ such that \[|\gamma(\tau,t)|_{H_t}+|(\dbarE+\ad\,t\varphi)\gamma(\tau,t)|_{H_t,g_\Sigma}\leq B_1\exp(-B_2 t)\lVert \tau\rVert.\]
\end{lemma}
\begin{proof}
This is similar to \cite[Lemma~4.20]{mochizuki2023asymptotic2}, by using Lemma \ref{lem:V_prime_Est} (4), Proposition \ref{Prop:Bound_Inv_Laplacian}, and \cite[Proposition~2.36]{mochizuki2023asymptotic2}.
\end{proof}
Define ${\tt V}(\tau,t):={\tt V}'(\tau,t)-(\dbarE+\ad\,t\varphi)\gamma(\tau,t)$. Then ${\tt V}(\tau,t)$ is a harmonic representative of $\iota_t^{\mathrm{v}}([\tau])$.
\begin{proposition}\label{prop:vert_comp}
  There exist $B_1,B_2>0$ such that
  \begin{equation}\label{eq:Vert_Comp}
      |g_{L^2}(\iota_t^\mathrm{v}(\tau),\iota_t^\mathrm{v}(\tau))-\gsf(\iota_t^\mathrm{v}(\tau),\iota_t^\mathrm{v}(\tau))|\leq B_1\exp(-B_2 t)\lVert \tau\rVert^2.
  \end{equation}
\end{proposition}
\begin{proof}
  The inequality is equivalent to
  \[\Big|\int_\Sigma \big(H_t({\tt V}(\tau,t)^{1,0},{\tt V}(\tau,t)^{1,0})-H_t({\tt V}(\tau,t)^{0,1},{\tt V}(\tau,t)^{0,1})\big)+\int_{\nS}\tau\bar{\tau} \Big|\leq B_1\exp(-B_2 t)\lVert\tau\rVert^2.\]
   Note that $(\dbarE+\ad\, t\varphi)^\ast_{H_t,g_\Sigma}{\tt V}(\tau,t)=0$ and $F_{\tau'}$ only has $(0,1)$-part, we have
  \begin{align*}
    \int_\Sigma \big(H_t({\tt V}(\tau,t)^{1,0},{\tt V}(\tau,t)^{1,0})-H_t({\tt V}(\tau,t)^{0,1},{\tt V}(\tau,t)^{0,1})\big)=-\int_\Sigma H_t(F_{\tau'},{\tt V}(\tau,t)^{0,1})&\\=-\int_{\Sigma\backslash\bigcup\Sigma_p(1/2)} \Tr(F_{\tau'}\cdot ({\tt V}(\tau,t)^{0,1})_{H_t}^\dagger)&,
  \end{align*}
  since $F_{\tau'}$ is supported on $\Sigma\backslash\bigcup\Sigma_p(1/2)$. Since $\partial \tau=0$, we have \[\int_{\nS}\tau\bar{\tau}=\int_{\nS}\tau'\bar{\tau}=\int_{\Sigma\backslash\bigcup\Sigma_p(1/2)}\Tr(F_{\tau'}\cdot (F_\tau)^\dagger_{H_\infty}).\]
  \eqref{eq:Vert_Comp} follows from Lemma \ref{lem:V_prime_Est}, Lemma \ref{lem:Vert_Gamma_Est}, and the fact that $H_t-H_\infty=O(e^{-\delta t})$ on $\Sigma\backslash\bigcup\Sigma_p(1/2)$.
\end{proof}
\subsubsection{Horizontal deformations}\label{subsubsec:hor_deform}
Let $\nu\in H^0(\nS,K_{\nS})^-$ be a holomorphic 1-form on $\nS$. The multiplication of $\nu$ on $\calL$ induces a smooth section of $\Omega^{1,0}_{\Sigma^\circ}\otimes \End^0(\mathcal{E})$, which is a harmonic 1-form of $(\End^0(\mathcal{E}),\ad\,\varphi,H_\infty)|_{\Sigma^\circ}$. Similar to Lemma \ref{lem:Ftau_Est}, we have the following.
\begin{lemma}\label{lem:Fnu_Est}
    There exist positive constants $B_i$ ($i=0,1,2$) such that
    \[ |F_\nu|_{H_t,g_\Sigma}\leq B_0\lVert\nu\rVert,\quad |(\dbarE+\ad\,t\varphi)^\ast_{H_t,g_{\Sigma}}F_\nu|_{H_t}\leq B_1\exp(-B_2t)\lVert\nu\rVert,\]
    on $\Sigma\backslash N(q)$ for any $t\geq 1$.
\end{lemma}

Let $\iota_t^{\mathrm{h}}:H^0(\nS,K_{\nS})^-\to H^1(\Sigma,\mathsf{Def}(\mathcal{E},t\varphi))$ be the injective map defined in Corollary~\ref{prop:splitting}, where we identify $\nSt$ with $\nS$ as before. We will study the harmonic representative ${\tt H}(\nu,t)$ of $\iota_t^{\mathrm{h}}(\nu)$.

\newcommand{\m}{\mathrm{m}}
Let $p\in Z(q)$. There exists a holomorphic function $\alpha_p$ on $\nS|_{\Sigma_p}$ such that $\partial \alpha_p=\nu|_{\nS|_{\Sigma_p}}$ and $\alpha_p(\tilde{p})=0$ for all $\tilde{p}\in\tilde{\pi}^{-1}(p)$. The multiplication of $\alpha_p$ on $\calL$ induces a holomorphic endomorphism $F_{\alpha_p}$ of $\mathcal{E}_p$ such that $[\varphi,F_{\alpha_p}]=0$.

Let $H_t^{\app}$ be the metric constructed in Section \ref{subsec:app_sol}. Then $H_t^{\app}$ is compatible with $C_p$, $H_t^{\app}=H_\infty$ in $\Sigma\backslash\bigcup \Sigma_p(1)$, and $H_t^{\app}=H_t^{\m}$ in $\Sigma_p(1/2)$ where $H_t^{\m}$ is the model fiducial solution defined in Section \ref{subsec:ModelSolution}.  We define the family of tangents
\[{\tt H}_p(\nu,t)=(\partial_{E,H_t^{\app}}+\ad\,t\varphi^\dagger_{H_t^{\app}})F_{\alpha_p} \]
contained in $A^1(\mathsf{Def}(\mathcal{E}_p,t\varphi_p,C_p))$, which is a harmonic 1-form of $(\End^0(\mathcal{E}_p),t\ad\,\varphi_p,H_t^{\m})$ on $\Sigma_p(1/2)$. The following lemma is similar to Lemma \ref{lem:Vert_Local_Est1}.
\begin{lemma}\label{lem:Hor_Local_Est1}
 For any $1/4<r_2<r_1<2$, there exist $B_i>0$ ($i=0,1,2,3$) such that
 \[|F_{\alpha_p}|\leq B_0\lVert \nu\rVert,\quad |{\tt H}_p(\nu,t)-F_\nu|_{H_t}\leq B_1e^{-B_2 t}\lVert \nu\rVert,\quad |{\tt H}_p(\nu,t)|_{H_t}\leq B_3\lVert \nu\rVert,\]
  on $\Sigma_p(r_1)\backslash \Sigma_p(r_2)$ for any $t\geq 1$.
\end{lemma}
By Lemma \ref{lemm:local_exactness} and Lemma \ref{lem:Hor_Local_Est1} we have the following.
\begin{lemma}[{\cite[Lemma~4.13]{mochizuki2023asymptotic2}}]
  There exists a unique $\sigma_p(\nu,t)\in A^0(\Sigma_p^\circ,\Asy_{C_p}^0(\mathcal{E}_p))$ such that
  \[{\tt H}_p(\nu,t)-F_\nu=(\dbarE+\ad\, t\varphi)\sigma_p(\nu,t),\quad |\sigma_p(\nu,t)|_{H_t^{\app}}\leq B_1e^{-B_2 t}\lVert \nu\rVert \text{ on }\Sigma_p(r_1)\backslash \Sigma_p(r_2) ,\]
  for some constants $B_1,B_2>0$ and any $1/4<r_2<r_1<2$.
\end{lemma}
Similar to Lemma \ref{lem:Vert_Local_Est2}, we have the following.
\begin{lemma}\label{lem:Hor_Local_Est2}
For any $0<r<2$, there exist $B,B'>0$ such that $\lVert {\tt H}_p(\nu,t) \rVert_{L^2(\Sigma_p(r)),H_t}\leq B\lVert\nu\rVert$ and $|{\tt H}_p(\nu,t)|_{H_t,g_\Sigma}\leq B't^{2/(m_p+1)}\lVert\nu\rVert$ on $\Sigma_p(r)$ for any $t\geq 1$.
\end{lemma}

\begin{lemma}\label{lem:Hor_Local_Est3}
There exist $B_1,B_2>0$ such that for any $t\geq 1$ we have
\[ |(\dbarE+t\ad\,\varphi)^{\ast}_{H_t,g_\Sigma} {\tt H}_p(\nu,t)|_{H_t}\leq B_1\exp(-B_2 t)\lVert\nu\rVert. \]
\end{lemma}
\begin{proof}
Note that $(\dbarE+t\ad\,\varphi)^{\ast}_{H_t^{\m},g_\Sigma}{\tt H}_p(\nu,t)=0$ on $\Sigma_p(1/2)$. The result follows from the previous lemma and Theorem \ref{thm:ExpConvergence}.
\end{proof}

There is a cocycle ${\tt H}'(\nu,t)\in A^1(\mathsf{Def}(\mathcal{E},t\varphi))$ such that the following conditions hold.
\begin{enumerate}
  \item On $\Sigma_p(1/2)$, we have $ {\tt H}'(\nu,t)={\tt H}_p(\nu,t)$.
  \item On $\Sigma_p(1)\backslash \overline{\Sigma}_p(1/2)$, we have ${\tt H}'(\nu,t)=F_\nu+(\dbarE+\ad\,t\varphi)(\chi_p\cdot \sigma_p(\nu,t))$.
  \item On $\Sigma\backslash \bigcup_{p\in Z(q)}\Sigma_p(1)$, we have ${\tt H}'(\nu,t)=F_\nu$.
\end{enumerate}
By the construction and Lemmas \ref{lem:Fnu_Est}-\ref{lem:Hor_Local_Est3}, we obtain the following.
\begin{lemma}\label{lem:H_prime_Est}
  \begin{enumerate}
 \item $(\dbarE+\ad\, t\varphi){\tt H}'(\nu,t)=0$, and $\rho_t({\tt H}'(\nu,t))=\pi_{\ast}(\nu)$.
 \item There exist $B_1,B_2>0$ such that $|{\tt H}'(\nu,t)-F_{\nu}|_{H_t,g_\Sigma}\leq B_1\exp(-B_2t) \lVert \nu\rVert$ on $\Sigma\setminus\bigcup \Sigma_p(1/2)$ and $|{\tt H}'(\nu,t)-{\tt H}_p(\nu,t)|_{H_t,g_\Sigma}\leq B_1\exp(-B_2t)\lVert \nu\rVert$ on $\Sigma_p(1)$.
\item There exists $B>0$ such that $|{\tt H}'(\nu,t)|_{H_t,g_\Sigma}\leq Bt^{2/(m_p+1)}\|\nu\|$ and $\|{\tt H}'(\nu,t)\|_{L^2,H_t}\leq B\|\nu\|$ for any $t\geq 1$.
 \item There exist $B_3,B_4>0$ such that $\big|(\dbarE+\ad\, t\varphi)_{H_t,g_\Sigma}^{\ast}{\tt H}'(\nu,t)\big|_{H_t}\leq B_3\exp(-B_4t)\lVert\nu\rVert$ on $\Sigma$ for any $t\geq 1$.
  \end{enumerate}
\end{lemma}

There exists a unique $\gamma(\nu,t)\in A^0(\End^0(\mathcal{E}))$
such that
\[
 \Delta_t \gamma(\nu,t)
 =(\dbarE+\ad\, t\varphi)^{\ast}_{H_t,g_\Sigma}{\tt H}'(\nu,t).
\]
Similar to Lemma \ref{lem:Vert_Gamma_Est}, we have the following.
\begin{lemma}\label{lem:Hor_Gamma_Est}
There exist $B_1,B_2>0$ such that
the following holds on $\Sigma$ for any $t\geq 1$:
\begin{equation}
 |\gamma(\nu,t)|_{H_t}
+|(\dbarE+\ad\, t\varphi)\gamma(\nu,t)|_{H_t,g_\Sigma}
 \leq B_1\exp(-B_2t)\|\nu\|.
\end{equation}
\end{lemma}
Let ${\tt H}(\nu,t)
 ={\tt H}'(\nu,t)-(\dbarE+\ad\, t\varphi) \gamma(\nu,t)$.
Then, ${\tt H}(\nu,t)$ is a harmonic $1$-form of $(\End^0(\mathcal{E}),t\ad\,\varphi,H_t)$,
satisfying
$[{\tt H}(\nu,t)]=[{\tt H}'(\nu,t)]=\iota_t^{\mathrm{h}}(\nu)$.
\begin{proposition}\label{prop:hor_comp}
There exist $B_1, B_2>0$ such that
\begin{align}
 \big|
 g_{L^2}(\iota^{\mathrm{h}}_t(\nu),\iota^{\mathrm{h}}_t(\nu))
 -g_{\mathrm{sf}}(\iota^{\mathrm{h}}_t(\nu),\iota^{\mathrm{h}}_t(\nu))
 \big|&\leq B_1\exp(-B_2t)\|\nu\|^2. \label{eq:Hor_Comp}\\
  \big|g_{L^2}(\iota^{\mathrm{v}}_t(\tau),\iota^{\mathrm{h}}_t(\nu))\big|&\leq B_1\exp(-B_2t)\|\nu\|\|\tau\|. \label{eq:Mix_Comp}
\end{align}
\end{proposition}
\begin{proof}
We first prove \eqref{eq:Hor_Comp}. By Lemmas \ref{lem:H_prime_Est}, \ref{lem:Hor_Gamma_Est}, and Theorem \ref{thm:ExpConvergence}, it suffices to show
\begin{multline*}
    \Big|\int_{\Sigma_p(1/2)} H_t^{\m}({\tt H}_p(\nu,t)^{1,0},{\tt H}_p(\nu,t)^{1,0})-H_t^{\m}({\tt H}_p(\nu,t)^{0,1},{\tt H}_p(\nu,t)^{0,1})-\int_{\nS|\Sigma_p(1/2)}\nu \bar{\nu}\Big|\\
    \leq B_1\exp(-B_2 t)\lVert\nu\rVert^2. \end{multline*}
Using the integration by parts formula from \cite[Proposition~2.27]{mochizuki2023asymptotic2}, the left-hand side equals
\[\Big|\int_{\partial \Sigma_p(1/2)}H_t^{\m}(F_{\alpha_p},{\tt H}_p(\nu,t)^{1,0})-H_t^{\m}(F_{\alpha_p},{\tt H}_p(\nu,t)^{0,1})-\int_{\nS|\Sigma_p(1/2)}\nu \bar{\nu} \Big|. \]
 By Lemma \ref{lem:Hor_Local_Est1}, we have
 \[\Big|\int_{\partial \Sigma_p(1/2)} H_t^{\m}(F_{\alpha_p},{\tt H}_p(\nu,t)^{1,0})\Big|\leq B_1\exp(-B_2 t)\lVert\nu\rVert^2. \]
 On the other hand,
 \[\int_{\nS|_{\Sigma_p(1/2)}}\nu\bar{\nu}=\int_{\partial\nS|_{\Sigma_p(1/2)}}\alpha_p\bar{\nu}=\int_{\partial\Sigma_p(1/2)}\Tr(F_{\alpha_p}\cdot F_{\bar{\nu}}).  \]
 By Lemma \ref{lem:Hor_Local_Est1} and the exponential decay of $H_t^{\m}-H_\infty$ near $\partial\Sigma_p(1/2)$, we obtain
 \[ |({\tt H}_p(\nu,t)^{1,0})^\dagger_{H_t^{\m}}-F_{\bar{\nu}}|_{H_{t}^{\m}}\leq B_1\exp(-B_2 t)\lVert\nu\rVert. \]
 Then \eqref{eq:Hor_Comp} follows. Now we prove \eqref{eq:Mix_Comp}. Since $(\dbarE+\ad\, t\varphi)^\ast_{H_t,g_\Sigma}{\tt V}(\tau,t)=0$ and $F_{\tau'}$ only has $(0,1)$-part, we have
\[
    \int_\Sigma H_t({\tt V}(\tau,t),{\tt H}(\nu,t))=-\int_\Sigma H_t(F_{\tau'},{\tt H}(\nu,t)^{0,1})=-\int_{\Sigma\backslash\bigcup\Sigma_p(1/2)} \Tr(F_{\tau'}\cdot ({\tt H}(\nu,t)^{0,1})_{H_t}^\dagger),
\]
  since $F_{\tau'}$ is supported on $\Sigma\backslash\bigcup\Sigma_p(1/2)$. Then
  \eqref{eq:Mix_Comp} follows from Lemmas \ref{lem:Ftau_Est}, \ref{lem:Vert_Local_Est1}, \ref{lem:H_prime_Est}, \ref{lem:Hor_Gamma_Est}, and the fact that $H_t-H_\infty=O(e^{-\delta t})$ on $\Sigma\backslash\bigcup\Sigma_p(1/2)$.
\end{proof}
Combining Propositions \ref{prop:vert_comp} and \ref{prop:hor_comp}, Theorem \ref{thm:intro_main1} follows.

\subsection{Other strata}
Here we consider the asymptotics of $g_{L^2}$ for general strata of the Hitchin base $\calB_{\bp}$. Let $\MM_{\bp}\subset \MM$ consist of $(\calE,\varphi)$ such that the zeros of $q=\det\varphi$ correspond to the partition $\bp$. Assume that $\bp$ contains at least one odd number. By Lemma \ref{lemm:exact_sequence_M_Vmax}, for any $(\ME,\vp)\in \MM_{\bp,V_{\max}}$, there is a short exact sequence
\begin{equation}\label{eq:SES_Mp}
  0\to H^1(\nS,\mathcal{O}_{\nS})^-\to T_{(\mathcal{E},\varphi)}\MM_{\bp,V_{\max}}\xrightarrow{d\calH}H^0(\Sigma,K^2(-D_0(q)))\to 0.
  \end{equation}
By Lemma \ref{lemm:tangent_space_base_stratum} (assuming $D_0(q)$ is generic when $r_{\mathrm{even}}+r_{\mathrm{odd}}\leq 2g-2$),
\begin{align*}
  \dim H^0(\Sigma,K^2(-D_0(q)))&=r_{\mathrm{even}}+r_{\mathrm{odd}}-(g-1) \\
  &\leq g-1+\frac{r_{\mathrm{odd}}}{2}=\dim H^1(\nS,\mathcal{O}_{\nS})^-=\dim\Prym(\nS/\Sigma).
\end{align*}
The equality holds if and only if $q$ has at most order 2 zeros.

 Recall that by \eqref{eq:tau_map}, we have an injection \[\tau_{\MB_{\bp}}: H^0(\Sigma,K^2(-D_0(q))) \to H^0(\nS,K_{\nS})^-,\quad \dot{q} \mapsto \frac{\tilde{\pi}^*\dot{q}}{2\tilde{\omega}}.\] Denote the image of $\tau_{\MB_{\bp}}$ by $H^0(\nS,K_{\nS})^-_{\MB_{\bp}}$, which is identified with $H^0(\Sigma,K^2(-D_0(q)))$. Similar to Corollary \ref{prop:splitting}, the short exact sequence \eqref{eq:SES_Mp} has a splitting $\iota^{\mathrm{h}}:H^0(\nS,K_{\nS})^-_{\MB_{\bp}}\to T_{(\mathcal{E},\varphi)}\MM_{\bp,V_{\max}}$. This induces the following isomorphism:
\[T_{(\mathcal{E},\varphi)}\MM_{\bp,V_{\max}}\cong H^1(\nS,\mathcal{O}_{\nS})^-\oplus H^0(\nS,K_{\nS})^-_{\MB_{\bp}}.  \]
Using this isomorphism, we define the generalized semi-flat metric $g_{\mathrm{sf},\bp}$ on $\MM_{\bp,V_{\max}}$ by
\[g_{\mathrm{sf},\bp}|_{(\mathcal{E},\varphi)}([\tau],[\tau])=2\sqrt{-1}\int_{\nS}\big(\tau^{1,0}\wedge \overline{\tau^{1,0}}-\tau^{0,1}\wedge\overline{\tau^{0,1}}\big).\]
where $\tau$ is a harmonic 1-form on $\nS$ with $[\tau^{0,1}]\in H^1(\nS,\mathcal{O}_{\nS})^-$ and $[\tau^{1,0}]\in H^0(\nS,K_{\nS})^-_{\MB_{\bp}}$. By the same arguments as in the previous subsection, we obtain the following convergence result.
\begin{theorem}\label{thm:asymptotic_closed_strata}
  Let $(\ME,t\vp)$ be a ray in $\MM_{\bp,V_{\max}}$ (assumed to be generic when $r_{\mathrm{even}}+r_{\mathrm{odd}}\leq 2g-2$). Then there exists  $\varepsilon > 0$ such that
  \[
  (g_{L^2}|_{\MM_{\bp,V_{\max}}} - g_{\sfm,\bp})|_{(\ME, t\vp)} = \MO(e^{-\varepsilon t}),
  \]
  as $t \to \infty$.
\end{theorem}

\subsection{The case of $d=2g-2$}\label{ssec:d=2g-2}
The case of the subintegrable system $\calM^{2g-2}$ was treated in \cite{Francoetal_unramified}. In this case, a quadratic differential $q \in \calB^{2g-2}$ defines a square root of the trivial line bundle by $M=K(-\frac12 \mathsf{div}(q))$. The normalized spectral curve is the associated unramified double cover $\tilde{\pi}: \wt{S}_q \to \Sigma$. The 2-torsion points $\Prym(\wt{S}_q)[2]$ form a discrete set, hence the cover $\tilde{\pi}$ is constant on connected components of $\calB^{2g-2}$. For each connected component, the pushforward defines a hyperkähler map $T^\vee \Prym(\wt{S}_q) \to \calM^{2g-2}$ by \cite[Proposition 3.2]{Francoetal_unramified}. The connected components of $\calM^{2g-2}$ are identified with the quotient of $T^\vee \mathsf{Prym}(\wt{S}_q)$ by the Galois group. In particular, the solutions to the Hitchin equations on $\calM^{2g-2}$ decouple, and the restricted Hitchin metric is invariant under rescaling the Higgs field. Finally, the hyperkähler metric on $T^\vee \Prym(\wt{S}_q)$ coincides with the canonical flat hyperkähler metric  $g_{\sfm,2g-2}$ on the cotangent bundle induced by the flat Kähler metric on $H^0(\wt{S}_q,K_{\wt{S}_q})$, as described in \cite[Theorem 2.1]{freed1999special}. In conclusion, the restricted Hitchin metric $g_{L^2}|_{\MM^{2g-2}}$ coincides with the flat hyperkähler metric $g_{\sfm,2g-2}$.

\bibliographystyle{plain}
\bibliography{references}
\end{document}